\documentclass{CMS_IMEXRK2}
\usepackage[paperwidth=7in, paperheight=10in, margin=.875in]{geometry}
\usepackage[backref,colorlinks,linkcolor=red,anchorcolor=green,citecolor=blue]{hyperref}
\usepackage{amsfonts,amssymb}
\usepackage{amsmath}
\usepackage{graphicx}
\usepackage{cite}
\usepackage{enumerate}

\usepackage{float}
\usepackage[caption=false, font=normalsize, labelfont=sf]{subfig}
\usepackage{multirow}
\usepackage{multicol}

\renewcommand{\theequation}{\arabic{section}.\arabic{equation}}

\allowdisplaybreaks
\begin{document}
	\title{IMEX-RK methods for Landau-Lifshitz equation with arbitrary damping\thanks{Received date, and accepted date (The correct dates will be entered by the editor).}}
	
	
	\author{Yan Gui\thanks{School of Mathematical Sciences, Soochow University, Suzhou,China, (guiyan\_maths@163.com).}
		\and Cheng Wang\thanks{Mathematics Department, University of Massachusetts, North Dartmouth, MA 02747, USA, (cwang1@umassd.edu).}
		\and Jingrun Chen\thanks{School of Mathematical Sciences, University of Science and Technology of China, Hefei, China and Suzhou Institute for Advanced Research, University of Science and Technology of China, Suzhou, China, (jingrunchen@ustc.edu.cn).}}

	\pagestyle{myheadings} \markboth{IMEX-RK METHODS FOR LL EQUATION WITH ARBITRARY DAMPING}{YAN GUI, CHENG WANG, AND JINGRUN CHEN} \maketitle
	
	\begin{abstract}
		Magnetization dynamics in ferromagnetic materials is modeled by the Landau-Lifshitz (LL) equation, a nonlinear system of partial differential equations. Among the numerical approaches, semi-implicit schemes are widely used in the micromagnetics simulation, due to a nice compromise between accuracy and efficiency. At each time step, only a linear system needs to be solved and a projection is then applied to preserve the length of magnetization. However, this linear system contains variable coefficients and a non-symmetric structure, and thus an efficient linear solver is highly desired. If the damping parameter becomes large, it has been realized that efficient solvers are only available to a linear system with constant, symmetric, and positive definite (SPD) structure. In this work, based on the implicit-explicit Runge-Kutta (IMEX-RK) time discretization, we introduce an artificial damping term, which is treated implicitly.  The remaining terms are treated explicitly. This strategy leads to a semi-implicit scheme with the following properties: (1) only a few linear system with constant and SPD structure needs to be solved at each time step; (2) it works for the LL equation with arbitrary damping parameter; (3) high-order accuracy can be obtained with high-order IMEX-RK time discretization. Numerically, second-order and third-order IMEX-RK methods are designed in both the 1-D and 3-D domains. A comparison with the backward differentiation formula scheme is undertaken, in terms of accuracy and efficiency. The robustness of both numerical methods is tested on the first benchmark problem from National Institute of Standards and Technology. The linearized stability estimate and optimal rate convergence analysis are provided for an alternate IMEX-RK2 numerical scheme as well.
	\end{abstract}
	
	\section{Introduction}\label{intro}
	
	Ferromagnetic materials have the intrinsic magnetic order (magnetization), whose dynamics is modeled by the Landau-Lifshitz (LL) equation \cite{Landau1935On, gilbert1955lagrangian}. The equation stands for a non-local nonlinear problem with non-convex constraint and possible degeneracy.
	The past few decades have witnessed the progress of numerical methods for the LL equation; see \cite{An2022, kruzik2006recent, CJreview2007, cimrak2007survey} for reviews and references therein.
	There are explicit algorithms (e.g. \cite{alouges2006convergence,bartels2008numerical}), fully implicit ones (e.g. \cite{bartels2006convergence, fuwa2012finite}), and semi-implicit schemes (e.g. \cite{weinan2001numerical, wang2001gauss, cimrak2005error, gao2014optimal, boscarino2016high, an2016optimal, An2021}). An explicit method updates the magnetization without any need to solver linear or nonlinear system of equations, while it suffers from the severe stability constraint on the temporal step-size. Implicit schemes are unconditionally stable and preserve the physical properties of the LL equation, such as energy dissipation and length conservation, while a nonlinear system of equation needs to be solved at each time step. Semi-implicit schemes only solve a linear system of equations at each time step and are unconditionally stable in micromagnetics simulations. Therefore, a semi-implicit method provides a nice balance between accuracy and efficiency.
	
	One typical semi-implicit method is the Gauss-Seidel projection method (GSPM) \cite{weinan2001numerical, wang2001gauss, garcia2003improved, Li2020TwoIG}. Using the vectorial structure of LL equation, GSPM achieves unconditional stability in micromagnetics simulations; and only a few linear systems need to be solved, with constant, symmetric, and positive definite coefficients. However, the temporal accuracy of GSPM is limited to the first-order. Another semi-implicit approach is based on the backward differentiation formula (BDF) for temporal derivative and the one-sided extrapolation for nonlinear terms \cite{xie2020second,chen2021convergence,Lubich2021}. High-order accuracy can be obtained in time using the BDF approach. However, only first-order and second-order BDFs are unconditionally stable. Meanwhile, the linear system of equations has variable coefficients and non-symmetric structure, thus no fast solver is available. Meanwhile, for time-dependent nonlinear partial differential equations in general, implicit-explicit (IMEX) schemes have been extensively applied \cite{boscarino2016high}. The basic idea is to treat dominant linear term implicitly and the remaining nonlinear terms explicitly. For the LL equation, the second-order IMEX has been studied in \cite{xie2020second}. Two linear systems, with variable coefficients and non-symmetric structure, need to be solved. Thus IMEX2 can hardly compete with BDF2 in terms of accuracy and efficiency.
	
	In this work, we propose an implicit-explicit Runge-Kutta (IMEX-RK) scheme for solving the LL equation, based on the recent development of IMEX-RK method for the nonlinear diffusion equation \cite{wang2020local}. The basic idea is to introduce an artificial linear diffusion term and treat it implicitly. All the remaining terms are treated explicitly. RK methods are employed for the time discretization. Only a few linear systems, with constant coefficients and SPD structure, need to be solved. In the existing literature, such linear numerical schemes have only be reported for the large damping parameter~ \cite{cai2022second}. Instead, the IMEX-RK method works for the LL equation general damping parameters, which is very important since the damping parameter may be small in most magnetic materials \cite{Brown1963micromagnetics}. Moreover, higher-order numerical schemes could be constructed using RK approaches. Numerical results have demonstrated an advantage of IMEX-RK schemes over the BDF2 approach in terms of accuracy and efficiency. The performance of IMEX-RK schemes has also been verified over a large range of artificial damping parameters and the first benchmark problem for a ferromagnetic thin film material from National Institute of Standards and Technology (NIST).
	
	Because of the robust numerical results, a theoretical analysis of the proposed IMEX-RK numerical schemes is highly desired. However, such an analysis turns out to be very challenging, due to the multi-stage nature, as well as the highly complicated nonlinear terms in the vector form. Some convergence analysis works have been reported for the IMEX-RK numerical methods to various nonlinear PDEs in the existing literature, such as the convection-diffusion equation~\cite{WangH2016}, the porous media equation~\cite{Nan2021, WangH2019}, etc. Meanwhile, the degree of nonlinearity of the LL equation is even higher than these reported PDE models, and the associated theoretical analysis is expected to be more challenging. In this article, we choose an alternate IMEX-RK2 numerical algorithm, in which the explicit part satisfies the strong stability-preserving (SSP) property~\cite{Conde2017, gottlieb01a} (denoted as the SSP-IMEX-RK2 scheme), and provide a linearized stability estimate and convergence analysis for the SSP-IMEX-RK2 method to a simplified version of the LL equation. Many state-of-arts techniques, such as a rough error estimate to control the discrete $\ell^\infty$ and $W_h^{1,8}$ norms of the numerical solution, combined with a refined error estimate, have to be applied to obtain an optimal rate convergence analysis for the SSP-IMEX-RK2 scheme to the nonlinear LL equation.

    The rest of paper is organized as follows. In Section~\ref{S:method}, we introduce the LL model and then propose the IMEX-RK schemes, including the SSP-IMEX-RK2 algorithm. For the convenience of comparison, we also briefly review the BDF2 method. Accuracy and efficiency tests of the IMEX-RK schemes are provided in Section~\ref{S:Numerical tests} with a detailed check for the dependence on the artificial damping parameter. Section~\ref{section:micromagnetics simulations} is devoted to the micromagnetics simulations of the IMEX-RK methods, including equilibrium structures and the first benchmark problem from NIST.  The linearized stability estimate and the convergence analysis for the SSP-IMEX-RK2 scheme to a simplified version of nonlinear LL equation are provided in Section~\ref{sec: convergence}. Finally, some concluding remarks are made in Section~\ref{section:conclusion}.

\section{The model and the proposed methods}\label{S:method}

\subsection{Landau-Lifshitz equation}\label{model}
The LL equation is a phenomenological model for magnetization dynamics in a ferromagnetic material, which takes the following form
\begin{align}
	\boldsymbol{m}_{t}=-\boldsymbol{m} \times \boldsymbol{h}_{\mathrm{eff}}-\alpha \boldsymbol{m} \times\left(\boldsymbol{m} \times \boldsymbol{h}_{\mathrm{eff}}\right),
	\label{eq-1}
\end{align}
with the homogeneous Neumann boundary condition
\begin{align}
	{{\left. \frac{\partial \boldsymbol m}{\partial \nu } \right|}_{\partial \Omega }}=0.
	\label{eq-2}
\end{align}
Here $\Omega$ is a bounded domain occupied by the ferromagnetic material, and the magnetization $\boldsymbol m:\Omega \subset {\mathbb{R}^{d}}\to {\mathbb{S}^{2}},d=1,2,3$ is a 3-D vector field with the length constraint $\left| \boldsymbol m \right|=1$, and $\nu$ is the unit outward normal vector along $\partial\Omega$. The first term of the right hand side in  \eqref{eq-1} is the gyromagnetic term, while the
second term represents the damping term with a dimensionless parameter $\alpha>0$.

For a uniaxial material, the free energy is given by
\begin{align*}
	F[\boldsymbol{m}]=\frac{\mu_{0} M_{s}^{2}}{2} \int_{\Omega}\left(\epsilon|\nabla \boldsymbol{m}|^{2}+Q\left(m_{2}^{2}+m_{3}^{2}\right)-\boldsymbol{h}_{s} \cdot \boldsymbol{m}-2 \boldsymbol{h}_{e} \cdot \boldsymbol{m}\right) \mathrm{d} \boldsymbol{x},
\end{align*}
in which the listed terms correspond to the exchange energy, the anisotropy energy, the magnetostatic energy, and the Zeeman energy, respectively. The effective field ${{\boldsymbol h}_{\text{eff}}}$ can be obtained by taking the variation of $F[\boldsymbol{m}]$ with respect to $\boldsymbol m$,
and consists of the exchange field, the anisotropy field, the stray field ${{\boldsymbol h}_{\text{s}}}$, and the external field ${{\boldsymbol h}_{\text{e}}}$ of the following form
\begin{align*}
	{{\boldsymbol h}_{\text{eff}}}=\epsilon \Delta \boldsymbol m-Q({{m}_{2}}{{\boldsymbol e}_{2}}+{{m}_{3}}{{\boldsymbol e}_{3}})+{{\boldsymbol h}_{s}}+{{\boldsymbol h}_{e}} .
\end{align*}

	\begin{table}[htbp]
	\centering
	\caption{Typical values of the physical parameters for Permalloy, which is an alloy of Nickel (80\%) and Iron (20\%) frequently used in magnetic storage devices.}\label{tab}
 	\begin{tabular}{|c|p{5cm}|}
		\hline
		\multicolumn{2}{|c|}{\bf Physical Parameters for Permalloy}\\
		\hline
		$K_u$ & $1.0\times10^2\;\mathrm{J/m^3}$ \\
		\hline
		$C_{ex}$& $1.3\times10^{-11}\;\mathrm{J/m}$ \\
		\hline	
		$M_s$ & $8\times 10^5\;\mathrm{A/m}$ \\
		\hline
		$\mu_0$ & $4\pi \times 10^{-7}\;\mathrm{N/A^2}$ \\
		\hline
		$\alpha$ & $0.01$ \\
		\hline
	\end{tabular}
\end{table}

In the above representation, $Q={{K}_{u}}/({{\mu }_{0}}M_{s}^{2})$ and $\epsilon ={{C}_{ex}}/({{\mu }_{0}}M_{s}^{2}{{L}^{2}})$ are the dimensionless parameters with ${{C}_{ex}}$ the exchange constant, ${{K}_{u}}$ the anisotropy constant, $L$ the diameter of ferromagnetic body, ${{\mu }_{0}}$ the permeability of vacuum, and $M_{s}$ the saturation magnetization, respectively.. Typical values of the physical parameters for Permalloy are included as shown in Table~\ref{tab}.

The two unit vectors ${{\boldsymbol e}_{2}}={{(0,1,0)}^{T}}$, ${{\boldsymbol e}_{3}}={{(0,0,1)}^{T}}$, and $\Delta$ stands for the standard Laplacian operator. The stray field ${{\boldsymbol h}_{\text{s}}}$ takes the form
\begin{align*}
	{{\boldsymbol h}_{s}}=-\nabla \int_{\Omega }{\nabla N(\boldsymbol  x-\boldsymbol y)\cdot \boldsymbol m(\boldsymbol y)d\boldsymbol y},
\end{align*}
where $N(\boldsymbol x)=-\frac{1}{4\pi \left| \boldsymbol x \right|}$ is the Newtonian potential.

For simplicity, we denote
\begin{align}\label{lowerterm}
	\emph{ \textbf{f}}=-Q({{m}_{2}}{{\boldsymbol e}_{2}}+{{m}_{3}}{{\boldsymbol e}_{3}})+{{\boldsymbol h}_{s}}+{{\boldsymbol h}_{e}}.
\end{align}
Accordingly, the LL equation can be rewritten as
\begin{align}
	{{\boldsymbol m}_{t}}=-\boldsymbol m\times (\epsilon \Delta \boldsymbol m+\emph{ \textbf{f}})-\alpha \boldsymbol m\times \boldsymbol m\times (\epsilon \Delta \boldsymbol m+\emph{ \textbf{f}}).
	\label{eq-3}
\end{align}
Thanks to the constraint $|\boldsymbol{m}|=1$, we obtain an equivalent form
\begin{align}
	\boldsymbol{m}_{t}=\alpha(\epsilon \Delta \boldsymbol{m}+\boldsymbol{f})+\alpha\left(\epsilon|\nabla \boldsymbol{m}|^{2}-\boldsymbol{m} \cdot \boldsymbol{f}\right) \boldsymbol{m}-\boldsymbol{m} \times(\epsilon \Delta \boldsymbol{m}+\boldsymbol{f}).
	\label{eq-4}
\end{align}

Some notations are introduced for discretization and numerical approximation. The 1-D domain is set as $\Omega=(0,1)$, and the 3-D version becomes $\Omega =( 0,1 )^3$, and the final time is denoted as $T$. In the 1-D domain, we divide $\Omega$ into $N$ equal parts with $h=1/N$. Fig.~\ref{fig1} displays a schematic picture of 1-D spatial grids, with ${{x}_{i-\frac{1}{2}}}=(i-\frac{1}{2})h,i=1,2,\cdots ,N$.
\begin{figure}[htbp]
	\centering
	\includegraphics[width=12cm,height=3cm]{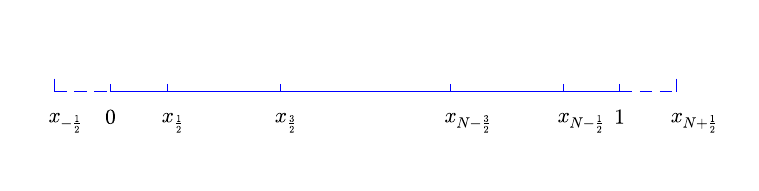}
	\caption{Spatial grids in 1D where ${{x}_{-\frac{1}{2}}}$ and ${{x}_{N+\frac{1}{2}}}$ are two ghost points.}\label{fig1}
\end{figure}
The 3-D grid points in can be similarly constructed. For convenience, we set $h_x=h_y=h_z = h$, and
${{\mathbf m}_{i,j,k}}=\mathbf m((i-\frac{1}{2})h,(j-\frac{1}{2})h,(k-\frac{1}{2})h), 0\le i,j,k\le N+1 $.

The second-order centered difference for $\Delta \mathbf m$ in the 3-D domain is formulated as
\begin{equation*}
	\begin{split}
		{{\Delta }_{h}}{{\mathbf m}_{i,j,k}} =& \frac{{{\mathbf m}_{i+1,j,k}}-2{{\mathbf m}_{i,j,k}}+{{\mathbf m}_{i-1,j,k}}}{{h^{2}}}\\
		&+\frac{{{\mathbf m}_{i,j+1,k}}-2{{\mathbf m}_{i,j,k}}+{{\mathbf m}_{i,j-1,k}}}{{h^{2}}}\\
		&+\frac{{{\mathbf m}_{i,j,k+1}}-2{{\mathbf m}_{i,j,k}}+{{\mathbf m}_{i,j,k-1}}}{{h^{2}}}.
	\end{split}
\end{equation*}
The second-order approximation of Neumann boundary condition in \eqref{eq-2} gives
\begin{eqnarray}
	\begin{aligned}
		& {{\mathbf m}_{0,j,k}}={{\mathbf m}_{1,j,k}}, \, \, \, {{\mathbf m}_{N,j,k}}={{\mathbf m}_{N+1,j,k}}, \, \, \, j,k=1,\cdots,N ,  \\
		& {{\mathbf m}_{i,0,k}}={{\mathbf m}_{i,1,k}}, \, \, \, {{\mathbf m}_{i,N,k}}={{\mathbf m}_{i,N+1,k}}, \, \, \, i,k=1,\cdots,N, \\
		& {{\mathbf m}_{i,j,0}}={{\mathbf m}_{i,j,1}}, \, \, \, {{\mathbf m}_{i,j,N}}={{\mathbf m}_{i,j,N+1}}, \, \, \, i,j=1,\cdots,N.
	\end{aligned}
	\label{BC-1}
\end{eqnarray}
The discrete gradient operator $\nabla_{h} \boldsymbol{m}$ with $\boldsymbol{m}=(u, v, w)^{T}$ becomes
\begin{equation*}
	\nabla_{h} \boldsymbol{m}_{i, j, k}=\left[\begin{array}{lll}
		\frac{u_{i+1, j, k}-u_{i-1, j, k}}{2h} & \frac{v_{i+1, j, k}-v_{i-1, j, k}}{2h} & \frac{w_{i+1, j, k}-w_{i-1, j, k}}{2h} \\
		\frac{u_{i, j+1, k}-u_{i, j-1, k}}{2h} & \frac{v_{i, j+1, k}-v_{i, j-1, k}}{2h} & \frac{w_{i, j+1, k}-w_{i, j-1, k}}{2h} \\
		\frac{u_{i, j, k+1}-u_{i, j, k-1}}{2h} & \frac{v_{i, j, k+1}-v_{i, j, k-1}}{2h} & \frac{w_{i, j, k+1}-w_{i, j, k-1}}{2h}
	\end{array}\right].
\end{equation*}
For temporal discretization, we set ${{t}^{n}}=nk$ with $k$ the step-size and $n\le \left[ \frac{T}{k} \right]$.

\subsection{Second-order backward differentiation formula method}\label{bdf}
The BDF numerical method \cite{xie2020second,chen2021convergence} is based on the BDF temporal discretization and the one-sided extrapolation. For comparison, we recall the BDF2 algorithm as
\begin{align}
	\left\{\begin{array}{c}
		\dfrac{\frac{3}{2} \tilde{\boldsymbol{m}}_{h}^{n+2}-2 \boldsymbol{m}_{h}^{n+1}+\frac{1}{2} \boldsymbol{m}_{h}^{n}}{k}=-\hat{\boldsymbol{m}}_{h}^{n+2} \times\left(\epsilon \Delta_{h} \tilde{\boldsymbol{m}}_{h}^{n+2}+\hat{\boldsymbol{f}}_{h}^{n+2}\right) \\
		-\alpha \hat{\boldsymbol{m}}_{h}^{n+2} \times\left(\hat{\boldsymbol{m}}_{h}^{n+2} \times\left(\epsilon \Delta_{h} \tilde{\boldsymbol{m}}_{h}^{n+2}+\hat{\boldsymbol{f}}_{h}^{n+2}\right)\right) \\
		\boldsymbol{m}_{h}^{n+2}=\frac{\tilde{\boldsymbol{m}}_{h}^{n+2}}{\left|\tilde{\boldsymbol{m}}_{h}^{n+2}\right|},
	\end{array}\right.
	\label{eq-5}
\end{align}
where $\tilde{\boldsymbol{m}}_{h}^{n+2}$ is an intermediate magnetization, and $\hat{\boldsymbol{m}}_{h}^{n+2}, \hat{\boldsymbol{f}}_{h}^{n+2}$ are given by the following extrapolation formula:
$$
\begin{aligned}
	\hat{\boldsymbol{m}}_{h}^{n+2} &=2 \boldsymbol{m}_{h}^{n+1}-\boldsymbol{m}_{h}^{n}, \quad
	\hat{\boldsymbol{f}}_{h}^{n+2} &=2 \boldsymbol{f}_{h}^{n+1}-\boldsymbol{f}_{h}^{n},
\end{aligned}
$$
with $\boldsymbol{f}_{h}^{n}=-Q\left(m_{2}^{n} \boldsymbol e_{2}+m_{3}^{n} \boldsymbol e_{3}\right)+\boldsymbol{h}_{s}^{n}+\boldsymbol{h}_{e}^{n}$, corresponding to \eqref{lowerterm}. The presence of cross product in \eqref{eq-5} yields a linear system of equations with variable coefficients and non-symmetric structure. Often, GMRES is employed as the numerical solver. The BDF2 algorithm \eqref{eq-5} treats both the gyromagentic and damping terms semi-implicitly, i.e., $\Delta m$ is treated implicitly while the prefactors are treated explicitly.

Since $|\boldsymbol{m}|=1$, the strength of gyromagnetic term is controlled by $\epsilon\Delta \boldsymbol{m} + \boldsymbol{f}$. Meanwhile, the strength of damping term is controlled by $\alpha(\epsilon\Delta \boldsymbol{m} + \boldsymbol{f})$. Since $\alpha < 1$ for many magnetic materials, it is reasonable to treat both the gyromagentic and damping terms semi-implicitly. However, for large $\alpha$, it is possible to treat $\alpha\epsilon\Delta \boldsymbol{m}$ part in the damping term implicitly, and the gyromagnetic term and all remaining terms explicitly, as demonstrated in \cite{cai2022second}. The stability and convergence of the scheme is proved under the condition $\alpha>3$ \cite{cai2023}. Starting with \eqref{eq-4}, the BDF2 algorithm in \cite{cai2022second} becomes
\begin{align}
	\left\{\begin{array}{l}
		\dfrac{\frac{3}{2} \tilde{\boldsymbol{m}}_{h}^{n+2}-2 \boldsymbol{m}_{h}^{n+1}+\frac{1}{2} \boldsymbol{m}_{h}^{n}}{k}=-\hat{\boldsymbol{m}}_{h}^{n+2} \times\left(\epsilon \Delta_{h} \hat{\boldsymbol{m}}_{h}^{n+2}+\hat{\boldsymbol{f}}_{h}^{n+2}\right) \\
		\quad+\alpha\left(\epsilon \Delta_{h} \tilde{\boldsymbol{m}}_{h}^{n+2}+\hat{\boldsymbol{f}}_{h}^{n+2}\right) \\
		\quad+\alpha\left(\epsilon\left|\nabla_{h} \hat{\boldsymbol{m}}_{h}^{n+2}\right|^{2}-\hat{\boldsymbol{m}}_{h}^{n+2} \cdot \hat{\boldsymbol{f}}_{h}^{n+2}\right) \hat{\boldsymbol{m}}_{h}^{n+2} \\
		\boldsymbol{m}_{h}^{n+2}=\frac{\tilde{\boldsymbol{m}}_{h}^{n+2}}{\left|\tilde{\boldsymbol{m}}_{h}^{n+2}\right|}
	\end{array}\right.
	\label{eq-6}
\end{align}
where $$
\begin{aligned}
	\hat{\boldsymbol{m}}_{h}^{n+2} &=2 \boldsymbol{m}_{h}^{n+1}-\boldsymbol{m}_{h}^{n}, \quad
	\hat{\boldsymbol{f}}_{h}^{n+2} &=2 \boldsymbol{f}_{h}^{n+1}-\boldsymbol{f}_{h}^{n}.
\end{aligned}
$$
At each time step, only one linear system with constant coefficients and SPD structure needs to be solved with fast solvers, such as fast Fourier transform (FFT).

\subsection{Implicit-explicit Runge-Kutta methods}\label{IMEX-RK}
For a given time-dependent nonlinear problem, the basic idea of implicit-explicit methods is to treat a dominant linear term implicitly and the remaining terms explicitly. The success of IMEX methods relies on the dominance of linear term that is implicitly treated \cite{ascher1997implicit}. While it is not the case for all problems, the introduction of an artificial term may help, such as the linear diffusion term introduced for the nonlinear diffusion equation \cite{wang2020local}. For the LL equation, the linear diffusion term does not dominate the magnetization dynamics, and thus a direct application of IMEX method does not work. Motivated by this observation and the work in \cite{wang2020local}, we propose to introduce an artificial diffusion term in the LL equation which will be implicitly treated while all other terms are explicitly treated explicitly. RK methods are employed for the time discretization.

For ease of description, we first list the Butcher tableau for second-order and third-order RK methods in \eqref{eq-8} and \eqref{eqn:RK3}.
\begin{equation}
	\begin{array}{c|ccc|ccc}
		0 & 0 & 0 & 0 & 0 & 0 & 0 \\
		1/2 & 0 & 1/2 & 0 & 1/2 & 0 & 0 \\
		1 & 1/2 & 0 & 1/2 & 0 & 1 & 0 \\
		\hline & 1/2 & 0 & 1/2 & 0 & 1 & 0
	\end{array}
	\label{eq-8}
\end{equation}
\begin{equation}\label{eqn:RK3}
	\begin{array}{c|ccccc|ccccc}
		\rule{0pt}{20pt}
		0 & 0 & 0 & 0 & 0 & 0 & 0 & 0 & 0 & 0 & 0 \\
		1/2 & 0 & 1/2 & 0 & 0 & 0 & 1/2 & 0 & 0 & 0 & 0 \\
		2/3 & 0 & 1/6 & 1/2 & 0 & 0 & 11/18 &1/18 & 0 & 0 & 0 \\
		1/2 & 0 & -1/2 & 1/2 & 1/2 & 0 & 5/6 & -5/6 & 1/2 & 0 & 0 \\
		1 & 0 & 3/2 & -3/2 & 1/2 & 1/2 & 1/4 & 7/4 & 3/4 & -7/4 & 0 \\
		\hline & 0 & 3/2 & -3/2 & 1/2 & 1/2 & 1/4 & 7/4 & 3/4 & -7/4 & 0
	\end{array}
\end{equation}

We add an artificial damping term $\beta\Delta \boldsymbol{m}$ into \eqref{eq-3} and rewrite the LL equation as
\begin{align}
	{{\boldsymbol m}_{t}} = \underbrace{-\boldsymbol m\times (\epsilon \Delta \boldsymbol m+\emph{ \textbf{f}})-\alpha \boldsymbol m\times \boldsymbol m\times (\epsilon \Delta \boldsymbol m+\emph{ \textbf{f}}) -\beta   {{\Delta }} \boldsymbol m}_{{N(t,\boldsymbol m)}}+\underbrace{\beta   {{\Delta }}  \boldsymbol m}_{{L(t,\boldsymbol m)}},
	\label{eq-9}
\end{align}
where the artificial term is denoted as $L(t,\boldsymbol m)$, and all the remaining terms are included in  $N(t,\boldsymbol m)$.

Therefore, at time step $t_n$, the corresponding marching algorithms in IMEX-RK2 and IMEX-RK3 methods are
\begin{equation}\label{IMEX-RK2}
	\left\{\begin{array}{l}
		\boldsymbol{\tilde{m}}_{1}=\boldsymbol{{m}}_{n} \\
		{\boldsymbol{\tilde{m}}_{2}}=\boldsymbol{{\tilde{m}}}_{1}+\frac{k}{2}\left(L(t_{n+\frac{1}{2}}, {\boldsymbol{\tilde{m}}_{2}})+N\left(t_{n}, \boldsymbol{{\tilde{m}}}_{1}\right)\right)\\
		{\boldsymbol{\tilde{m}}_{3}}=\boldsymbol{{\tilde{m}}}_{1}+\frac{k}{2}\left(L\left(t_{n}, \boldsymbol{{\tilde{m}}}_{1}\right)+L\left(t_{n+1}, {\boldsymbol{\tilde{m}}_{3}}\right)+2 N(t_{n+\frac{1}{2}}, \boldsymbol{\tilde{m}}_{2})\right)\\
		\boldsymbol{{m}}_{n+1}=\boldsymbol{{\tilde{m}}}_{1}+\frac{k}{2}\left(L\left(t_{n}, \boldsymbol{{\tilde{m}}}_{1}\right)+L\left(t_{n+1}, \boldsymbol{\tilde{m}}_{3}\right)+2 N(t_{n+\frac{1}{2}}, \boldsymbol{\tilde{m}}_{2})\right)
	\end{array}\right.,
\end{equation}
and
\begin{equation}\label{IMEX-RK3}
	\left\{\begin{array}{l}
		\boldsymbol{{\tilde{m}}}_{1}=\boldsymbol{{m}}_{n} \\ {\boldsymbol{\tilde{m}}_{2}}=\boldsymbol{{\tilde{m}}}_{1}+\frac{k}{2}\left(L(t_{n+\frac{1}{2}}, {\boldsymbol{\tilde{m}}_{2}})+N\left(t_{n}, \boldsymbol{{\tilde{m}}}_{1}\right)\right) \\ {\boldsymbol{\tilde{m}}_{3}}=\boldsymbol{{\tilde{m}}}_{1}+\frac{k}{6}L({{t}_{n+\frac{1}{2}}},
		\boldsymbol{\tilde{m}}_{2})+\frac{k}{2}L({{t}_{n+\frac{2}{3}}},{\boldsymbol{\tilde{m}}_{3}})+\frac{11k}{18}N({{t}_{n}},
		\boldsymbol{{\tilde{m}}}_{1})+\frac{k}{18}N({{t}_{n+\frac{1}{2}}},\boldsymbol{\tilde{m}}_{2})\\
		{\boldsymbol{\tilde{m}}_{4}}=\boldsymbol{{\tilde{m}}}_{1}-\frac{k}{2}L({{t}_{n+\frac{1}{2}}},
		\boldsymbol{\tilde{m}}_{2})+\frac{k}{2}L({{t}_{n+\frac{2}{3}}},\boldsymbol{\tilde{m}}_{3})
		+\frac{k}{2}L({{t}_{n+\frac{1}{2}}},{\boldsymbol{\tilde{m}}_{4}})\\
		+\frac{5k}{6}N({{t}_{n}},\boldsymbol{{\tilde{m}}}_{1})-\frac{5k}{6}N({{t}_{n+\frac{1}{2}}},
		\boldsymbol{\tilde{m}}_{2})+\frac{k}{2}N({{t}_{n+\frac{2}{3}}},\boldsymbol{\tilde{m}}_{3})\\
		{\boldsymbol{\tilde{m}}_{5}}=\boldsymbol{{\tilde{m}}}_{1}+\frac{3k}{2}L({{t}_{n+\frac{1}{2}}},
		\boldsymbol{\tilde{m}}_{2})-\frac{3k}{2}L({{t}_{n+\frac{2}{3}}},\boldsymbol{\tilde{m}}_{3})+\frac{k}{2}L({{t}_{n+\frac{1}{2}}},
		\boldsymbol{\tilde{m}}_{4})+\frac{k}{2}L({{t}_{n+1}},{\boldsymbol{\tilde{m}}_{5}})\\
		+\frac{k}{4}N({{t}_{n}},\boldsymbol{{\tilde{m}}}_{1})+\frac{7k}{4}N({{t}_{n+\frac{1}{2}}},
		\boldsymbol{\tilde{m}}_{2})+\frac{3k}{4}N({{t}_{n+\frac{2}{3}}},\boldsymbol{\tilde{m}}_{3})-\frac{7k}{4}N({{t}_{n+\frac{1}{2}}},
		\boldsymbol{\tilde{m}}_{4})\\
		\boldsymbol{{m}}_{n+1}=\boldsymbol{{\tilde{m}}}_{1}+\frac{3k}{2} L({{t}_{n+\frac{1}{2}}},\boldsymbol{\tilde{m}}_{2})-\frac{3k}{2} L({{t}_{n+\frac{2 }{3}}},\boldsymbol{\tilde{m}}_{3})+\frac{k}{2} L({{t}_{n+\frac{1 }{2}}},\boldsymbol{\tilde{m}}_{4})+\frac{k}{2}L({{t}_{n+1}},\boldsymbol{\tilde{m}}_{5})\\
		+\frac{k }{4}N({{t}_{n}},\boldsymbol{{\tilde{m}}}_{1})+\frac{7k }{4}N({{t}_{n+\frac{1 }{2}}},\boldsymbol{\tilde{m}}_{2})+\frac{3k }{4}N({{t}_{n+\frac{2 }{3}}},\boldsymbol{\tilde{m}}_{3})-\frac{7k }{4}N({{t}_{n+\frac{1 }{2}}},\boldsymbol{\tilde{m}}_{4})
	\end{array}\right. .
\end{equation}

In addition, if ${\boldsymbol f}$ is time-independent, so that the rewritten LL equation~\eqref{eq-9} is autonomous, an alternate IMEX-RK2 numerical algorithm could be chosen: 
\begin{equation}\label{SSP-IMEX-RK2}
	\left\{\begin{array}{l}
		\boldsymbol{\tilde{m}}_{1}=\boldsymbol{{m}}_{n} \\
		{\boldsymbol{\tilde{m}}_{2}}=\boldsymbol{{\tilde{m}}}_{1} +\frac{k}{4} L({\boldsymbol{\tilde{m}}_{2}}) \\
		{\boldsymbol{\tilde{m}}_{3}}=\boldsymbol{{\tilde{m}}}_{1} + \frac{k}{2} N (\boldsymbol{{\tilde{m}}}_{2} ) +\frac{k}{4} L({\boldsymbol{\tilde{m}}_{3}}) \\
		{\boldsymbol{\tilde{m}}_{4}}=\boldsymbol{{\tilde{m}}}_{1} + \frac{k}{2} \left( N ( \boldsymbol{{\tilde{m}}}_{2}
		) + N ( \boldsymbol{{\tilde{m}}}_{3} ) \right) + \frac{k}{3} \left( L({\boldsymbol{\tilde{m}}_{2}}) + L({\boldsymbol{\tilde{m}}_{3}}) + L({\boldsymbol{\tilde{m}}_{4}}) \right) \\
		\boldsymbol{{m}}_{n+1}=\boldsymbol{{\tilde{m}}}_{1} + \frac{k}{3} \left( N ( \boldsymbol{{\tilde{m}}}_{2}
		) + N ( \boldsymbol{{\tilde{m}}}_{3} ) + N ( \boldsymbol{{\tilde{m}}}_{4} ) \right) \\
		\qquad \qquad \qquad
		+ \frac{k}{3} \left( L({\boldsymbol{\tilde{m}}_{2}}) + L({\boldsymbol{\tilde{m}}_{3}}) + L({\boldsymbol{\tilde{m}}_{4}}) \right) .
	\end{array}\right.,
\end{equation}
This IMEX-RK2 algorithm contains four stages, with three intermediate numerical solutions at each time step. On the other hand, the explicit part satisfies the strong stability-preserving  property~\cite{Conde2017, gottlieb01a}; as a result, we denote it as the SSP-IMEX-RK2 scheme. In addition, this numerical method contains stronger diffusion coefficients, in comparison with the standard IMEX-RK2 algorithm~\eqref{IMEX-RK2}, and this feature will greatly facilitate a theoretical justification of numerical stability and convergence analysis, as will be presented in Section~\ref{sec: convergence}.
\vskip2mm
\begin{remark}
	In the current work, the finite difference method is employed for the spatial discretization. It is worth mentioning that other spatial discretization, such as finite element method and discontinuous Galerkin method, can also be employed.
\end{remark}
\vskip2mm
\begin{remark}
	Two linear systems with constant coefficients and SPD structure are solved in IMEX-RK2 \eqref{IMEX-RK2}. Although only one linear system solver is needed, with constant coefficients and SPD structure, the method in \cite{cai2022second} only works in the large damping parameter case. Similarly, three linear system solvers are needed in the SSP-IMEX-RK2 method~\eqref{SSP-IMEX-RK2}, and four linear system solvers are needed in the IMEX-RK3 method \eqref{IMEX-RK3}, with constant coefficients and SPD structure. It will be verified in Section~\ref{S:Numerical tests} that IMEX-RK methods work for any damping parameter.
\end{remark}

\section{Numerical results}\label{S:Numerical tests}

In this section, we provide a series of numerical experiments for IMEX-RK methods, including accuracy check, efficiency comparison, and stability test in terms of different $\beta$ values. Denote $ \boldsymbol{m}_{e}$ the exact solution and $\boldsymbol{m}_{h}$ the numerical solution. To measure the error, we introduce the following notations in the discrete case.
\vskip2mm
\begin{definition} [$\ell^2$ inner product, ${{\left\| \cdot  \right\|}_{2}}$ norm]
	For grid functions ${{\boldsymbol f}_{h}}$ and ${{\boldsymbol g}_{h}}$ that take values on a uniform numerical grid, we define
	\begin{equation*}
		\left\langle\boldsymbol{f}_{h}, \boldsymbol{g}_{h}\right\rangle=h^{d} \sum_{\mathcal{I} \in \Lambda_{d}} \boldsymbol{f}_{\mathcal{I}} \cdot \boldsymbol{g}_{\mathcal{I}} ,
	\end{equation*}
	where $\Lambda_{d}$ is the set of grid point, and $\mathcal{I}$ is an index.
\end{definition}

In turn, the ${{\left\| \cdot  \right\|}_{2}}$ norm is defined as
\begin{equation*}
	\left\|\boldsymbol{f}_{h}\right\|_{2}=\left(\left\langle\boldsymbol{f}_{h}, \boldsymbol{f}_{h}\right\rangle\right)^{1 / 2}.
\end{equation*}
In addition, the discrete ${H^{1}}$ norm is defined as
\begin{equation*}
	\left\|\boldsymbol{f}_{h}\right\|_{H^{1}}^{2}:=\left\|\boldsymbol{f}_{h}\right\|_{2}^{2}+\left\|\nabla_{h} \boldsymbol{f}_{h}\right\|_{2}^{2}.
\end{equation*}

\begin{definition}[ ${{\left\| \cdot  \right\|}_{\infty}}$ and ${{\left\| \cdot  \right\|}_{p}}$ norms in the discrete sense]
	For grid functions ${{\boldsymbol f}_{h}}$ that take values on a uniform numerical grid, we define
	\begin{equation*}
		\left\|\boldsymbol f_{h}\right\|_{\infty}=\max _{\mathcal{I} \in \Lambda_{d}}\left\|\boldsymbol{f}_{\mathcal{I}}\right\|_{\infty}, \quad\left\|\boldsymbol{f}_{h}\right\|_{p}=\left(h^{d} \sum_{I \in \Lambda_{d}}\left|\boldsymbol{f}_{\mathcal{I}}\right|^{p}\right)^{\frac{1}{p}}, 1 \leq p<+\infty .
	\end{equation*}
\end{definition}

\begin{lem}[Summation by parts]\label{summation}
	For any grid functions ${{\boldsymbol f}_{h}}$ and ${{\boldsymbol g}_{h}}$, with ${{\boldsymbol f}_{h}}$ satisfying the discrete boundary condition~\eqref{BC-1}, the following identity is valid:
	\begin{align} \label{sum1}
		\left\langle -\Delta_h \boldsymbol f_h,\boldsymbol g_h\right\rangle = \left\langle \nabla_h \boldsymbol f_h,\nabla_h \boldsymbol g_h\right\rangle .
	\end{align}
\end{lem}
The proof of the standard inverse inequality can be obtained in existing textbooks and references; we just cite the results here. In the sequel, for simplicity of notation, we will use the uniform constant $\mathcal{C}$ to denote all the controllable constants.
\vskip2mm
\begin{lem} [Inverse inequality] \cite{chen16, chenW20a, Ciarlet1978} \label{ccclemC1}
	For each vector-valued grid function $\boldsymbol f_h\in X$, we have
	\begin{align}
		&
		\|{\boldsymbol f}_h \|_{\infty} \leq \gamma {h}^{-1/2 }
		( \|{\boldsymbol f}_h \|_2 + \| \nabla_h \boldsymbol f_h \|_2 ) ,   \label{inverse-1} 	
		\\
		& 	
		\| \boldsymbol f_h \|_q \leq \gamma {h}^{-(\frac32 - \frac{3}{q}) } \| \boldsymbol f_h \|_2 ,  \, \, \, \forall 2 < q \le + \infty ,
		\label{inverse-2}	
	\end{align}
	in which constant $\gamma$ depends on $\Omega$, as well as the form of the discrete $\| \cdot \|_2$ norm. 
\end{lem}

The following discrete Sobolev inequality has been derived in the existing works~\cite{guan17a, guan14a}, for the discrete grid function with periodic boundary condition; an extension to the discrete homogeneous Neumann boundary condition can be made in a similar fashion.
\vskip2mm
\begin{lem} [Discrete Sobolev inequality] \cite{guan17a, guan14a} \label{lem: Sobolev-1}
	For a grid function $\boldsymbol f_h \in \boldsymbol X$,  we have the following discrete Sobolev inequality:
	\begin{align}
		\| \boldsymbol f_h \|_4 \le \mathcal{C} 	\| \boldsymbol f_h \|_2^\frac14 \cdot \| \boldsymbol f_h \|_{H_h^1}^\frac34
		\le  \mathcal{C}  ( \| \boldsymbol f_h \|_2  	
		+   	\| \boldsymbol f_h \|_2^\frac14 \cdot \| \nabla_h \boldsymbol f_h \|_2^\frac34 ) , \label{Sobolev-1}
	\end{align}
	in which the positive constant $\mathcal{C}$ only depends on the domain $\Omega$.
\end{lem}

In the numerical simulation, we set $\epsilon=1$ and $\boldsymbol{f}=0$ in \eqref{eq-4} for convenience. The 1-D exact solution is chosen to be
$$
\boldsymbol{m}_{e}=(\cos (X) \sin t, \sin (X) \sin t, \cos t)^{T},
$$
and the 3-D exact solution is set as
$$
\boldsymbol{m}_{e}=(\cos (X Y Z) \sin t, \sin (X Y Z) \sin t, \cos t)^{T},
$$
where $X=x^{2}(1-x)^{2}, Y=y^{2}(1-y)^{2}, Z=z^{2}(1-z)^{2}$. It is easy to check that the homogeneous Neumann boundary condition \eqref{eq-2} is satisfied. A forcing term $\boldsymbol{f}_{e}=\partial_{t} \boldsymbol{m}_{e}-\alpha \Delta \boldsymbol{m}_{e}-\alpha\left|\nabla \boldsymbol{m}_{e}\right|^{2}+\boldsymbol{m}_{e} \times$ $\Delta \boldsymbol{m}_{e}$ is included into the nonlinear part $N(t,\boldsymbol{m})$.

\subsection{Accuracy and efficiency test of IMEX-RK2}
In the 1-D computation, we fix $h=1/5000$ and record the error in terms of $k$ in Table~\ref{table2}, and fix $k=1e-3/10000$ and record the error in terms of $h$ in Table~\ref{table3}. The second-order accuracy of IMEX-RK2 is observed in both space and time.
\begin{table}[htbp]
	\centering
	\caption{Temporal accuracy of IMEX-RK2 in the 1-D computation ($\tau=1e-4, h=1/5000,\alpha=0.01$, and $\beta=5$).}\label{table2}
	\begin{tabular}{cccc}
		\hline$k$ & $\left\|\boldsymbol m_{h}-\boldsymbol m_{e}\right\|_{\infty}$ & $\left\|\boldsymbol m_{h}-\boldsymbol m_{e}\right\|_{2}$ & $\left\|\boldsymbol m_{h}-\boldsymbol m_{e}\right\|_{H^{1}}$ \\
		\hline
		$\tau / 5$ & $1.7011e-10$ & $2.7861e-11$ & $1.7582e-09$ \\
		$\tau / 10$ & $4.4130e-11$ & $8.1126e-12$ & $6.3812e-10$ \\
		$\tau / 15$ & $2.1077e-11$ & $3.8709e-12$ & $3.3541e-10$ \\
		$\tau / 20$ & $1.2028e-11$ & $2.2727e-12$ & $2.1471e-10$ \\
		$\tau / 25$ & $8.1095e-12$ & $1.4985e-12$ & $1.6165e-10$ \\
		order & $1.8930$ & $1.8152$ & $1.5001$ \\
		\hline
	\end{tabular}
\end{table}
\begin{table}[htbp]
	\centering
	\caption{Spatial accuracy of IMEX-RK2 in the 1-D computation ($k=1e-7,\alpha=0.01$, and $\beta=5$).}\label{table3}
	\begin{tabular}{cccc}
		\hline$h$ & $\left\|\boldsymbol m_{h}-\boldsymbol m_{e}\right\|_{\infty}$ & $\left\|\boldsymbol m_{h}-\boldsymbol m_{e}\right\|_{2}$ & $\left\|\boldsymbol m_{h}-\boldsymbol m_{e}\right\|_{H^{1}}$ \\
		\hline
		$1 / 50 $ & $2.7361e-09$ & $8.3040e-10$ & $4.5150e-07$ \\
		$1 / 100$ & $7.3177e-10$ & $2.1625e-10$ & $1.1307e-07$ \\
		$1 / 150$ & $3.2921e-10$ & $9.6819e-11$ & $5.0272e-08$ \\
		$1 / 200$ & $1.8597e-10$ & $5.4601e-11$ & $2.8281e-08$ \\
		$1 / 250$ & $1.1926e-10$ & $3.4986e-11$ & $1.8101e-08$ \\
		order & $1.9472$ & $1.9681$ & $1.9986$ \\
		\hline
	\end{tabular}
\end{table}
In the 3-D computation, we fix $h=1/16$ and record the error in terms of $k$ in Table~\ref{table4}, and fix $k=1/10000$ and record the error in terms of $h$ in Table~\ref{table5}. Again, the second-order accuracy of IMEX-RK2 is observed in both space and time.
\begin{table}[htbp]
	\centering
	\caption{Temporal accuracy of IMEX-RK2 in the 3-D computation ($h=1/16,\alpha=0.01$, and $\beta=5$). }\label{table4}
	\begin{tabular}{cccc}
		\hline$k$ & $\left\|\boldsymbol m_{h}-\boldsymbol m_{e}\right\|_{\infty}$ & $\left\|\boldsymbol m_{h}-\boldsymbol m_{e}\right\|_{2}$ & $\left\|\boldsymbol m_{h}-\boldsymbol m_{e}\right\|_{H^{1}}$ \\
		\hline
		$1/4$ & $0.0022$               & $0.0025$ & $0.0025$ \\
		$1/6 $ & $0.0010$ & $0.0011$ & $0.0011$ \\
		$1/8$ & $5.5504e-04$ & $6.3857e-04$ & $6.6930e-04$ \\
		$1/10$ & $3.6163e-04$ & $4.1354e-04$ & $4.6041e-04$ \\
		order & $1.9773$ & $1.9600$ & $1.8594$ \\
		\hline
	\end{tabular}
\end{table}
\begin{table}[htbp]
	\centering
	\caption{Spatial accuracy of IMEX-RK2 in the 3-D computation ($k=1/10000,\alpha=0.01$, and $\beta=5$). }\label{table5}
	\begin{tabular}{cccc}
		\hline$h$ & $\left\|\boldsymbol m_{h}-\boldsymbol m_{e}\right\|_{\infty}$ & $\left\|\boldsymbol m_{h}-\boldsymbol m_{e}\right\|_{2}$ & $\left\|\boldsymbol m_{h}-\boldsymbol m_{e}\right\|_{H^{1}}$ \\
		\hline
		$1 / 3 $ & $1.4756e-04$ & $1.7447e-04$ & $2.5609e-04$ \\
		$1 / 5 $ & $5.2131e-05$ & $6.1669e-05$ & $9.3673e-05$ \\
		$1 / 7 $ & $2.6372e-05$ & $3.1402e-05$ & $4.7958e-05$ \\
		$1 / 9 $ & $1.5935e-05$ & $1.8985e-05$ & $2.9167e-05$ \\
		$1 / 11$ & $1.0670e-05$ & $1.2707e-05$ & $1.9579e-05$ \\
		order & $2.0223 $ & $2.0155$ & $1.9794$ \\
		\hline
	\end{tabular}
\end{table}

In terms of the efficiency comparison, we plot the CPU time (in seconds) vs. the error $\left\|\boldsymbol m_{h}-\boldsymbol m_{e}\right\|_{\infty}$.
Results of BDF2 and IMEX-RK2 are visualized in Fig.~\ref{time_1D} and Fig.~\ref{space_1D} for the 1-D case, and in Fig.~\ref{time_3D} and  Fig.~\ref{space_3D} for the 3-D case.
\begin{figure}[htbp]
	\flushright	
	\subfloat[$k$]{\label{time_1D}\includegraphics[width=2.5in]{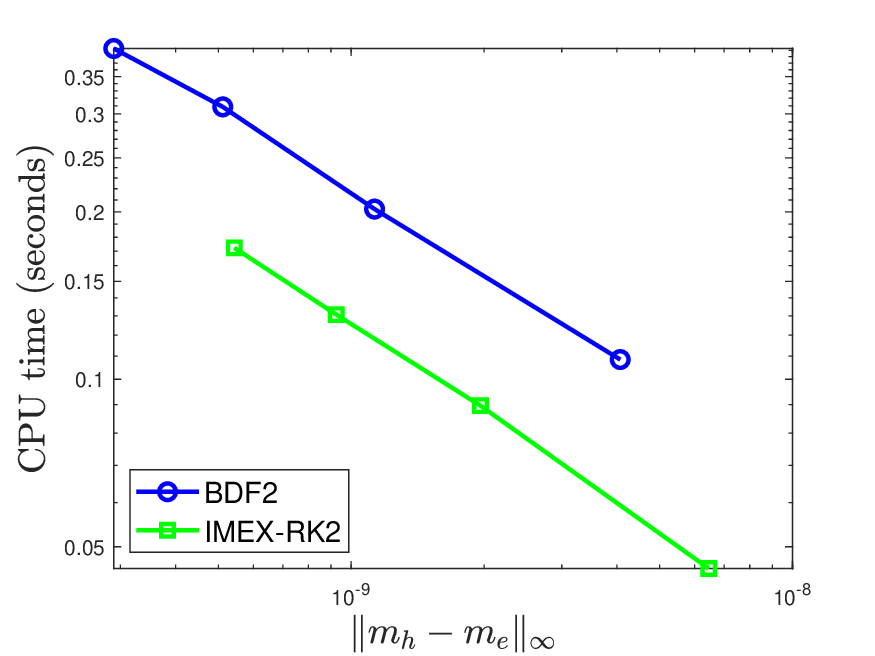}}
	\subfloat[$h$]{\label{space_1D}\includegraphics[width=2.5in]{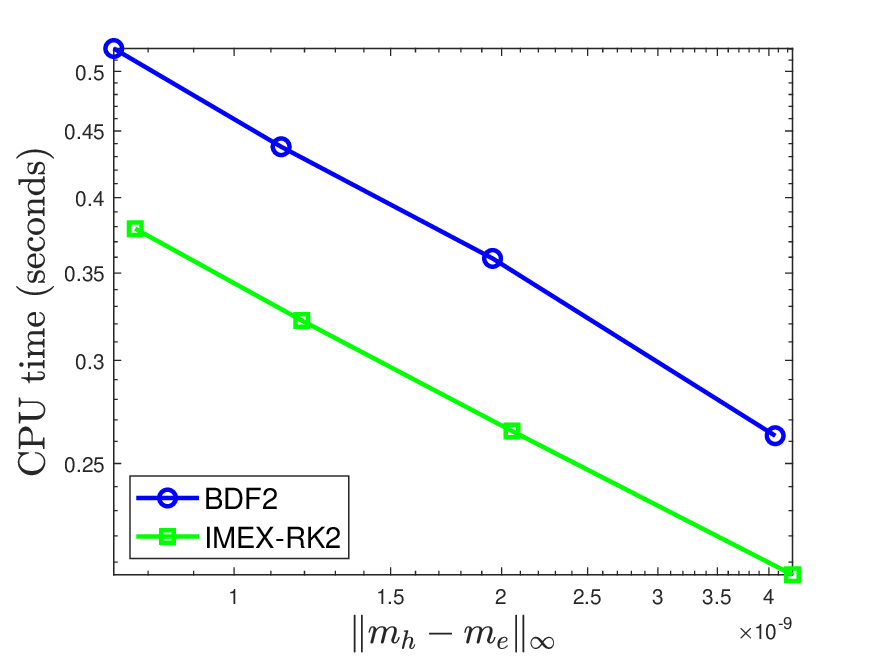}}
	\quad	
	\subfloat[$k$]{\label{time_3D}\includegraphics[width=2.5in]{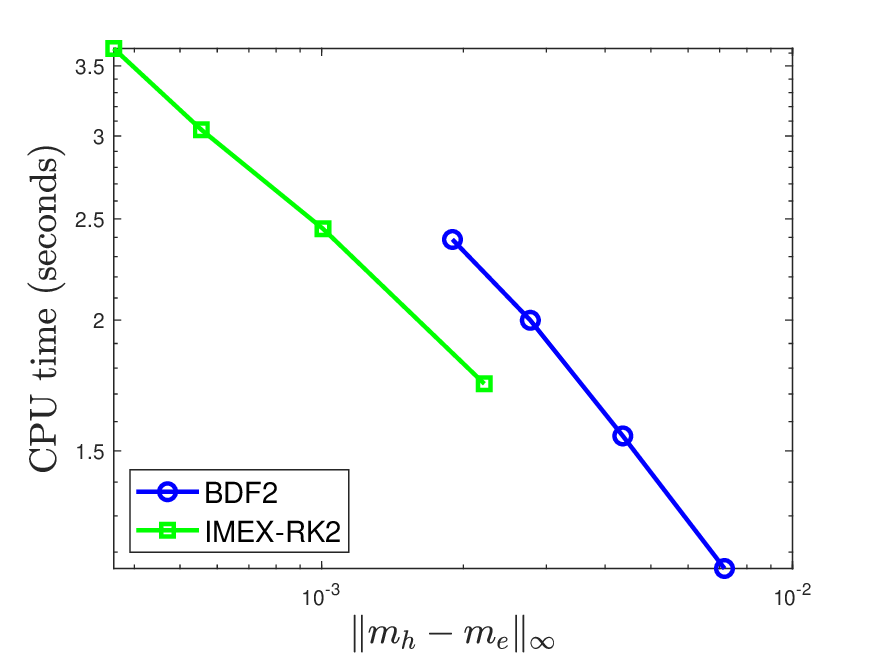}}
	\subfloat[$h$]{\label{space_3D}\includegraphics[width=2.5in]{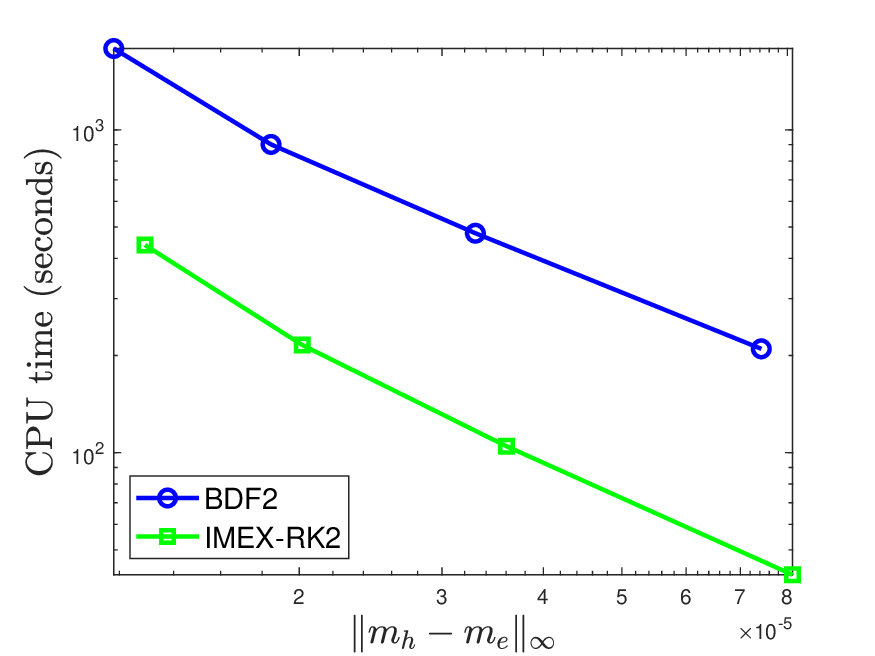}}
	\caption{ IMEX-RK2 vs. BDF2: CPU time (in seconds) as a function of error by varying $k$ and $h$, respectively. Top row: 1-D; Bottom row: 3-D.}
	\label{rk2 cpu time}
\end{figure}
By Fig.~\ref{rk2 cpu time}, IMEX-RK2 is superior to BDF2 in both the 1-D and 3-D computations.

\subsection{Accuracy and efficiency test of IMEX-RK3}
Using the same spatial resolution, we only test the temporal accuracy of IMEX-RK3 here. In the 1-D computation, we fix $k =0.01\times {{h}^{\frac{2}{3}}}$ and record the error in terms of $k$ in Table~\ref{table6}.
\begin{table}[htbp]
	\centering
	\caption{Temporal accuracy of IMEX-RK3 in the 1-D computation ($k =0.01\times {{h}^{\frac{2}{3}}}$, $\alpha=0.01$, and $\beta=5$). }\label{table6}
	\begin{tabular}{cccc}
		\hline$k$ & $\left\|\boldsymbol m_{h}-\boldsymbol m_{e}\right\|_{\infty}$ & $\left\|\boldsymbol m_{h}-\boldsymbol m_{e}\right\|_{2}$ & $\left\|\boldsymbol m_{h}-\boldsymbol m_{e}\right\|_{H^{1}}$ \\
		\hline
		$1 / 208$ & $0.0435$ & $0.0508$ & $0.0921$ \\
		$1 / 252$ & $0.0244$ & $0.0287$ & $0.0527$ \\
		$1 / 292$ & $0.0155$ & $0.0184$ & $0.0341$ \\
		$1 / 330$ & $0.0109$ & $0.0128$ & $0.0236$ \\
		order & $3.0235$ & $2.9930$ & $2.8900$ \\
		\hline
	\end{tabular}
\end{table}
In the 3-D computation, we fix $k =0.001\times {{h}^{\frac{2}{3}}}$ and record the error in terms of $k$ in Table
~\ref{table8}. The third-order temporal accuracy is observed for IMEX-RK3 in time, in both the 1-D and 3-D compuations.
\begin{table}[htbp]
	\centering
	\caption{Temporal accuracy of IMEX-RK3 in the 3-D compuation ($k =0.001\times {{h}^{\frac{2}{3}}}$,$\alpha=0.01$, and $\beta=5$).}\label{table8}
	\begin{tabular}{cccc}
		\hline$k$ & $\left\|\boldsymbol m_{h}-\boldsymbol m_{e}\right\|_{\infty}$ & $\left\|\boldsymbol m_{h}-\boldsymbol m_{e}\right\|_{2}$ & $\left\|\boldsymbol m_{h}-\boldsymbol m_{e}\right\|_{H^{1}}$ \\
		\hline
		$1/2080$ & $1.4755e-04$ & $1.7446e-04$ & $2.5608e-04$ \\
		$1/2520$ & $8.1182e-05$ & $9.6730e-05$ & $1.4581e-04$ \\
		$1/2924$ & $5.2128e-05$ & $6.1666e-05$ & $9.3671e-05$ \\
		$1/3302$ & $3.6028e-05$ & $4.2767e-05$ & $6.5140e-05$ \\
		order & $3.0460$ & $3.0422$ & $2.9621$ \\
		\hline
	\end{tabular}
\end{table}

Since IMEX-RK2 and IMEX-RK3 only differ in the temporal discretization, we further plot the CPU time (in seconds) of IMEX-RK2 and IMEX-RK3 in terms of the temporal error in the 1-D and 3-D cases; see Fig.~\ref{rk3 cpu time}. These results indicate that IMEX-RK3 is more efficient than IMEX-RK2, and thus is more efficient than BDF2.
\begin{figure}[htbp]
	\subfloat[1D]{\label{1d time}\includegraphics[width=2.5in]{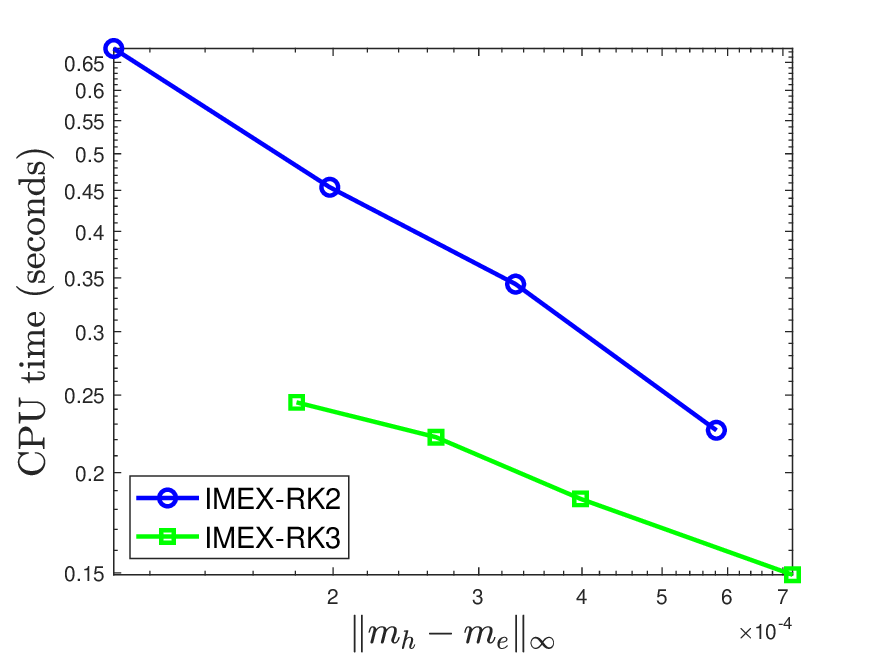}}
	\subfloat[3D]{\label{3d time}\includegraphics[width=2.5in]{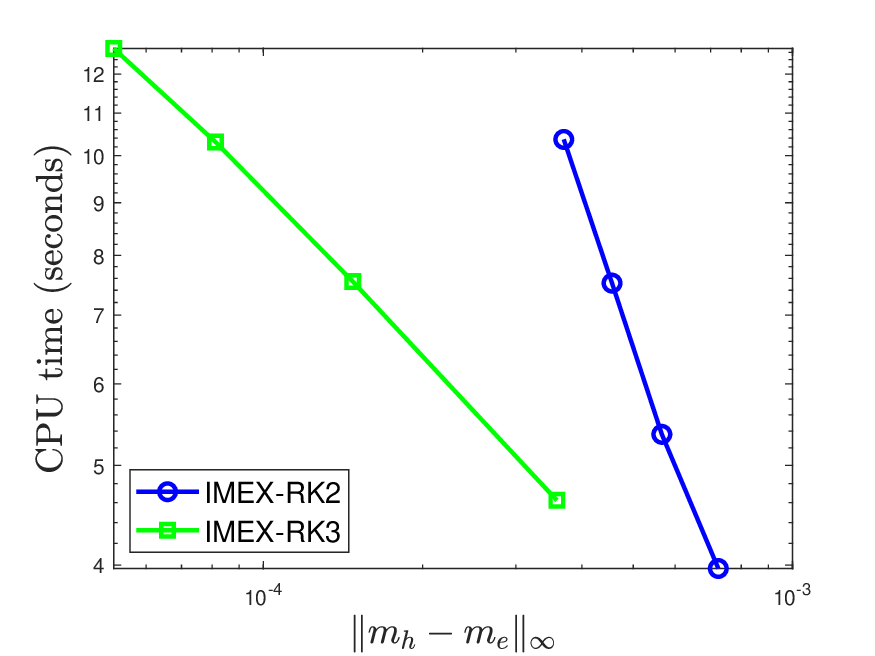}}
	\caption{ IMEX-RK2 vs. IMEX-RK3: CPU time (in seconds) as a function of temporal error by varying $k$.}
	\label{rk3 cpu time}
\end{figure}

\subsection{Accuracy test of SSP-IMEX-RK2}
In the 1-D test, we keep $h=5\times {{10}^{1}}k$ and record the error in terms of $k$ in Table
~\ref{table9}. In the 3-D computation, we fix $h=1\times {{10}^{3}}k$ and record the error in terms of $k$ in Table~\ref{table10}. Second-order accuracy of SSP-IMEX-RK2 has been observed in both the 1-D and 3-D cases. The accuracy curves are displayed in Fig.~\ref{ssp imex rk2}.

\begin{table}[htbp]
	\centering
	\caption{Temporal accuracy of SSP-IMEX-RK2 in the 1-D computation ($\tau=2e-2, \alpha=0.01$, and $\beta=5$).}\label{table9}
	\begin{tabular}{cccc}
		\hline$k$ & $\left\|\boldsymbol m_{h}-\boldsymbol m_{e}\right\|_{\infty}$ & $\left\|\boldsymbol m_{h}-\boldsymbol m_{e}\right\|_{2}$ & $\left\|\boldsymbol m_{h}-\boldsymbol m_{e}\right\|_{H^{1}}$ \\
		\hline
		$\tau / 3 $ & $5.4161e-05$ & $1.1077e-05$ & $0.0019$ \\
		$\tau / 4 $ & $3.1570e-05$ & $5.9625e-06$ & $0.0012$ \\
		$\tau / 5 $ & $2.0783e-05$ & $3.6472e-06$ & $7.7420e-04$ \\
		$\tau / 6 $ & $1.4901e-05$ & $2.4208e-06$ & $5.4551e-04$ \\
		order & $1.8640$ & $2.1927$ & $1.8076$ \\
		\hline
	\end{tabular}
\end{table}

\begin{table}[htbp]
	\centering
	\caption{Temporal accuracy of SSP-IMEX-RK2 in the 3-D computation ($\tau=1e-3,\alpha=0.01$, and $\beta=5$).}\label{table10}
	\begin{tabular}{cccc}
		\hline$h$ & $\left\|\boldsymbol m_{h}-\boldsymbol m_{e}\right\|_{\infty}$ & $\left\|\boldsymbol m_{h}-\boldsymbol m_{e}\right\|_{2}$ & $\left\|\boldsymbol m_{h}-\boldsymbol m_{e}\right\|_{H^{1}}$ \\
		\hline
		$\tau / 8 $ & $1.7697e-10$ & $6.6171e-11$ & $4.5228e-08$ \\
		$\tau / 10$ & $1.1599e-10$ & $4.5129e-11$ & $2.9500e-08$ \\
		$\tau / 12$ & $8.1594e-11$ & $3.3029e-11$ & $2.0741e-08$ \\
		$\tau / 14$ & $6.0414e-11$ & $2.5397e-11$ & $1.5370e-08$ \\
		order & $1.9203$ & $1.7115$ & $1.9284$ \\
		\hline
	\end{tabular}
\end{table}
\begin{figure}[htbp]
	\subfloat[1D]{\includegraphics[width=2.5in]{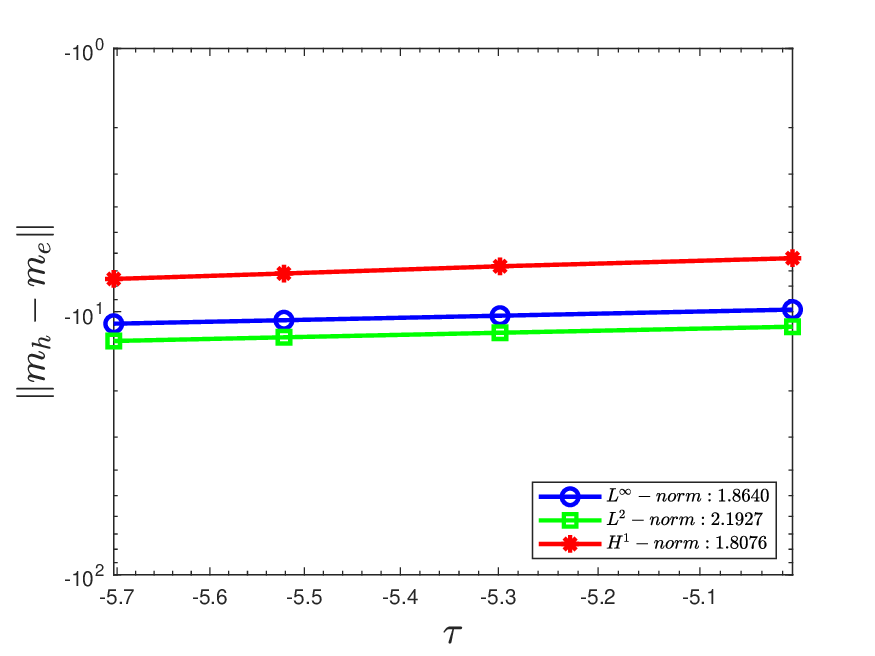}}
	\subfloat[3D]{\includegraphics[width=2.5in]{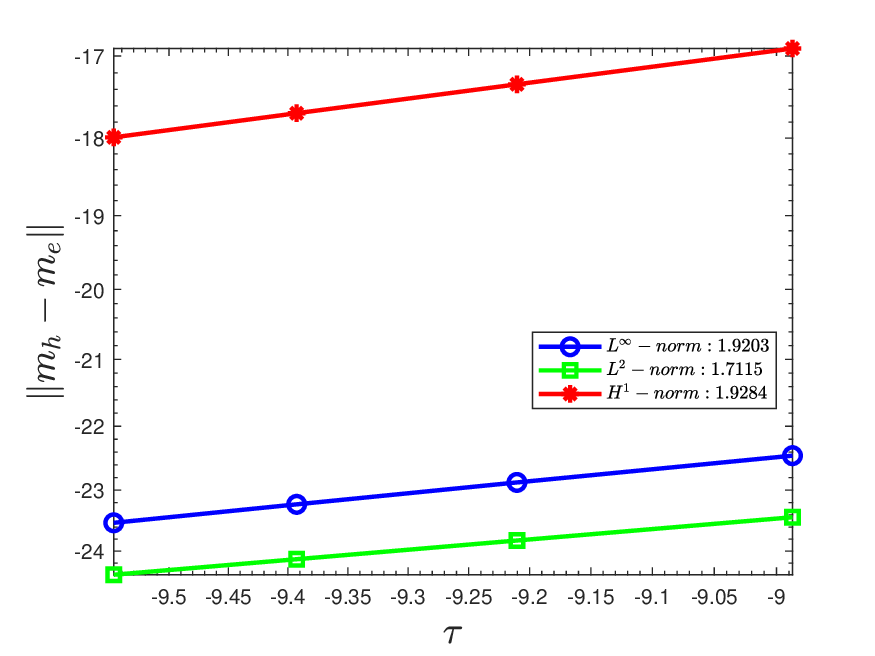}}
	\caption{Temporal accuracy in the 1-D and 3-D computations of the SSP-IMEX-RK2 scheme.}
	\label{ssp imex rk2}
\end{figure}

\subsection{Dependence on the damping parameter}
There are a few numerical methods that only several linear systems with constant coefficients and SPD structure need to be solved at each time step, including the first-order-in-time GSPM \cite{wang2001gauss,Li2020TwoIG}, the second-order-in-time method \cite{cai2022second}, and the current method. Numerically, the method in \cite{cai2022second} only works when $\alpha>1$, most magnetic materials correspond to $\alpha\ll 1$. If $\alpha>1$, we can set $\beta=\alpha$ and then apply the idea of IMEX-RK. Therefore, the current method works for a general damping parameter.

Next, we examine the performance of IMEX-RK on the choice of artificial damping parameter $\beta$. The 3-D results are recorded in Table~\ref{table12} and  Table~\ref{table13}. On the basis of these results, it is clear that IMEX-RK methods work for general artificial damping parameters. The 1-D results are similar and are not listed here. Therefore, we can set $\beta>1$ if $\alpha\ll 1$ and $\beta=\alpha$ if $\alpha\ge1$ numerically.

\begin{table}[htbp]
	\centering
	\small
	\renewcommand{\arraystretch}{1.2}
	\setlength{\tabcolsep}{0.7mm}{
		\caption{3-D errors of IMEX-RK2, with $h=500k$.}\label{table12}
		\begin{tabular}{|c|c|ccc|ccc|}
			\hline 
			\multirow{2}{*}{$\beta$} & \multirow{2}{*}{$k$} & \multicolumn{3}{c|}{$\alpha=0.001$} & \multicolumn{3}{c|}{$\alpha=0.01$} \\
			\cline { 3 - 8 } & & $L^{\infty}$ & $L^{2}$ & $H^{1}$ & $L^{\infty}$  & $L^{2}$ & $H^{1}$ \\
			
			\hline                & $1/1000$    & $3.59e-04$ & $4.26e-04$ & $5.44e-04$   & $3.59e-04$ & $4.26e-04 $& $5.44e-04$\\
			
			$1$  & $1/2000$    & $8.13e-05$ & $9.67e-05$ & $1.46e-04$   & $8.12e-05$ &$ 9.67e-05$ & $1.46e-04$\\
			
			& $1/4000$   & $2.02e-05$ &$ 2.40e-05$ & $3.68e-05$   &$ 2.02e-05 $& $2.40e-05$ & $3.68e-05$ \\\hline
			\multicolumn{2}{|c|}{\text { order }}
			& $2.0758 $              & $2.0749$&$1.9429$ &$ 2.0758 $&$ 2.0749$& $1.9429$ \\
			
			\hline                & $1/1000$    & $3.59e-04$ & $4.26e-04$ & $5.44e-04$   & $3.59e-04$ & $4.26e-04$ & $5.44e-04$\\
			
			$3$  & $1/2000$    & $8.13e-05$ & $9.67e-05$ &$ 1.46e-04 $  & $8.12e-05$ &$ 9.67e-05$ & $1.46e-04$\\

			&$ 1/4000$   & $2.02e-05$ & $2.40e-05$ & $3.68e-05 $  &$ 2.02e-05$ & $2.40e-05 $& $3.68e-05$ \\\hline
			\multicolumn{2}{|c|}{\text { order }}
			& $2.0758 $              & $2.0749$&$1.9429$ &$ 2.0758 $&$ 2.0749$& $1.9429$\\
			
			\hline                & $1/1000$    & $3.59e-04$ & $4.26e-04$ & $5.44e-04$   & $3.59e-04 $& $4.26e-04$ & $5.44e-04$\\
			
			$4$  & $1/2000$    &$ 8.13e-05$ & $9.67e-05$ & $1.46e-04$   & $8.12e-05$ & $9.67e-05$ &$ 1.46e-04$\\
			
			& $1/4000$   &$ 2.03e-05$ &$ 2.40e-05$ & $3.68e-05$   & $2.02e-05$ & $2.40e-05$ &$ 3.68e-05$ \\\hline
			\multicolumn{2}{|c|}{\text { order }}
			& $2.0722 $              & $2.0749$&$1.9429$ &$ 2.0758 $&$ 2.0749$& $1.9429$\\
			\hline
	\end{tabular}}
\end{table}

\begin{table}[htbp]
	\centering
	\small
	\renewcommand{\arraystretch}{1.2}
	\setlength{\tabcolsep}{0.7mm}{
		\caption{3-D error of IMEX-RK3, with $k =0.001 {{h}^{\frac{2}{3}}}$.}\label{table13}
		\begin{tabular}{|c|c|ccc|ccc|}
			\hline \multirow{2}{*}{$\beta$} & \multirow{2}{*}{$k$} & \multicolumn{3}{c|}{$\alpha=0.001$} & \multicolumn{3}{c|}{$\alpha=0.01$} \\
			\cline { 3 - 8 } & & $L^{\infty}$ & $L^{2}$ & $H^{1}$ & $L^{\infty}$ & $L^{2}$ & $H^{1}$ \\
			
			\hline                & $1/2080$ & $1.48e-04$ & $1.74e-04$ & $2.56e-04$& $1.48e-04$ & $1.74e-04 $&$ 2.56e-04$\\
			$1$ & $1/2520$ &$ 8.12e-05$ &$ 9.67e-05$ & $1.46e-04$& $8.12e-05$ & $9.67e-05 $& $1.46e-04$\\
			&$ 1/3302$ & $3.61e-05 $&$ 4.28e-05 $&$ 6.51e-05$&  $3.60e-05$ &$ 4.28e-05 $& $6.51e-05$\\\hline
			\multicolumn{2}{|c|}{\text { order }}
			& $3.0494$& $3.0335$ &$ 2.9644 $ & $3.0557$  & $ 3.0335$&$ 2.9644$ \\
			
			\hline                & $1/2080$ & $1.48e-04$ & $1.74e-04$ & $2.56e-04$& $1.48e-04$ & $1.74e-04 $& $2.56e-04$\\
			$3$ & $1/2520$ &$ 8.12e-05$ & $9.67e-05$ &$ 1.46e-04$& $8.12e-05$ & $9.67e-05$ & $1.46e-04$\\
			&$ 1/3302$ & $3.61e-05 $&$ 4.28e-05$ & $6.51e-05$&  $3.60e-05 $& $4.28e-05$ & $6.51e-05$\\\hline
			\multicolumn{2}{|c|}{\text { order }}
			& $3.0494$& $3.0335$ &$ 2.9644 $ & $3.0557$  & $ 3.0335$&$ 2.9644$  \\
			
			\hline                & $1/2080$ & $1.48e-04$ & $1.74e-04$ & $2.56e-04$& $1.48e-04 $&$ 1.74e-04 $& $2.56e-04$\\
			$4$ & $1/2520 $& $8.12e-05$ & $9.67e-05 $& $1.46e-04$& $8.12e-05$ & $9.67e-05$ & $1.46e-04$\\
			&$ 1/3302$ & $3.61e-05$ &$ 4.28e-05$ &$ 6.51e-05$&  $3.60e-05 $&$ 4.28e-05$ & $6.51e-05$\\\hline
			\multicolumn{2}{|c|}{\text { order }}
			& $3.0494$& $3.0335$ &$ 2.9644 $ & $3.0557$  & $ 3.0335$&$ 2.9644$  \\
			\hline
	\end{tabular}}
\end{table}
\vskip2mm
\begin{remark}
	We have demonstrated the high-order-in-time accuracy of IMEX-RK, including second-order and third-order ones. It is worth noting that even higher-order-in-time IMEX-RK methods for the LL equation can be designed; see the $s$-level method in \cite{gear1965hybrid} for a general purpose. This method does not necessarily have the $s$-order accuracy when $s>4$ in \cite{gear1965hybrid}. For instance, the $s$-level method has at most $(s-1)$-order accuracy when $s=5,6,7$.
\end{remark}

\section{Micromagnetics simulations}\label{section:micromagnetics simulations}

In this section, we apply IMEX-RK2 and IMEX-RK3 to conduct micromagnetics simulations, including different equilibrium structures and a benchmark problem from NIST \cite{NIST2000website} to examine the performance of our proposed methods in the real-world applications.

\subsection{Equilibrium states}
We use a spatial resolution $64\times 128\times 1$ on a $1\times 2\times 0.02$ $\mu {{m}^{3}}$ thin-film element with $\alpha =0.1$, a temporal step size $k=1$ picosecond (ps), and set $\beta =3$ in IMEX-RK methods. In the absence of an external field, multiple metastable states are often observed in ferromagnets, both experimentally and numerically \cite{zheng1997switching,schrefl2003finite}.

Starting with three different initial magnetization distributions, we obtain three equilibrium states in Fig.~\ref{metastable states}, including Landau state, C-state, and S-state. The arrow denotes the first two components of the magnetization vector and the color denotes the angle between them. It is clearly that IMEX-RK2,3 produce consistent results.
\begin{figure}[htbp]
	\flushright
	\subfloat{\label{fig:a}\includegraphics[width=1.7in]{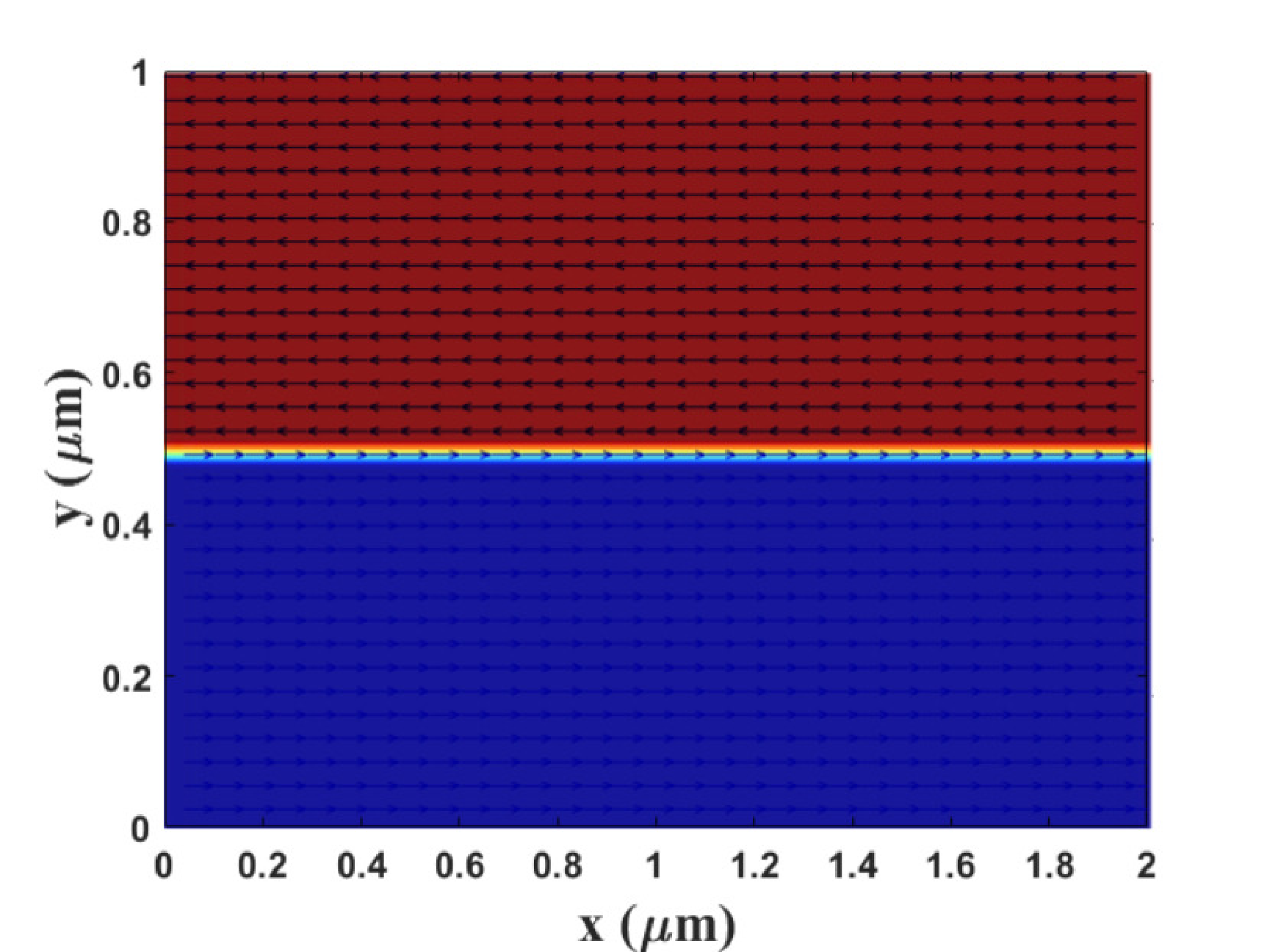}}
	\subfloat{\label{fig:b}\includegraphics[width=1.7in]{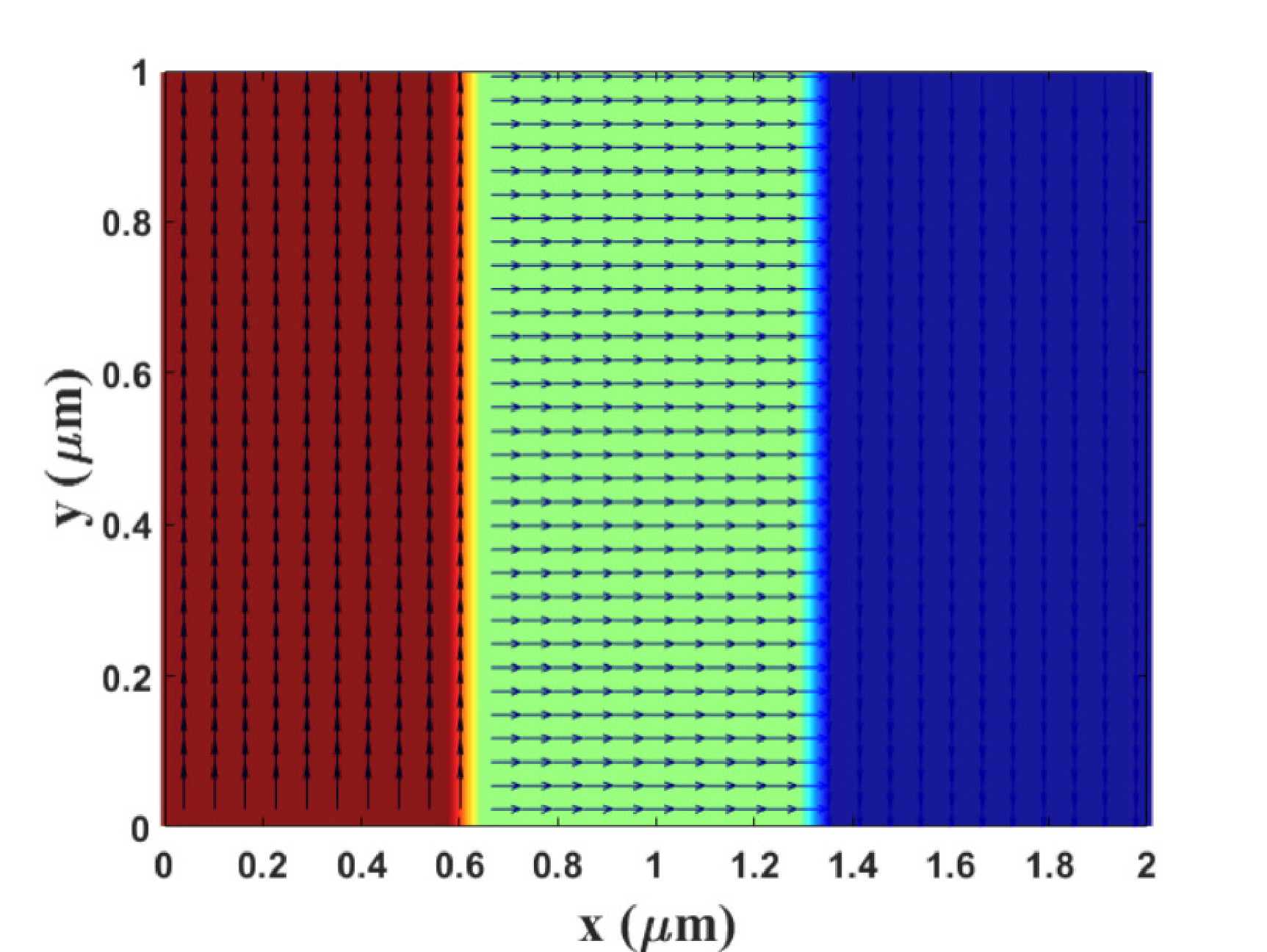}}
	\subfloat{\label{fig:c}\includegraphics[width=1.7in]{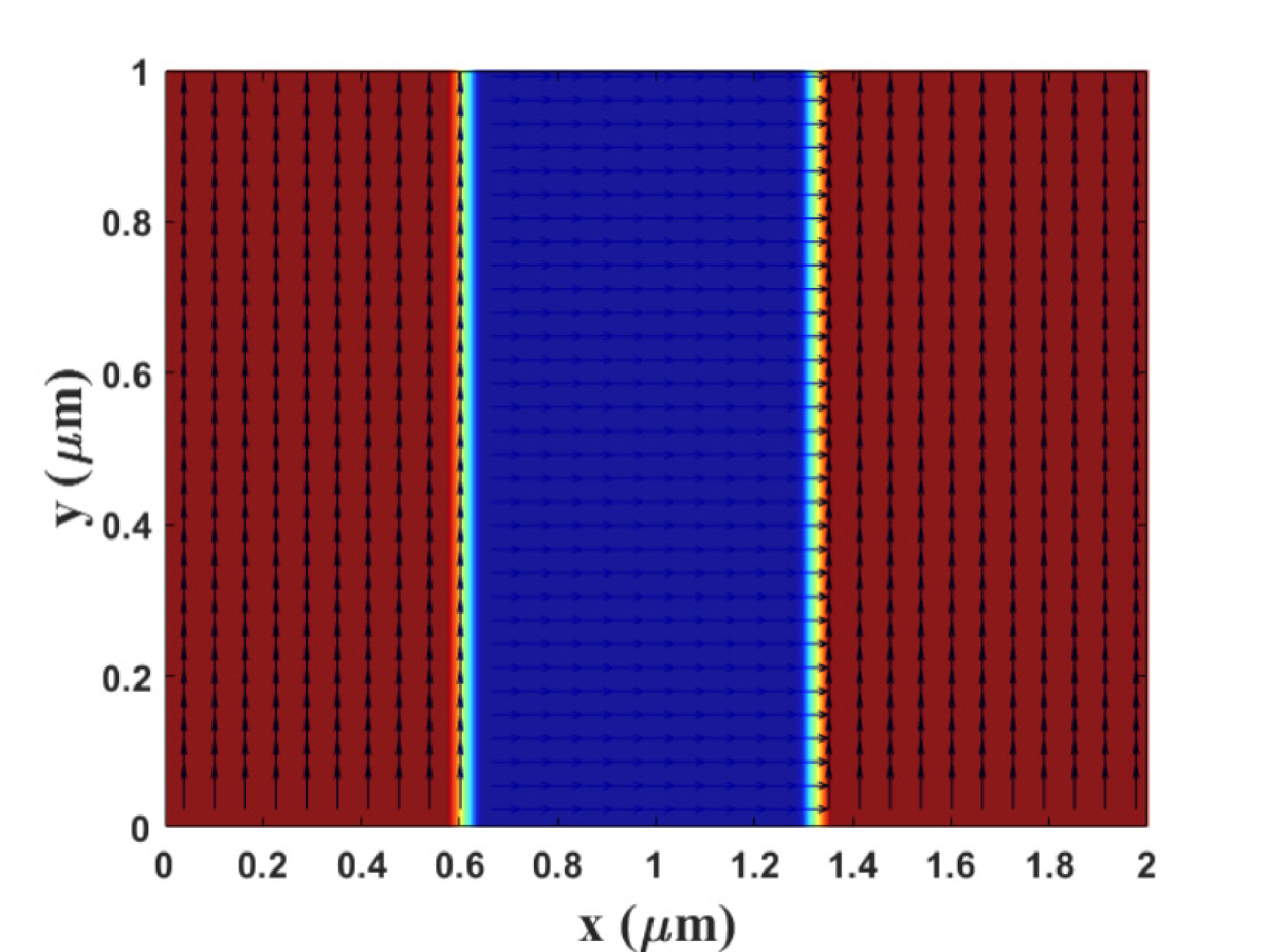}}
	\quad
	\subfloat{\label{fig:d}\includegraphics[width=1.7in]{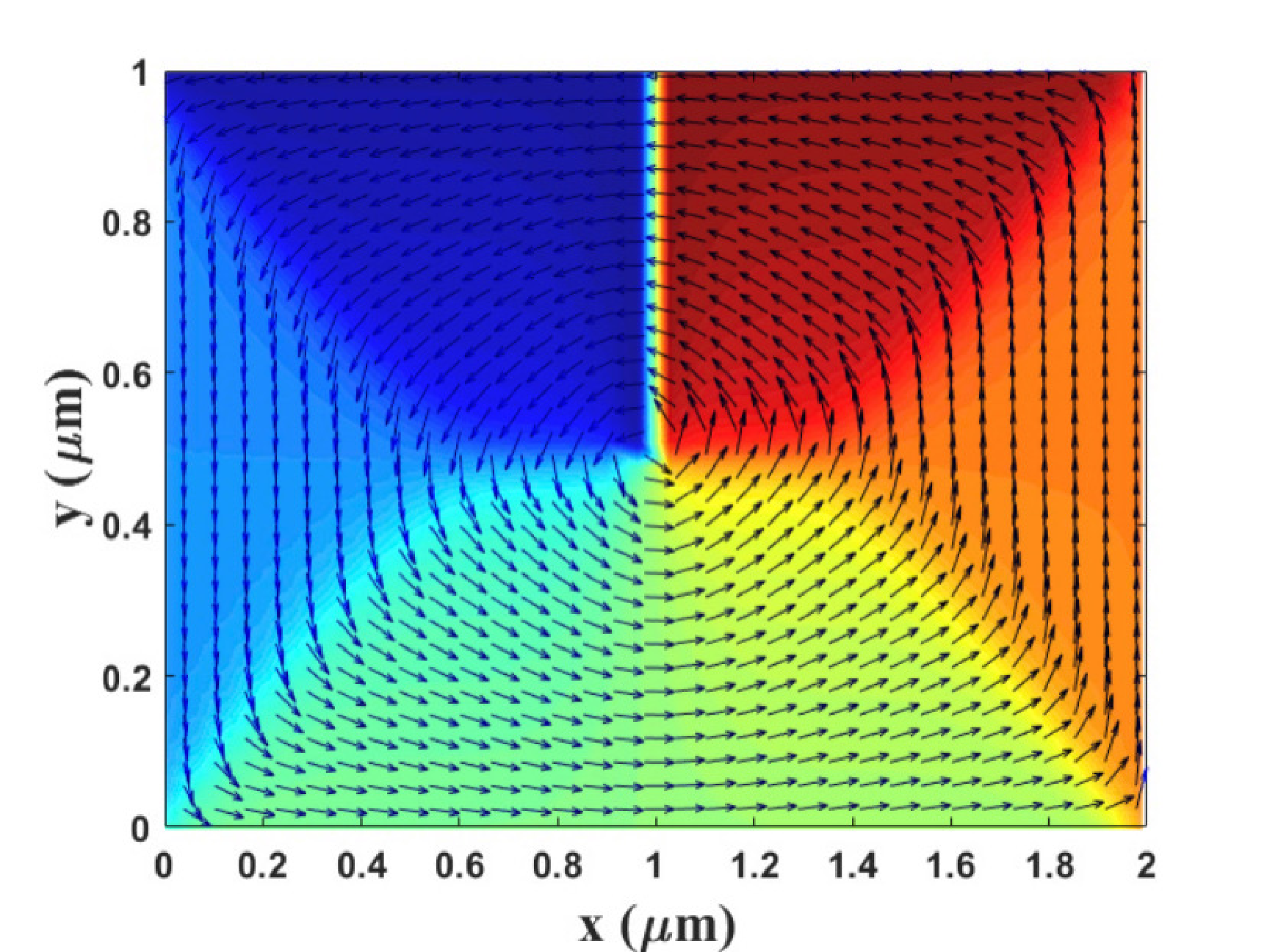}}
	\subfloat{\label{fig:e}\includegraphics[width=1.7in]{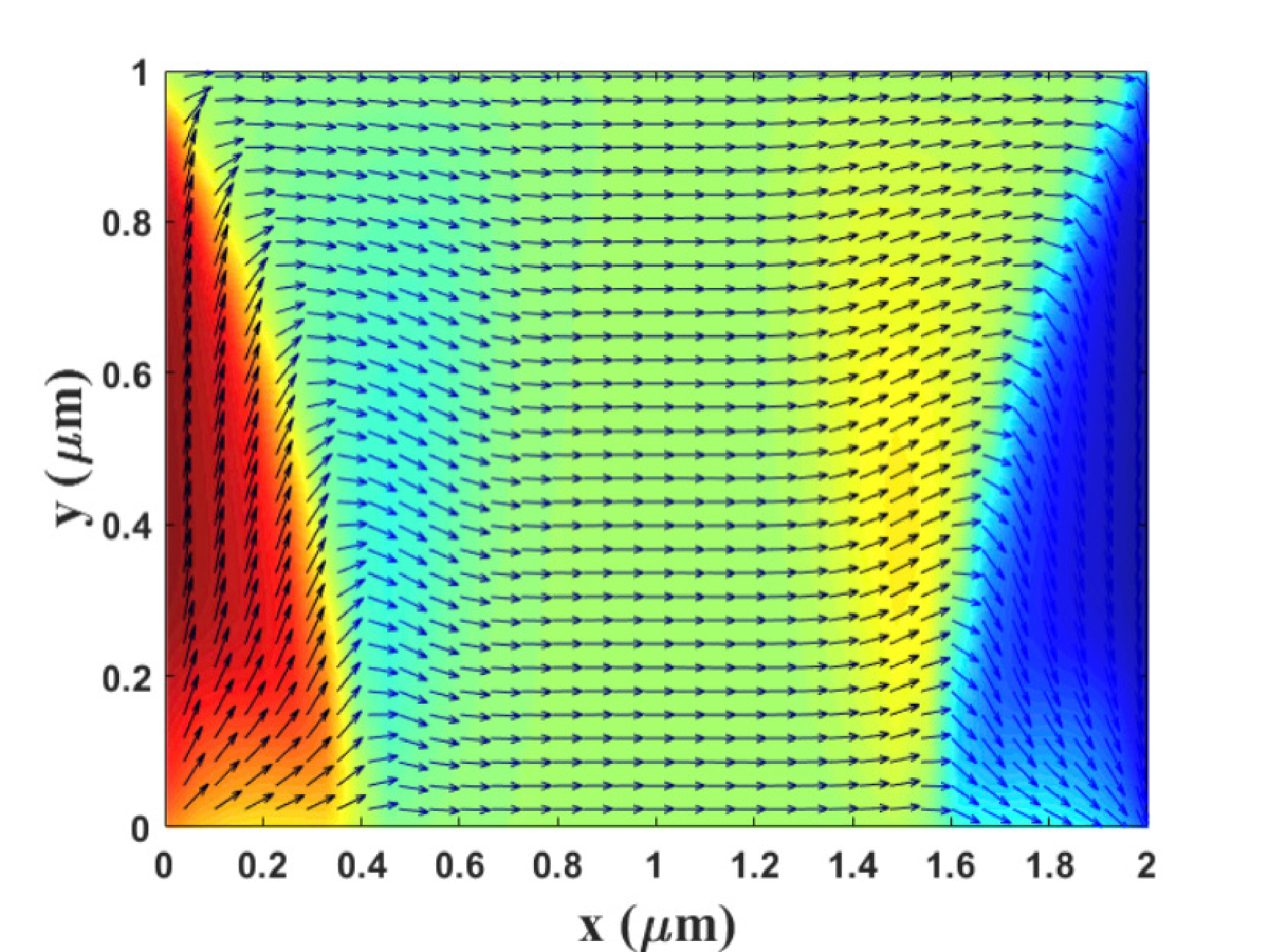}}
	\subfloat{\label{fig:f}\includegraphics[width=1.7in]{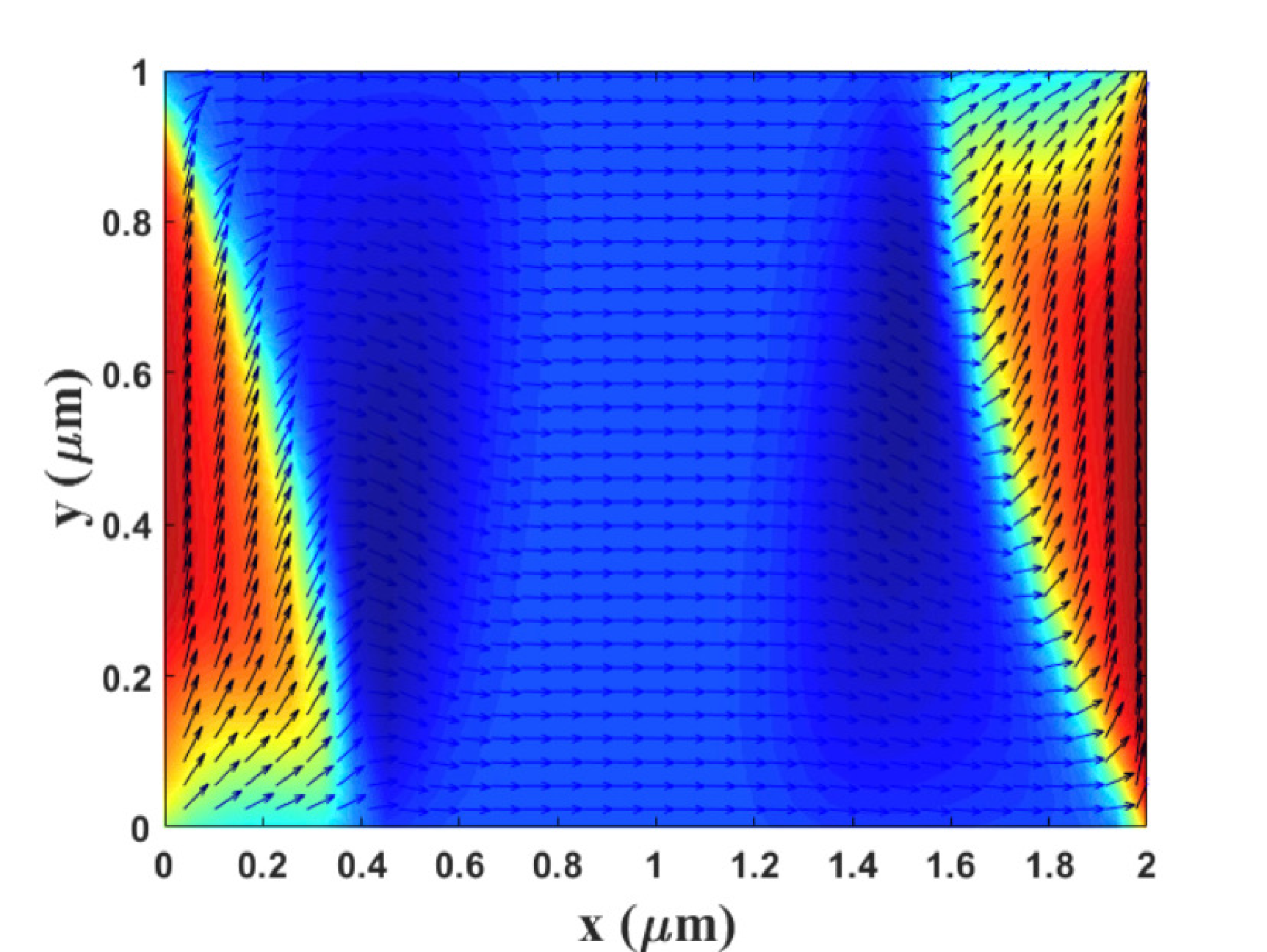}}
	\quad
	\subfloat{\label{fig:g}\includegraphics[width=1.7in]{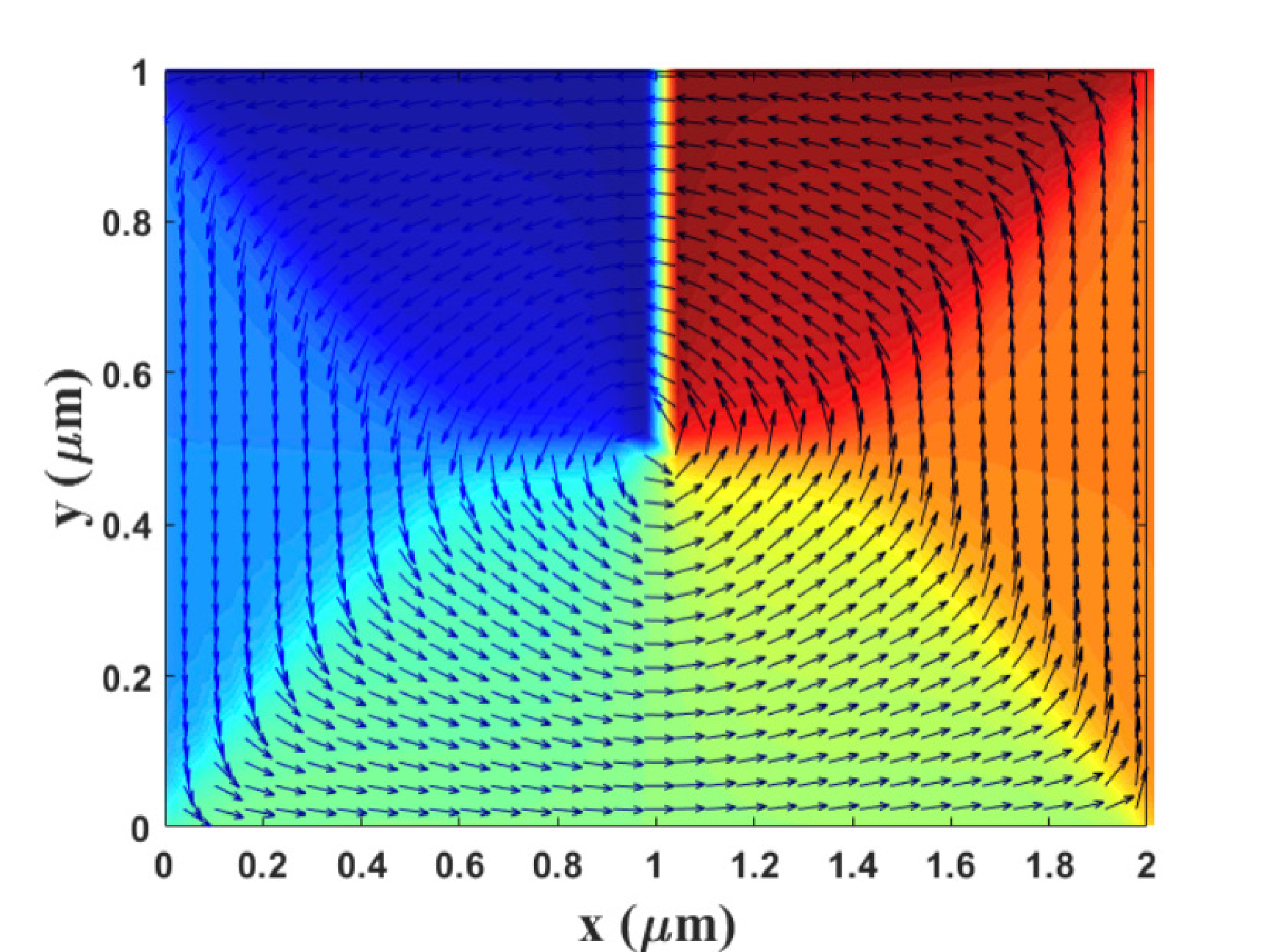}}
	\subfloat{\label{fig:h}\includegraphics[width=1.7in]{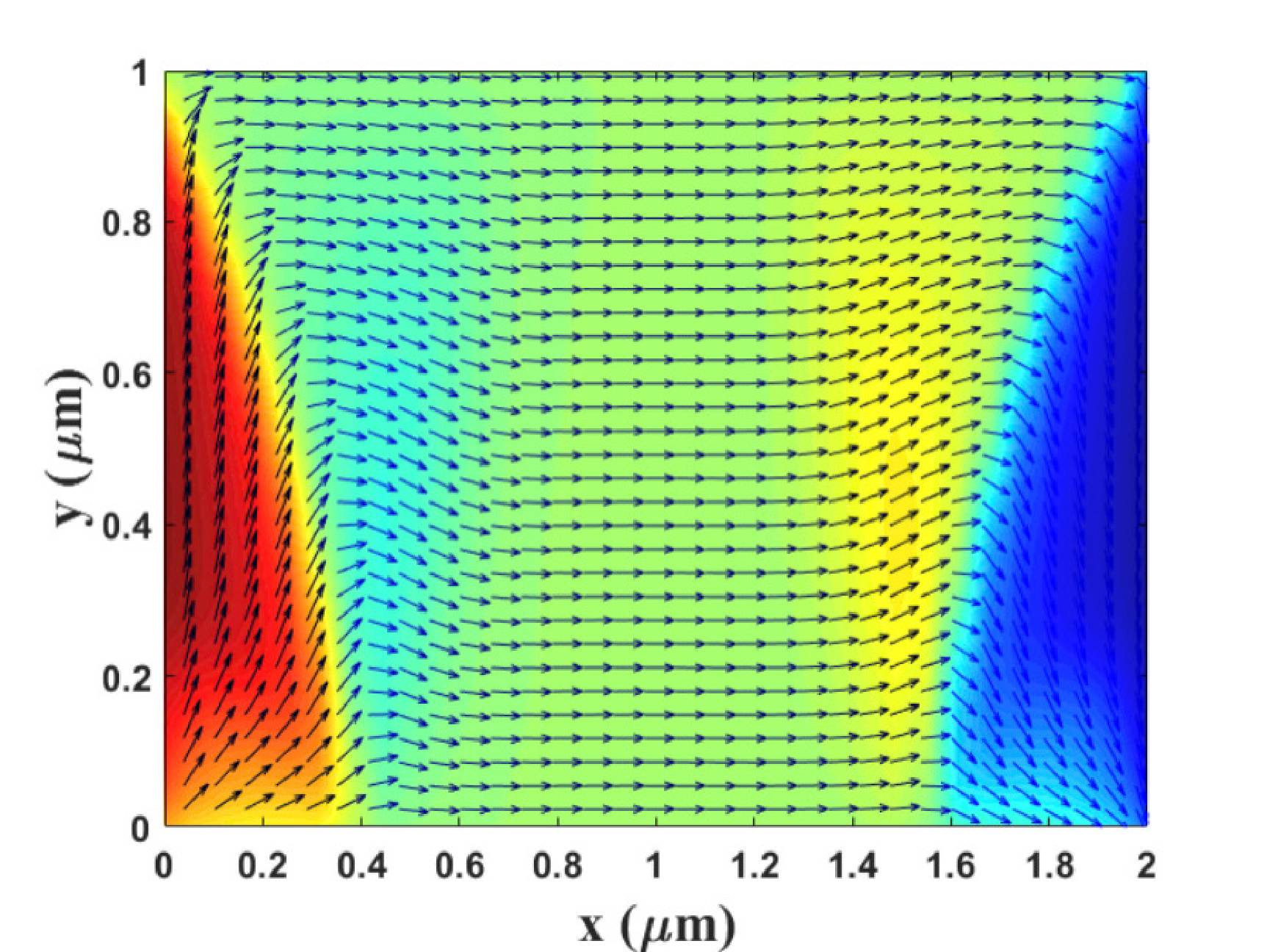}}
	\subfloat{\label{fig:i}\includegraphics[width=1.7in]{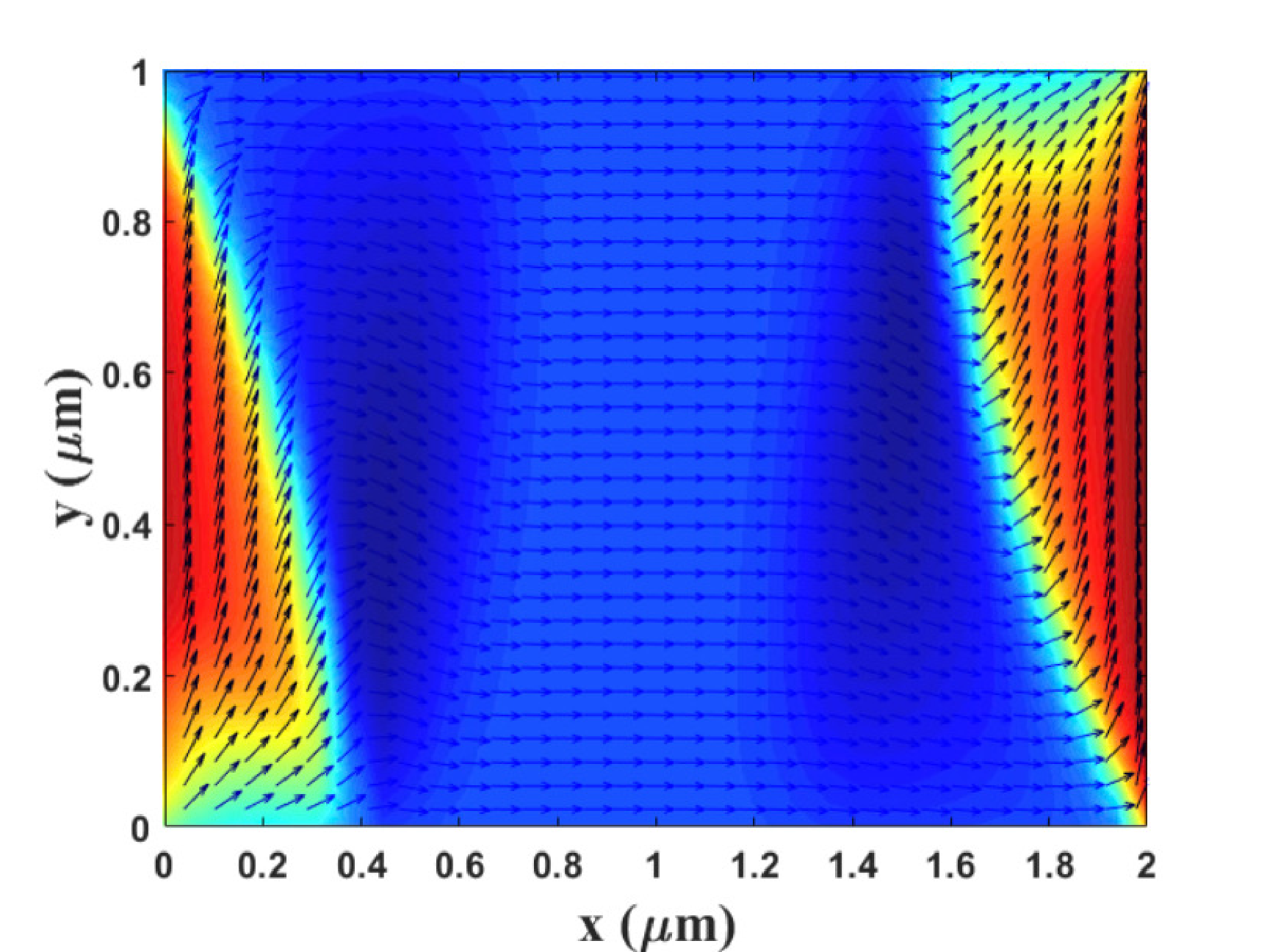}}
	\caption{Equilibrium states simulated by IMEX-RK methods. Top: Initial states; Middle: Equilibrium states by IMEX-RK2; Bottom: Equilibrium states by IMEX-RK3.
		Left: Landau state; Middle: C-state; Right: S-state.}
	\label{metastable states}
\end{figure}

\subsection{Benchmark peoblem from NIST}
To investigate the dynamical performance of IMEX-RK methods, we simulate a benchmark problem proposed by the Micromagnetic Modeling Activity Group (muMag) from NIST. A positive external field of strength ${{H}_{0}}={{\mu }_{0}}{{H}_{e}}$ in the unit of $mT$ is applied. The magnetization is able to reach a steady state. Once this state is reached, the applied external field is reduced by a certain amount and the material sample is allowed to reach another steady state again. Repeat the process until the hysteresis system attains a negative external field of strength ${{H}_{0}}$. This process is then implemented in reverse, increasing the field in small steps until the initial applied external field is reached. Afterward, we can plot the average magnetization at steady states as a function of the external field strength during the hysteresis loop. The stopping criterion for a steady state is that the relative change of the total energy is less than ${{10}^{-9}}$. For comparison with the available code $mo96a$ of the first standard problem from NIST, we set $100\times 50\times 1$ spatial resolution with mesh size $20\times 20\times 20$ $\rm n{{m}^{3}}$ and the canting angle $+{{1}^{{}^\circ }}$ of applied field from the nominal axis. The initial state is uniform and $\left[ \rm-50mT,+50mT \right]$ is split into 200 steps for both $x$-loop and $y$-loop.
\begin{figure}[htbp]
	\centering
	\subfloat[\small $H_0 // \textit{y-axis}$,~mo96a]{\label{a mo96a}\includegraphics[width=2.5in]{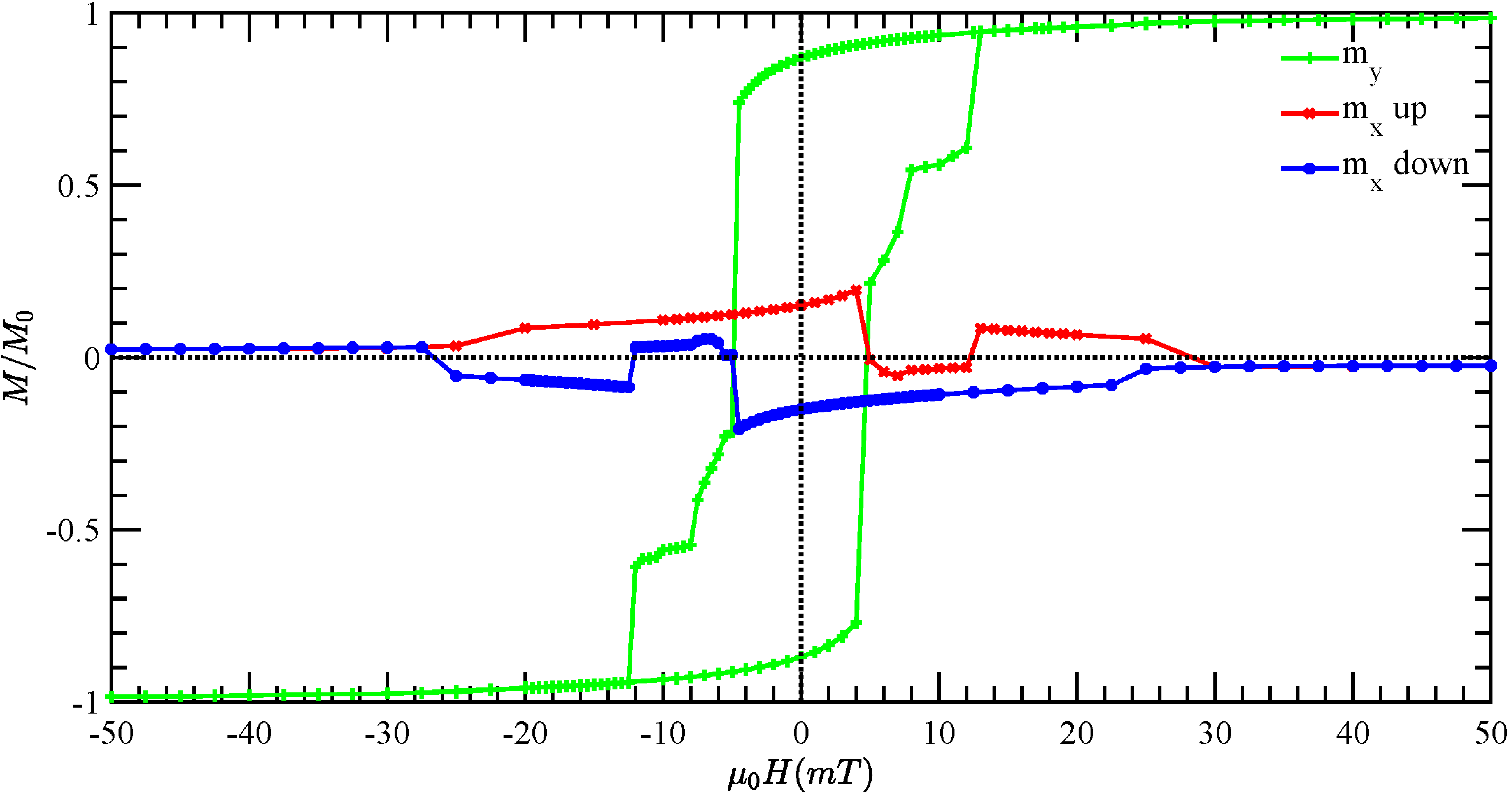}}
	\subfloat[\small $H_0 // \textit{x-axis}$,~mo96a]{\label{b mo96a}\includegraphics[width=2.5in]{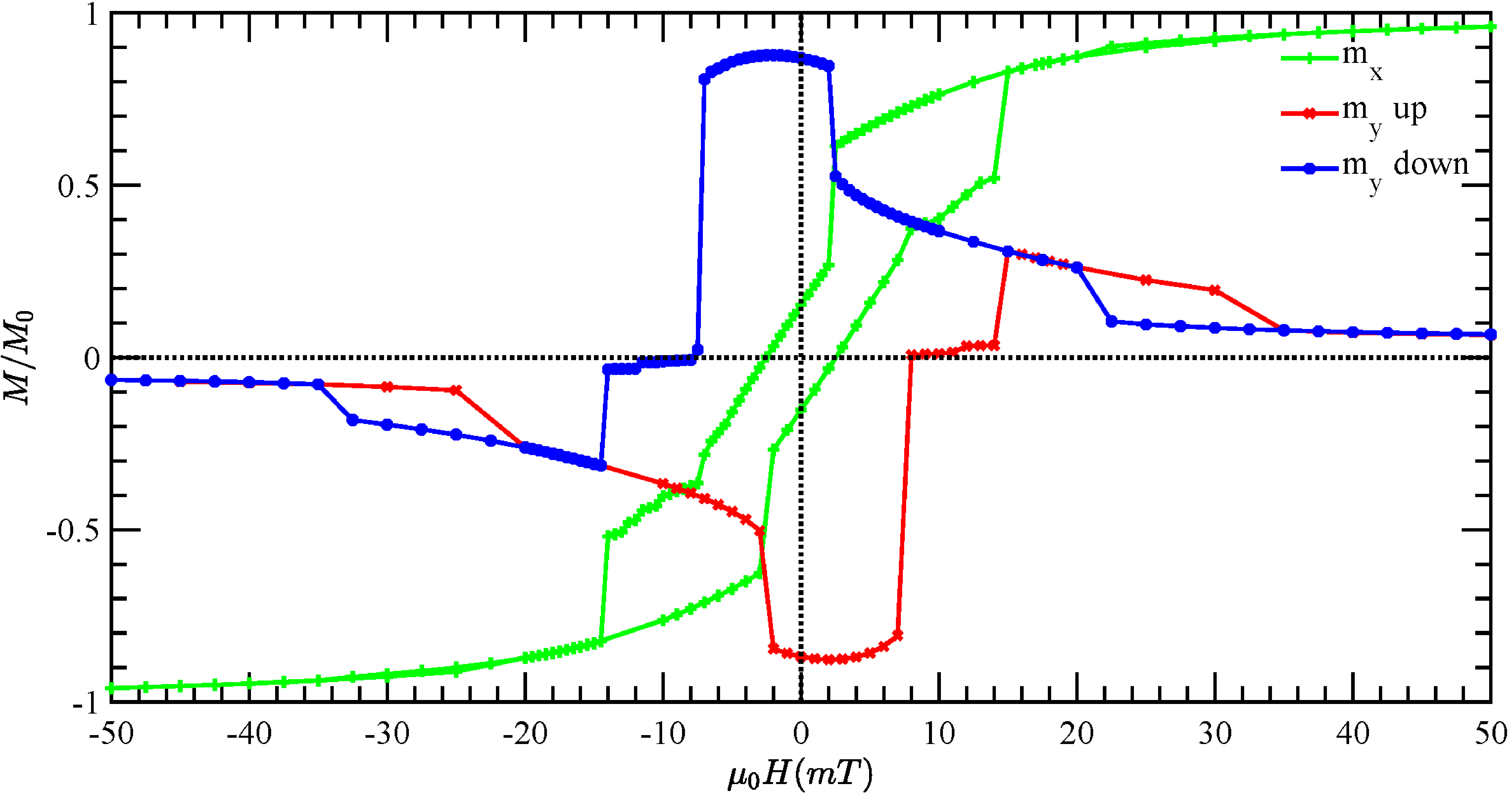}}
	\quad
	\subfloat[\small $H_0 // \textit{y-axis}$,~IMEX-RK2 ]{\label{c rk2}\includegraphics[width=2.5in]{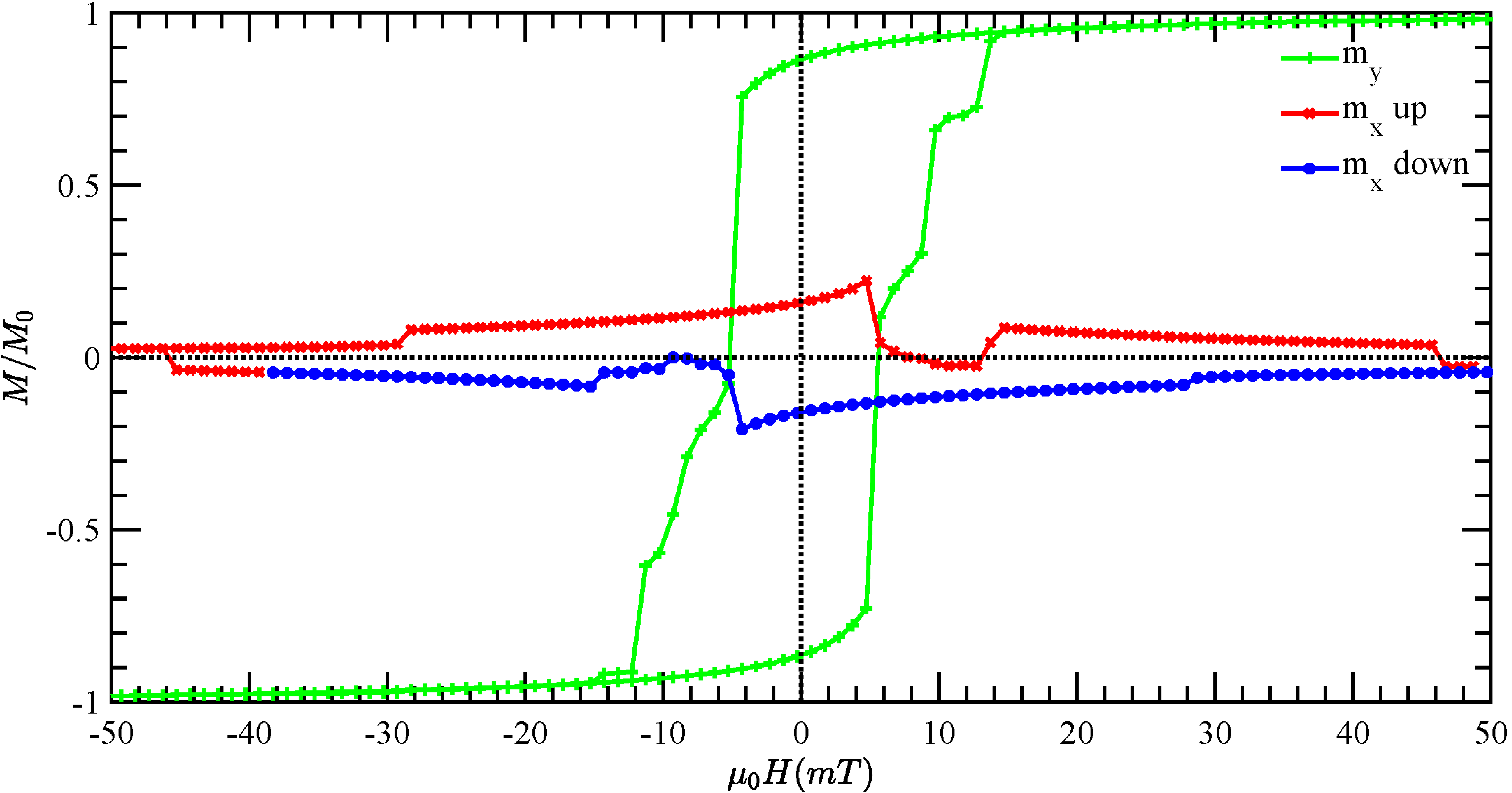}}
	\subfloat[\small $H_0 // \textit{x-axis}$,~IMEX-RK2 ]{\label{d rk2}\includegraphics[width=2.5in]{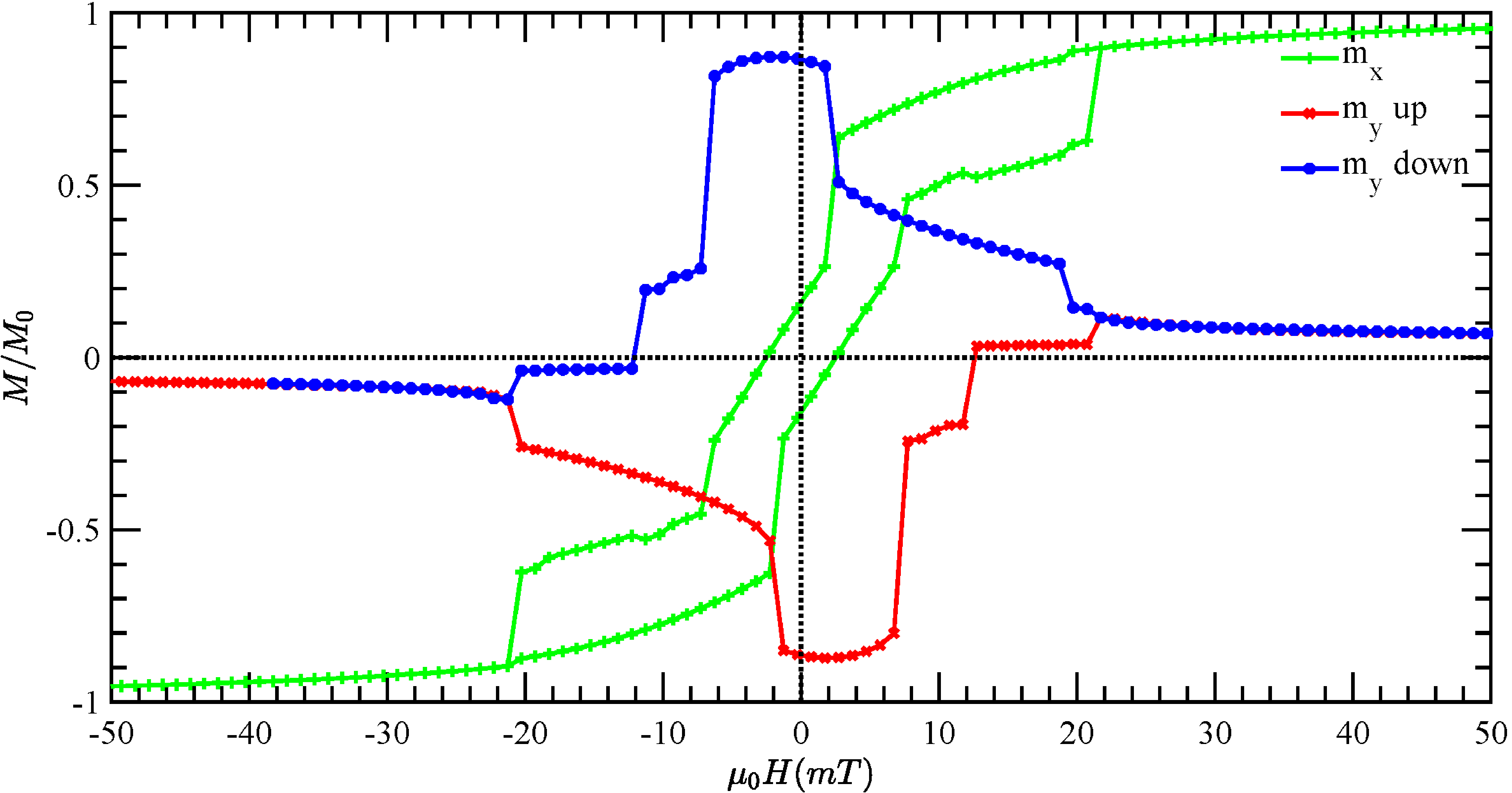}}
	\quad
	\subfloat[\small $H_0 // \textit{y-axis}$,~IMEX-RK3 ]{\label{c rk3}\includegraphics[width=2.5in]{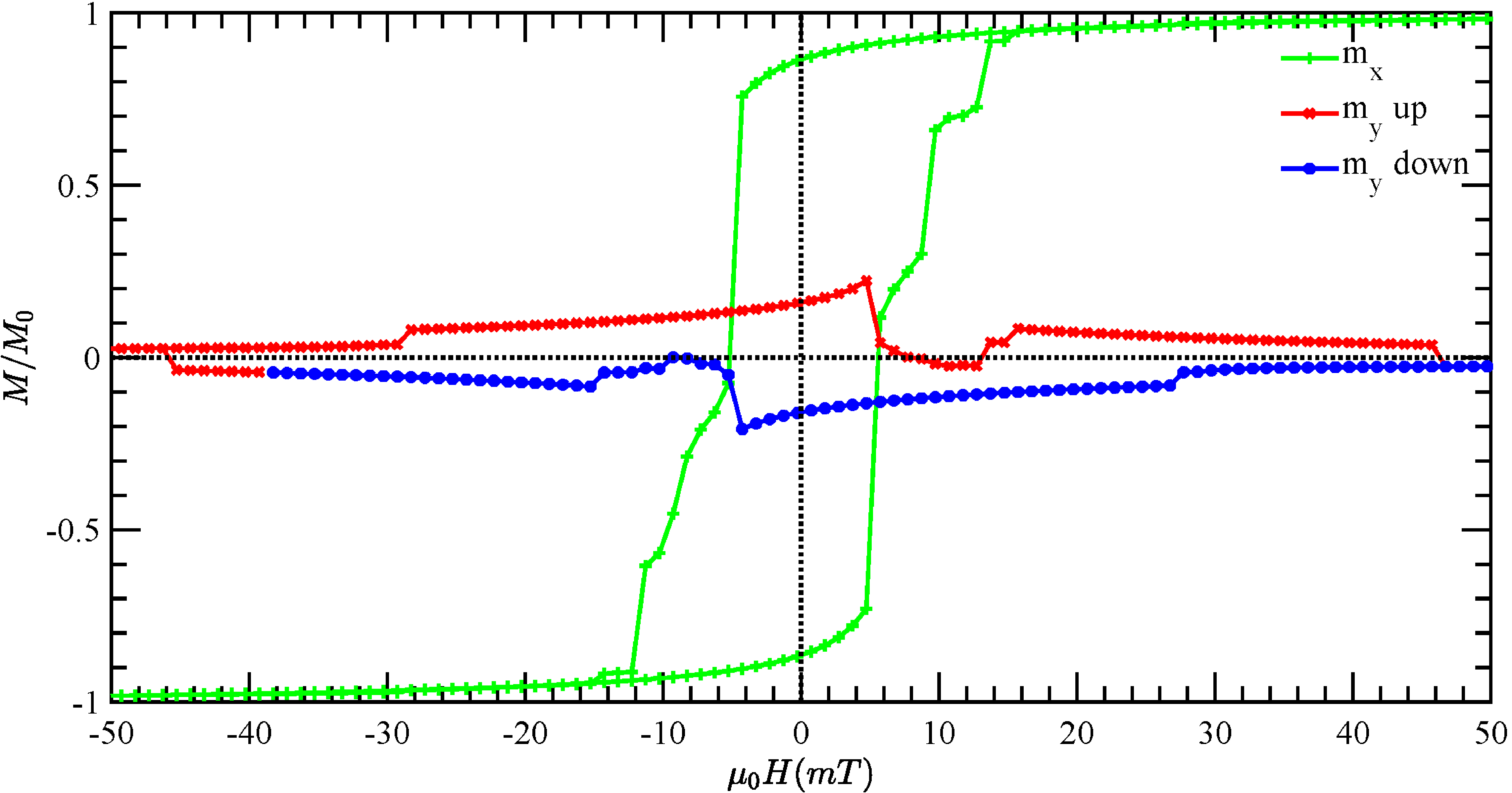}}
	\subfloat[\small $H_0 // \textit{x-axis}$,~IMEX-RK3 ]{\label{d rk3}\includegraphics[width=2.5in]{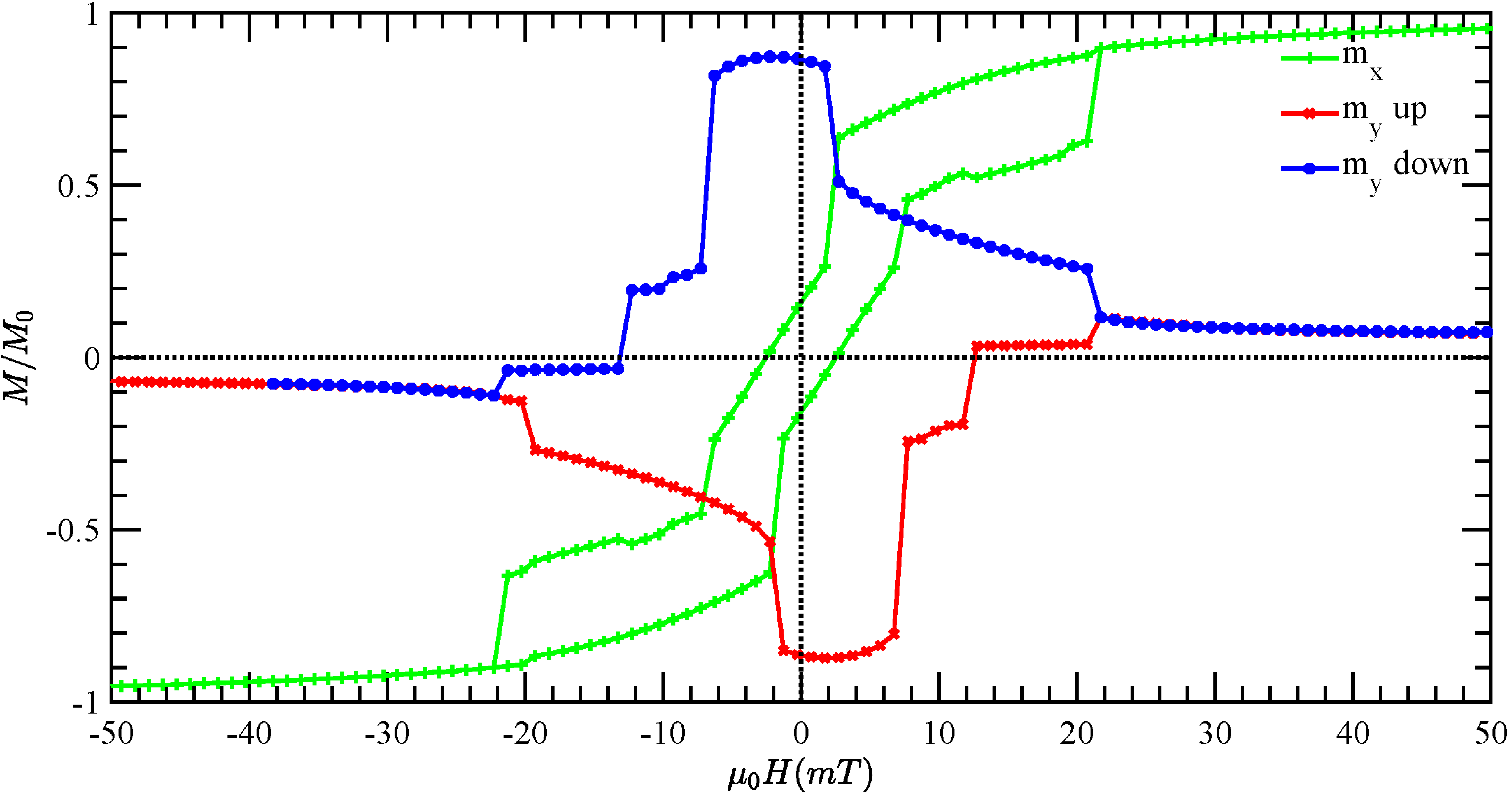}}
	\caption{ Hysteresis loops with $\alpha =0.1,\beta=3$, and the mesh size
		$20\times 20\times 20$ $n{{m}^{3}}$. The applied field is approximately parallel
		(canting angle $+{{1}^{{}^\circ }}$) to the $y$-axis (left column) and the $x$-axis
		(right column). Top:~mo96a; Middle:~IMEX-RK2; Bottom:~IMEX-RK3.}
	\label{Benchmark problem}
\end{figure}

Hysteresis loops generated by $mo96a$ are displayed in Fig.~\ref{a mo96a} and Fig.~\ref{b mo96a} when the applied field is approximately parallel to the $y$-(long) axis and the $x$-(short) axis. The average remanent magnetization in reduced units is given by $(-1.5120\times {{10}^{-1}},8.6964\times {{10}^{-1}},0)$ for the $y$-loop and $(1.5257\times {{10}^{-1}},8.6870\times {{10}^{-1}},0)$ for the $x$-loop. The coercive fields are 4.8871 $\rm mT$ in Fig.~\ref{a mo96a} and 2.5253 $\rm mT$ in Fig.~\ref{b mo96a}, respectively. Hysteresis loops generated by IMEX-RK2 method are presented in Fig.~\ref{c rk2} and Fig.~\ref{d rk2} when the applied field is approximately parallel to the long axis and the short axis, respectively. The average remanent magnetization in reduced units is $(-1.6099\times {{10}^{-1}},8.6096\times {{10}^{-1}},4.2423\times {{10}^{-7}})$ for the $y$-loop and $(-1.4274\times {{10}^{-1}},8.6656\times {{10}^{-1}},1.5753\times {{10}^{-7}})$ for the $x$-loop. The coercive fields are $ 5.4688~(\pm0.7)~\rm mT$ in Fig.~\ref{c rk2}  and $ 2.7188~(\pm0.2)~\rm mT$ in Fig.~\ref{d rk2}. Similarly, the IMEX-RK3 results are presented in the bottom row of Fig.~\ref{Benchmark problem}. The average remanent magnetization in reduced units is $(-1.6107\times {{10}^{-1}},8.6102\times {{10}^{-1}},9.5492\times {{10}^{-8}})$ for the $y$-loop and $(-1.4292\times {{10}^{-1}},8.6659\times {{10}^{-1}},4.5836\times {{10}^{-9}})$ for the $x$-loop. The coercive fields are $ 5.4688~(\pm0.7)~\rm mT$ in Fig.~\ref{c rk3} and $ 2.7188~(\pm0.2)~\rm mT$ in Fig.~\ref{d rk3}. Based on these results, we conclude that IMEX-RK methods work well for the benchmark problems from NIST, both qualitatively and quantitatively.

\section{Numerical stability and convergence analysis for the proposed SSP-IMEX-RK2 scheme} \label{sec: convergence}

A theoretical analysis for the proposed IMEX-RK numerical schemes turns out to be highly challenging, due to the multi-stage nature, as well as the highly complicated nonlinear terms in the vector form. For simplicity, we focus on the SSP-IMEX-RK2 numerical algorithm~\eqref{SSP-IMEX-RK2}. In the first step, a numerical stability is stablished for the linear part, i.e., in the simple case with only linear diffusion part (the term $\beta \Delta \boldsymbol m$) taken into consideration. Afterward, we provide a convergence analysis of the SSP-IMEX-RK2 scheme~\eqref{SSP-IMEX-RK2} for a simplified nonlinear model of LL equation, in which only the damping term is considered, while the gyromagnetic term is skipped.

\subsection{Linear stability estimate for the SSP-IMEX-RK2 scheme~\eqref{SSP-IMEX-RK2}} In the simple case with only linear diffusion term is considered, we denote $L_h = \beta \Delta_h$. The SSP-IMEX-RK2 scheme~\eqref{SSP-IMEX-RK2} is simplified as
\begin{equation}\label{SSP-IMEX-RK2-Linear}
	\left\{\begin{array}{l}
		\boldsymbol{\tilde{m}}_{1}=\boldsymbol{{m}}_{n} \\
		{\boldsymbol{\tilde{m}}_{2}}=\boldsymbol{{\tilde{m}}}_{1} +\frac{k}{4} L_h ({\boldsymbol{\tilde{m}}_{2}}) \\
		{\boldsymbol{\tilde{m}}_{3}}=\boldsymbol{{\tilde{m}}}_{1} +\frac{k}{4} L_h ({\boldsymbol{\tilde{m}}_{3}}) \\
		{\boldsymbol{\tilde{m}}_{4}}=\boldsymbol{{\tilde{m}}}_{1} + \frac{k}{3} \left( L_h ({\boldsymbol{\tilde{m}}_{2}}) + L_h({\boldsymbol{\tilde{m}}_{3}}) + L_h({\boldsymbol{\tilde{m}}_{4}}) \right) \\
		\boldsymbol{{m}}_{n+1}=\boldsymbol{{\tilde{m}}}_{1} + \frac{k}{3} \left( L_h ({\boldsymbol{\tilde{m}}_{2}}) + L_h ({\boldsymbol{\tilde{m}}_{3}}) + L_h ({\boldsymbol{\tilde{m}}_{4}}) \right)
	\end{array}\right. .
\end{equation}
For the convenience of the stability analysis, numerical system could be rewritten as
\begin{align}
	&
	\frac{\tilde{\boldsymbol m}_2 - \boldsymbol m_n}{k} = \frac{\beta}{4} \Delta_h \tilde{\boldsymbol m}_2 ,   \label{SSP-IMEX-RK2-Linear-1-1}
	\\
	&
	\frac{\tilde{\boldsymbol m}_3 - \boldsymbol m_n}{k} = \frac{\beta}{4} \Delta_h \tilde{\boldsymbol m}_3 ,   \label{SSP-IMEX-RK2-Linear-1-2}
	\\
	&
	\frac{\tilde{\boldsymbol m}_4 - \boldsymbol m_n}{k} = \frac{\beta}{3}  \Delta_h ( \tilde{\boldsymbol m}_2  + \tilde{\boldsymbol m}_3
	+  \tilde{\boldsymbol m}_4 ) ,   \label{SSP-IMEX-RK2-Linear-1-3}
	\\
	& \boldsymbol m_{n+1} = \tilde{\boldsymbol m}_4 . \label{SSP-IMEX-RK2-Linear-1-4}
\end{align}
Moreover, to reveal the numerical stability of this Runge-Kutta style algorithm, we subtract~\eqref{SSP-IMEX-RK2-Linear-1-1} from \eqref{SSP-IMEX-RK2-Linear-1-2}, \eqref{SSP-IMEX-RK2-Linear-1-2} from \eqref{SSP-IMEX-RK2-Linear-1-3}, and arrive at the following equivalent numerical system:
\begin{align}
	&
	\frac{\tilde{\boldsymbol m}_2 - \boldsymbol m_n}{k} = \frac{\beta}{4} \Delta_h \tilde{\boldsymbol m}_2 ,   \label{SSP-IMEX-RK2-Linear-2-1}
	\\
	&
	\frac{\tilde{\boldsymbol m}_3 - \tilde{\boldsymbol m}_2}{k} = \frac{\beta}{4} \Delta_h ( \tilde{\boldsymbol m}_3 - \tilde{\boldsymbol m}_2 ) ,   \label{SSP-IMEX-RK2-Linear-2-2}
	\\
	&
	\frac{\tilde{\boldsymbol m}_4 - \tilde{\boldsymbol m}_3}{k} = \beta  \Delta_h ( \frac13 \tilde{\boldsymbol m}_2  + \frac{1}{12} \tilde{\boldsymbol m}_3
	+  \frac13 \tilde{\boldsymbol m}_4 ) ,   \label{SSP-IMEX-RK2-Linear-2-3}
	\\
	& \boldsymbol m_{n+1} = \tilde{\boldsymbol m}_4 . \label{SSP-IMEX-RK2-Linear-2-4}
\end{align}

Taking a discrete inner product with~\eqref{SSP-IMEX-RK2-Linear-2-1} by $2 \tilde{\boldsymbol m}_2$ gives
\begin{equation}
	\| \tilde{\boldsymbol m}_2 \|_2^2 - \| \boldsymbol m_n \|_2^2 + \| \tilde{\boldsymbol m}_2 - \boldsymbol m_n \|_2^2
	+ \frac{\beta}{2} k \| \nabla_h \tilde{\boldsymbol m}_2 \|_2^2 = 0 ,  \label{stability-1}
\end{equation}
in which the summation-by-parts formula~\eqref{sum1} has been applied. Similarly, taking a discrete inner product with~\eqref{SSP-IMEX-RK2-Linear-2-2} by $2 \tilde{\boldsymbol m}_3$,  with~\eqref{SSP-IMEX-RK2-Linear-2-3} by $2 \tilde{\boldsymbol m}_4$, leads to
\begin{align}
	&
	\| \tilde{\boldsymbol m}_3 \|_2^2 - \| \tilde{\boldsymbol m}_2 \|_2^2 + \| \tilde{\boldsymbol m}_3 - \tilde{\boldsymbol m}_2 \|_2^2
	+ \frac{\beta}{2} k \| \nabla_h \tilde{\boldsymbol m}_3 \|_2^2
	= \frac{\beta}{2} k \langle \nabla_h \tilde{\boldsymbol m}_2 ,  \nabla_h \tilde{\boldsymbol m}_3 \rangle ,  \label{stability-2}
	\\
	&
	\| \tilde{\boldsymbol m}_4 \|_2^2 - \| \tilde{\boldsymbol m}_3 \|_2^2 + \| \tilde{\boldsymbol m}_4 - \tilde{\boldsymbol m}_3 \|_2^2
	+ \frac{2 \beta}{3} k \| \nabla_h \tilde{\boldsymbol m}_4 \|_2^2   \nonumber
	\\
	&  \qquad \qquad
	= - \frac{2 \beta}{3} k \langle \nabla_h \tilde{\boldsymbol m}_2 ,  \nabla_h \tilde{\boldsymbol m}_4 \rangle
	- \frac{\beta}{6} k \langle \nabla_h \tilde{\boldsymbol m}_3 ,  \nabla_h \tilde{\boldsymbol m}_4 \rangle  .  \label{stability-3}
\end{align}
In turn, a combination of~\eqref{stability-1} and \eqref{stability-3} indicates that
\begin{align}
	&
	\| \tilde{\boldsymbol m}_4 \|_2^2 - \| \boldsymbol m_n \|_2^2 + \| \tilde{\boldsymbol m}_2 - \boldsymbol m_n \|_2^2
	+ \| \tilde{\boldsymbol m}_3 - \tilde{\boldsymbol m}_2 \|_2^2
	+ \| \tilde{\boldsymbol m}_4 - \tilde{\boldsymbol m}_3 \|_2^2   \nonumber
	\\
	& \qquad
	+ \frac{\beta}{2} k \| \nabla_h \tilde{\boldsymbol m}_2 \|_2^2
	+ \frac{\beta}{2} k \| \nabla_h \tilde{\boldsymbol m}_3 \|_2^2
	+ \frac{2 \beta}{3} k \| \nabla_h \tilde{\boldsymbol m}_4 \|_2^2   \nonumber
	\\
	&  \qquad
	= \frac{\beta}{2} k \langle \nabla_h \tilde{\boldsymbol m}_2 ,  \nabla_h \tilde{\boldsymbol m}_3 \rangle
	- \frac{2 \beta}{3} k \langle \nabla_h \tilde{\boldsymbol m}_2 ,  \nabla_h \tilde{\boldsymbol m}_4 \rangle
	- \frac{\beta}{6} k \langle \nabla_h \tilde{\boldsymbol m}_3 ,  \nabla_h \tilde{\boldsymbol m}_4 \rangle  .  \label{stability-4}
\end{align}
Meanwhile, a careful application of Cauchy inequality implies that
\begin{align}
	&
	\frac12 \langle \nabla_h \tilde{\boldsymbol m}_2 ,  \nabla_h \tilde{\boldsymbol m}_3 \rangle
	\le \frac14 ( \| \nabla_h \tilde{\boldsymbol m}_2 \|_2^2 + \|  \nabla_h \tilde{\boldsymbol m}_3  \|_2^2 ) ,
	\label{stability-5-1}
	\\
	&
	- \frac23 \langle \nabla_h \tilde{\boldsymbol m}_2 ,  \nabla_h \tilde{\boldsymbol m}_4 \rangle
	\le \frac29 \| \nabla_h \tilde{\boldsymbol m}_2 \|_2^2 + \frac12 \|  \nabla_h \tilde{\boldsymbol m}_4  \|_2^2 ,
	\label{stability-5-2}
	\\
	&
	- \frac16 \langle \nabla_h \tilde{\boldsymbol m}_3 ,  \nabla_h \tilde{\boldsymbol m}_4 \rangle
	\le \frac{1}{12} \| \nabla_h \tilde{\boldsymbol m}_3 \|_2^2 + \frac{1}{12} \|  \nabla_h \tilde{\boldsymbol m}_4  \|_2^2 .
	\label{stability-5-3}
\end{align}
Therefore, a substitution of these estimates into~\eqref{stability-4} yields
\begin{align}
	&
	\| \tilde{\boldsymbol m}_4 \|_2^2 - \| \boldsymbol m_n \|_2^2 + \| \tilde{\boldsymbol m}_2 - \boldsymbol m_n \|_2^2
	+ \| \tilde{\boldsymbol m}_3 - \tilde{\boldsymbol m}_2 \|_2^2
	+ \| \tilde{\boldsymbol m}_4 - \tilde{\boldsymbol m}_3 \|_2^2   \nonumber
	\\
	& \qquad
	+ \frac{\beta}{36} k \| \nabla_h \tilde{\boldsymbol m}_2 \|_2^2
	+ \frac{\beta}{6} k \| \nabla_h \tilde{\boldsymbol m}_3 \|_2^2
	+ \frac{\beta}{12} k \| \nabla_h \tilde{\boldsymbol m}_4 \|_2^2   \le 0  .  \label{stability-6}
\end{align}
By the fact that $\boldsymbol m_{n+1} = \tilde{\boldsymbol m}_4$, we obtain the linear stability estimate for the SSP-IMEX-RK2 scheme:
\begin{align}
	&
	\| \boldsymbol m_{n+1} \|_2^2 - \| \boldsymbol m_n \|_2^2 + \| \tilde{\boldsymbol m}_2 - \boldsymbol m_n \|_2^2
	+ \| \tilde{\boldsymbol m}_3 - \tilde{\boldsymbol m}_2 \|_2^2
	+ \| \boldsymbol m_{n+1} - \tilde{\boldsymbol m}_3 \|_2^2   \nonumber
	\\
	& \qquad
	+ \frac{\beta}{36} k \| \nabla_h \tilde{\boldsymbol m}_2 \|_2^2
	+ \frac{\beta}{6} k \| \nabla_h \tilde{\boldsymbol m}_3 \|_2^2
	+ \frac{\beta}{12} k \| \nabla_h \boldsymbol m_{n+1} \|_2^2   \le 0  , \label{stability-7-1}
\end{align}
which in turn gives the $\ell^\infty (0, T: \ell^2) \cap \ell^2 (0, T; H_h^1)$ bound of the numerical solution, if only is the linear diffusion part is considered:
\begin{equation}
	\| \boldsymbol m_{n+1} \|_2 + \Big( \frac{\beta}{12} k \sum_{j=1}^{n+1} \| \nabla_h \boldsymbol m_j \|_2^2 \Big)^\frac12
	\le \| \boldsymbol m_0 \|_2 . \label{stability-7-2}
\end{equation}

\begin{remark}
	In comparison with the standard three-stage IMEX-RK2 method~\eqref{IMEX-RK2}, the SSP-IMEX-RK2 algorithm~\eqref{SSP-IMEX-RK2} contains four stages, with three intermediate numerical solutions, so that more computations are needed at each time step. Meanwhile, the stability analysis in this section reveals that, this numerical algorithm contains stronger diffusion coefficients than the standard IMEX-RK2 algorithm. In more details, the diffusion part in the standard IMEX-RK2 method~\eqref{IMEX-RK2} essentially correspond to the Crank-Nicolson approximation, which may face a serious theoretical difficulty in the nonlinear analysis, while additional diffusion terms appear in the stability estimate~\eqref{stability-6} for the SSP-IMEX-RK2 algorithm~\eqref{SSP-IMEX-RK2}. This subtle fact will greatly facilitate the convergence analysis in the next subsection. 
\end{remark}

\subsection{Convergence analysis of the SSP-IMEX-RK2 scheme~\eqref{SSP-IMEX-RK2} for a simplified nonlinear LL equation} We consider a simplified nonlinear LL equation~\eqref{eq-3} in this subsection, in which only the damping term is included, while the gyromagnetic term is skipped for simplicity:
\begin{align}
	{{\boldsymbol m}_{t}}= -\alpha \boldsymbol m \times ( \boldsymbol m\times (\epsilon \Delta \boldsymbol m+\emph{ \textbf{f}}) ) .
	\label{equation-LL-mod}
\end{align}

For a vector function $\boldsymbol m$ with $| \boldsymbol m | \equiv 1$, the following identity is recalled
\begin{equation}
	- \boldsymbol m \times ( \boldsymbol m \times \Delta \boldsymbol m ) = \Delta \boldsymbol m + | \nabla \boldsymbol m |^2 \boldsymbol m  .
	\label{LL-reformulation-1}
\end{equation}
In turn, by taking $\beta = \alpha \epsilon$, the nonlinear term $N (\boldsymbol m)$ could be rewritten as
\begin{equation}
	\begin{aligned}
		N (\boldsymbol m) = & - \alpha \boldsymbol m \times ( \boldsymbol m \times ( \epsilon \Delta \boldsymbol m + \boldsymbol f ) ) - \beta \Delta \boldsymbol m
		\\
		= & \beta ( \Delta \boldsymbol m + | \nabla \boldsymbol m |^2 \boldsymbol m ) - \alpha \boldsymbol m \times ( \boldsymbol m \times \boldsymbol f )
		- \beta \Delta \boldsymbol m
		\\
		= &
		\beta | \nabla \boldsymbol m |^2 \boldsymbol m  - \alpha \boldsymbol m \times ( \boldsymbol m \times \boldsymbol f ) .
	\end{aligned}
	\label{LL-reformulation-2}
\end{equation}
Subsequently, the discrete form of the nonlinear term becomes
\begin{equation}
	N_h (\boldsymbol m) = \beta | {\mathcal A}_h \nabla_h \boldsymbol m |^2 \boldsymbol m - \alpha \boldsymbol m \times ( \boldsymbol m \times \boldsymbol f ) ,
	\label{NL-discrete-1}
\end{equation}
in which ${\mathcal A}_h\nabla_h$ (second approximation to the gradient operator) is an average gradient operator defined for the gird function $\boldsymbol m=(u_h, v_h, w_h)^T\in \boldsymbol X$ as $\mathcal{A}_h\nabla_h\boldsymbol m_h=\nabla_h\mathcal{A}_h\boldsymbol m_h$ and $\mathcal{A}_h \boldsymbol m =(\mathcal{A}_xu_h,\mathcal{A}_y v_h,\mathcal{A}_z w_h)$:
\begin{equation*}
	\mathcal{A}_x u_{i,j,\ell}=\frac{u_{i,j,\ell}+u_{i-1,j,\ell}}{2},\,
	\mathcal{A}_y v_{i,j,\ell}=\frac{v_{i,j,\ell}+v_{i,j-1,\ell}}{2},\,
	\mathcal{A}_z w_{i,j,\ell}=\frac{w_{i,j,\ell}+w_{i,j,\ell-1}}{2}.
\end{equation*}

As a result, the SSP-IMEX-RK2 numerical algorithm is formulated as
\begin{equation} \label{SSP-IMEX-RK2-NL}
	\begin{aligned}
		\boldsymbol{\tilde{m}}_{1}= & \boldsymbol{{m}}_{n} , \\
		{\boldsymbol{\tilde{m}}_{2}}= & \boldsymbol{{\tilde{m}}}_{1} +\frac{k}{4} L_h ({\boldsymbol{\tilde{m}}_{2}}) , \\
		{\boldsymbol{\tilde{m}}_{3}}= & \boldsymbol{{\tilde{m}}}_{1} + \frac{k}{2} N_h (\boldsymbol{{\tilde{m}}}_{2} ) +\frac{k}{4} L_h ({\boldsymbol{\tilde{m}}_{3}}) , \\
		{\boldsymbol{\tilde{m}}_{4}}= & \boldsymbol{{\tilde{m}}}_{1} + \frac{k}{2} \left( N_h ( \boldsymbol{{\tilde{m}}}_{2} ) + N_h ( \boldsymbol{{\tilde{m}}}_{3} ) \right) + \frac{k}{3} \left( L_h ({\boldsymbol{\tilde{m}}_{2}}) + L_h ({\boldsymbol{\tilde{m}}_{3}}) + L_h ({\boldsymbol{\tilde{m}}_{4}}) \right) , \\
		\boldsymbol{{m}}_{n+1}= & \boldsymbol{{\tilde{m}}}_{1} + \frac{k}{3} \left( N_h ( \boldsymbol{{\tilde{m}}}_{2} ) + N_h ( \boldsymbol{{\tilde{m}}}_{3} ) + N_h ( \boldsymbol{{\tilde{m}}}_{4} ) \right)
		\\
		&
		+ \frac{k}{3} \left( L_h ({\boldsymbol{\tilde{m}}_{2}}) + L_h ({\boldsymbol{\tilde{m}}_{3}}) + L_h ({\boldsymbol{\tilde{m}}_{4}}) \right) .
	\end{aligned}
\end{equation}
For the convenience of the Runge-Kutta analysis, this numerical system could be equivalently rewritten as
\begin{align}
	&
	\frac{\tilde{\boldsymbol m}_2 - \boldsymbol m_n}{k} = \frac{\beta}{4} \Delta_h \tilde{\boldsymbol m}_2 ,   \label{SSP-IMEX-RK2-NL-2-1}
	\\
	&
	\frac{\tilde{\boldsymbol m}_3 - \tilde{\boldsymbol m}_2}{k} = \frac12 N_h ( \tilde{\boldsymbol m}_2 )
	+ \frac{\beta}{4} \Delta_h ( \tilde{\boldsymbol m}_3 - \tilde{\boldsymbol m}_2 ) ,   \label{SSP-IMEX-RK2-NL-2-2}
	\\
	&
	\frac{\tilde{\boldsymbol m}_4 - \tilde{\boldsymbol m}_3}{k} = \frac12 N_h ( \tilde{\boldsymbol m}_3 )
	+ \beta  \Delta_h ( \frac13 \tilde{\boldsymbol m}_2  + \frac{1}{12} \tilde{\boldsymbol m}_3
	+  \frac13 \tilde{\boldsymbol m}_4 ) ,   \label{SSP-IMEX-RK2-NL-2-3}
	\\
	& \frac{\boldsymbol m_{n+1} - \tilde{\boldsymbol m}_4}{k} =  - \frac16 (  N_h ( \tilde{\boldsymbol m}_2 ) + N_h ( \tilde{\boldsymbol m}_3 ) )
	+  \frac13 N_h ( \tilde{\boldsymbol m}_4 ) . \label{SSP-IMEX-RK2-NL-2-4}
\end{align}

Denote $\Phi$ as the exact solution to the LL equation~\eqref{equation-LL-mod}, with the regularity
\begin{equation}
	\Phi \in {\mathcal R} =  C^3 ([0,T]; [C^0(\bar{\Omega})]^3) \cap C^2([0,T]; [C^2(\bar{\Omega})]^3) \cap L^{\infty}([0,T]; [C^4(\bar{\Omega})]^3)  . \label{regularity-1}
\end{equation}
The main theoretical result is stated in the following theorem.

\begin{thm} \label{thm: convergence}
	Assume that the exact solution $\Phi$ of \eqref{equation-LL-mod} has the regularity ${\mathcal R}$. Denote ${\boldsymbol m}^n$ ($n\ge0$) as the numerical solution obtained from~\eqref{SSP-IMEX-RK2-NL}, or equivalently \eqref{SSP-IMEX-RK2-NL-2-1}-\eqref{SSP-IMEX-RK2-NL-2-4}, with the initial  error satisfying $\|\mathcal{P}_h \Phi (\cdot,t_0) - \boldsymbol m_0 \|_2 +\|\nabla_h ( \mathcal{P}_{h} \Phi (\cdot,t_0) - \boldsymbol m_0 ) \|_2 = \mathcal{O} (h^2)$. In addition, a linear refinement assumption is made for the time step size: $C_1 h \le k \le C_2 h$.  Then the following convergence result holds for $1 \le n\le \left\lfloor\frac{T}{k}\right\rfloor$ as $k, h \to0^+$:	
	\begin{align} \label{convergence-0}
		\| \Phi (\cdot,t_n) - \boldsymbol m^n \|_2
		&\leq \mathcal{C} ( k^2+h^2) ,
	\end{align}	
	in which the constant $\mathcal{C}>0$ is independent of $k$ and $h$.
\end{thm}

\begin{proof}{\bf (Proof of Theorem \ref{thm: convergence}.)}\
Around the boundary section $z=0$, we set $\hat{z}_0 = - \frac12 h$, $\hat{z}_1 = \frac12 h$, and we can extend the profile $\Phi$ to the numerical ``ghost" points, according to the extrapolation formula~\eqref{BC-1}:
\begin{equation}
	\Phi_{i,j,0}= \Phi_{i,j,1} , \quad
	\Phi_{i,j,N+1} = \Phi_{i,j,N} ,  \label{exact-3}
\end{equation}
and the extrapolation for other boundaries can be formulated in the same manner. The proof of such an extrapolation yields a higher order $\mathcal{O}(h^5)$ approximation, instead of the standard $\mathcal{O}(h^3)$ accuracy. See the related derivation in~\cite{cai2023}, as well as the related consistency analysis works~\cite{chen2021convergence, STWW2003, Wang2000, Wang2004}.

Given the exact solution $\Phi$, we denote $\Phi^n = \Phi (\cdot , t^n)$. To facilitate the Runge-Kutta analysis, three more intermediate approximate solutions are constructed at each time step, following the same algorithm as in~\eqref{SSP-IMEX-RK2-NL}:
\begin{align}
	\tilde{\Phi}^{n,(2)} = & \Phi^n + \frac{\beta k}{4} \Delta_h \tilde{\Phi}^{n,(2)} ,   \label{consistency-1-1}
	\\
	\tilde{\Phi}^{n,(3)} = & \tilde{\Phi}^n + \frac{k}{2} N_h ( \tilde{\Phi}^{n,(2)} )
	+ \frac{\beta k}{4} \Delta_h \tilde{\Phi}^{n,(3)}  ,   \label{consistency-1-2}
	\\
	\tilde{\Phi}^{n,(4)} = & \tilde{\Phi}^n + \frac{k}{2} (  N_h ( \tilde{\Phi}^{n,(2)} )
	+ N_h ( \tilde{\Phi}^{n,(3)} ) )  \nonumber
	\\
	&
	+ \frac{\beta k}{3}  \Delta_h ( \tilde{\Phi}^{n,(2)}  + \tilde{\Phi}^{n,(3)}
	+  \tilde{\Phi}^{n,(4)} ) ,   \label{consistency-1-3}
\end{align}
in which the homogeneous discrete Neumann boundary condition (similar to~\eqref{BC-1}, \eqref{exact-3}) is imposed for $\tilde{\Phi}^{n,(j)}$, $j = 2, 3, 4$. Subsequently, a careful Taylor expansion (associated with the SSP-IMEX-RK schemes) reveals the following consistency estimate, for the exact solution at the next time step:
\begin{align}
	\Phi^{n+1} = & \tilde{\Phi}^n + \frac{k}{3} (  N_h ( \tilde{\Phi}^{n,(2)} )
	+ N_h ( \tilde{\Phi}^{n,(3)} ) + N_h ( \tilde{\Phi}^{n,(4)} ) )  \nonumber
	\\
	&
	+ \frac{\beta k}{3}  \Delta_h ( \tilde{\Phi}^{n,(2)}  + \tilde{\Phi}^{n,(3)}
	+  \tilde{\Phi}^{n,(4)} )  + k \tau_0^n ,  \quad  \| \tau_0^n \|_2 \le {\mathcal C} ( k^2 + h^2 ) .    \label{consistency-2}
\end{align}
Of course, with a similar transformation as in~\eqref{SSP-IMEX-RK2-NL-2-1}-\eqref{SSP-IMEX-RK2-NL-2-4}, the exact solution $\Phi^n$, $\Phi^{n+1}$ and the constructed profiles $\tilde{\Phi}^{n,(j)}$ ($j = 2, 3, 4$) satisfy the following numerical system:
\begin{align}
	&
	\frac{\tilde{\Phi}^{n,(2)} - \Phi^n}{k} = \frac{\beta}{4} \Delta_h \tilde{\Phi}^{n,(2)} ,   \label{consistency-3-1}
	\\
	&
	\frac{\tilde{\Phi}^{n,(3)} - \tilde{\Phi}^{n,(2)}}{k} = \frac12 N_h ( \tilde{\Phi}^{n,(2)} )
	+ \frac{\beta}{4} \Delta_h ( \tilde{\Phi}^{n,(3)} - \tilde{\Phi}^{n,(2)} ) ,   \label{consistency-3-2}
	\\
	&
	\frac{\tilde{\Phi}^{n,(4)} - \tilde{\Phi}^{n,(3)}}{k} = \frac12 N_h ( \tilde{\Phi}^{n,(3)} )
	+ \beta  \Delta_h ( \frac13 \tilde{\Phi}^{n,(2)}  + \frac{1}{12} \tilde{\Phi}^{n,(3)}
	+  \frac13 \tilde{\Phi}^{n,(4)} ) ,   \label{consistency-3-3}
	\\
	&
	\frac{\Phi^{n+1} - \tilde{\Phi}^{n,(4)}}{k} = - \frac16  (  N_h ( \tilde{\Phi}^{n,(2)} )
	+ N_h ( \tilde{\Phi}^{n,(3)} ) ) + \frac13 N_h ( \tilde{\Phi}^{n,(4)} )  + \tau_0^n .   \label{consistency-3-4}
\end{align}
It is clear that the constructed profiles $\tilde{\Phi}^{n,(j)}$ ($j = 2, 3, 4$) only depend on the exact solution $\Phi^n$, and the consistency estimate indicates that
\begin{equation}
	\| \tilde{\Phi}^{n,(j)} \|_\infty \le \frac98 , \quad \| \nabla_h \tilde{\Phi}^{n,(j)} \|_\infty \le {\mathcal C}^* ,
	\quad  j = 2, 3, 4 .
	\label{consistency-4}
\end{equation}

The following numerical error functions are defined:
\begin{equation}
	\begin{aligned}
		&
		\boldsymbol e^k = \Phi^k - \boldsymbol m_k ,  \, \, \, k=n, n+1 ,
		\\
		&
		\tilde{\boldsymbol e}^{n,(j)} = \tilde{\Phi}^{n,(j)} - \tilde{\boldsymbol m}_j , \, \, \, j =2 , 3, 4 ,
	\end{aligned}
	\label{error function-1}
\end{equation}
at a point-wise level. In addition, the following nonlinear error terms are introduced:
\begin{equation}
	{\mathcal NLE}^{n,(j)} = N_h ( \tilde{\Phi}^{n,(j)} )  -  N_h ( \tilde{\boldsymbol m}_j ) ,  \quad  j =2 , 3, 4 .
	\label{NL error-def}
\end{equation}
Therefore, subtracting the numerical scheme~\eqref{SSP-IMEX-RK2-NL-2-1}-\eqref{SSP-IMEX-RK2-NL-2-4} from the consistency estimate~\eqref{consistency-3-1}-\eqref{consistency-3-4} yields
\begin{align}
	&
	\frac{\tilde{\boldsymbol e}^{n,(2)} - \boldsymbol e^n}{k} = \frac{\beta}{4} \Delta_h \tilde{\boldsymbol e}^{n,(2)} ,   \label{consistency-5-1}
	\\
	&
	\frac{\tilde{\boldsymbol e}^{n,(3)} - \tilde{\boldsymbol e}^{n,(2)}}{k} = \frac12 {\mathcal NLE}^{n,(2)}
	+ \frac{\beta}{4} \Delta_h ( \tilde{\boldsymbol e}^{n,(3)} - \tilde{\boldsymbol e}^{n,(2)} ) ,   \label{consistency-5-2}
	\\
	&
	\frac{\tilde{\boldsymbol e}^{n,(4)} - \tilde{\boldsymbol e}^{n,(3)}}{k} = \frac12 {\mathcal NLE}^{n,(3)}
	+ \beta  \Delta_h ( \frac13 \tilde{\boldsymbol e}^{n,(2)}  + \frac{1}{12} \tilde{\boldsymbol e}^{n,(3)}
	+  \frac13 \tilde{\boldsymbol e}^{n,(4)} ) ,   \label{consistency-5-3}
	\\
	&
	\frac{\boldsymbol e^{n+1} - \tilde{\boldsymbol e}^{n,(4)}}{k} = - \frac16  (  {\mathcal NLE}^{n,(2)} + {\mathcal NLE}^{n,(3)}  )
	+ \frac13 {\mathcal NLE}^{n,(4)}  + \tau_0^n .   \label{consistency-5-4}
\end{align}
In addition, the discrete homogeneous Neumann boundary condition~\eqref{BC-1} is satisfied for both $\boldsymbol e^{n+1}$, $\boldsymbol e^n$, as well as the intermediate error functions $\tilde{\boldsymbol e}^{n, (j)}$, $j= 2, 3, 4$.

To facilitate the convergence proof, the following functional bound of the nonlinear error terms is needed.
\vskip2mm
\begin{lem}  \label{lem: NL error}
	Under the regularity estimate~\eqref{consistency-4} for the constructed profiles, and the following bound for the numerical solution in the IMEX-RK stages
	\begin{equation}
		\| \tilde{\boldsymbol m}_j \|_\infty  \le \frac54, \quad
		\| \nabla_h \tilde{\boldsymbol m}_j \|_8 \le \tilde{\mathcal C} := {\mathcal C}^* +1 ,
		\quad  j = 2, 3, 4 .
		\label{bound-m-1}
	\end{equation}
	we have an $\| \cdot \|_\frac85$ estimate for the nonlinear error terms:
	\begin{equation}
		\| {\mathcal NLE}^{n, (j)} \|_\frac85  \le \tilde{M} ( \|  \tilde{\boldsymbol e}^{n,(j)}  \|_2
		+ \| \nabla_h  \tilde{\boldsymbol e}^{n,(j)}  \|_2 ) ,   \quad  j = 2, 3, 4 ,  \label{NL error-0}
	\end{equation}
	in which $\tilde{M}$ only depends on $\alpha$, $\beta$, ${\mathcal C}^*$, $\tilde{\mathcal C}$, and the external force term $\boldsymbol f$.
\end{lem}

\begin{proof}{\bf (Proof of Lemma \ref{lem: NL error}.)}\
	A careful expansion of the nonlinear terms $N_h ( \tilde{\Phi}^{n,(j)}) $ and $N_h ( \boldsymbol m_j )$ gives
	\begin{align}
		{\mathcal NLE}^{n, (j)} = & N_h ( \tilde{\Phi}^{n,(j)} ) - N_h ( \tilde{\boldsymbol m}_j )  \nonumber
		\\
		= & \beta | {\mathcal A}_h \nabla_h \tilde{\Phi}^{n, (j)} |^2
		\tilde{\boldsymbol e}^{n, (j)} + \beta \Big(
		{\mathcal A}_h \nabla_h  ( \tilde{\Phi}^{n, (j)} + \tilde{\boldsymbol m}_j )
		\cdot {\mathcal A}_h \nabla_h  \tilde{\boldsymbol e}^{n, (j)}  \Big) \tilde{\boldsymbol m}_j  \nonumber
		\\
		&
		- \alpha \tilde{\boldsymbol m}_j \times ( \tilde{\boldsymbol e}^{n, (j)} \times \boldsymbol f )
		- \alpha \tilde{\boldsymbol e}^{n, (j)} \times ( \tilde{\Phi}^{n, (j)} \times \boldsymbol f )  .   \label{NL error-1}
	\end{align}
	In turn, an application of discrete H\"older inequality leads to
	\begin{align}
		&
		\Big\| \beta | {\mathcal A}_h \nabla_h \tilde{\Phi}^{n, (j)} |^2  \tilde{\boldsymbol e}^{n, (j)} \Big\|_\frac85
		\le \beta  \|   \nabla_h \tilde{\Phi}^{n, (j)} \|_\infty^2 \cdot  \| \tilde{\boldsymbol e}^{n, (j)} \|_\frac85
		\nonumber
		\\
		& \qquad \qquad \qquad \qquad \qquad  \quad
		\le \beta ( {\mathcal C}^* )^2  \| \tilde{\boldsymbol e}^{n, (j)} \|_\frac85
		\le C \beta ( {\mathcal C}^* )^2  \| \tilde{\boldsymbol e}^{n, (j)} \|_2 ,  \label{NL error-2-1}
		\\
		&
		\Big\| \beta \Big(
		{\mathcal A}_h \nabla_h  ( \tilde{\Phi}^{n, (j)} + \tilde{\boldsymbol m}_j )
		\cdot {\mathcal A}_h \nabla_h  \tilde{\boldsymbol e}^{n, (j)}  \Big) \tilde{\boldsymbol m}_j  \Big\|_\frac85    \nonumber
		\\
		\le &
		\beta (  \| \nabla_h  \tilde{\Phi}^{n, (j)} \|_8 + \| \nabla_h \tilde{\boldsymbol m}_j \|_8 )
		\cdot \| \nabla_h  \tilde{\boldsymbol e}^{n, (j)} \|_2 \cdot \| \tilde{\boldsymbol m}_j \|_\infty  \nonumber
		\\
		\le &
		C \beta (  {\mathcal C}^* + \tilde{\mathcal C} )  \| \nabla_h  \tilde{\boldsymbol e}^{n, (j)} \|_2 ,
		\label{NL error-2-2}
		\\
		&
		\Big\|  \alpha \tilde{\boldsymbol m}_j \times ( \tilde{\boldsymbol e}^{n, (j)} \times \boldsymbol f )   \Big\|_\frac85
		\le \alpha  \| \tilde{\boldsymbol m}_j \|_\infty \cdot \| \boldsymbol f \|_\infty \cdot \| \tilde{\boldsymbol e}^{n, (j)} \|_\frac85  \nonumber
		\\
		& \qquad \qquad \qquad
		\le \frac{5 \alpha}{4} C_0 \| \tilde{\boldsymbol e}^{n, (j)} \|_\frac85
		\le C \alpha C_0  \| \tilde{\boldsymbol e}^{n, (j)} \|_2 ,  \label{NL error-2-3}
		\\
		&
		\Big\|  \alpha \tilde{\boldsymbol e}^{n, (j)} \times ( \tilde{\Phi}^{n, (j)} \times \boldsymbol f )  \Big\|_\frac85
		\le \alpha  \| \tilde{\Phi}^{n, (j)} \|_\infty \cdot \| \boldsymbol f \|_\infty \cdot \| \tilde{\boldsymbol e}^{n, (j)} \|_\frac85  \nonumber
		\\
		& \qquad \qquad \qquad
		\le \frac{9 \alpha}{8} C_0 \| \tilde{\boldsymbol e}^{n, (j)} \|_\frac85
		\le C \alpha C_0  \| \tilde{\boldsymbol e}^{n, (j)} \|_2 ,  \label{NL error-2-4}
	\end{align}
	in which the regularity estimate~\eqref{consistency-4} and the functional bound~\eqref{bound-m-1} have been repeatedly applied, along with the fact that $\| g \|_8 \le C \| g \|_\infty$, $\| g \|_\frac85 \le C \| g \|_2$ (for any grid function $g$). Also notice an $\| \cdot \|_\infty$ bound for the external force term: $\| \boldsymbol f \|_\infty \le C_0$. As a result, a substitution of~\eqref{NL error-2-1}-\eqref{NL error-2-4} into \eqref{NL error-1} yields the nonlinear error estimate \eqref{NL error-0}, by taking $\tilde{M} =  C (  \beta ( (  {\mathcal C}^* )^2 +  {\mathcal C}^* + \tilde{\mathcal C} ) + \alpha C_0 )$. This completes the proof of Lemma~\ref{lem: NL error}.
\end{proof}

Before proceeding into the formal error estimate, we make the following a-priori assumption for the numerical error function at the previous time step:
\begin{equation} \label{bound-2}
	\| \boldsymbol e^n \|_2 \le k^{\frac{15}{8}} + h^{\frac{15}{8}} .
\end{equation}
Such an assumption will be recovered by the convergence analysis at the next time step $t^{n+1}$.

\noindent
{\bf Error estimate at Stage 1}  \, \, Taking a discrete inner product with~\eqref{consistency-5-1} by $2 \tilde{\boldsymbol e}^{n, (2)}$ gives
\begin{equation}
	\| \tilde{\boldsymbol e}^{n, (2)} \|_2^2 - \| \boldsymbol e^n \|_2^2 + \| \tilde{\boldsymbol e}^{n, (2)} - \boldsymbol e^n \|_2^2
	+ \frac{\beta}{2} k \| \nabla_h \tilde{\boldsymbol e}^{n, (2)} \|_2^2 = 0 ,  \label{convergence-1-1}
\end{equation}
with an application of the summation-by-parts formula~\eqref{sum1}, due to the discrete homogeneous Neumann boundary condition for $\tilde{\boldsymbol e}^{n, (2)}$. As a result, the following estimates available:
\begin{equation}
	\begin{aligned}
		\| \tilde{\boldsymbol e}^{n, (2)} \|_2 \le &  \| \boldsymbol e^n \|_2 \le  k^{\frac{15}{8}} + h^{\frac{15}{8}}  ,
		\\
		\| \nabla_h \tilde{\boldsymbol e}^{n, (2)} \|_2 \le  & \sqrt{2} \beta^{-\frac12} k^{-\frac12}
		\| \boldsymbol e^n \|_2 \le \sqrt{2} \beta^{-\frac12} \Big( k^{\frac{11}{8}} + \frac{h^{\frac{15}{8}}}{k^\frac12} \Big)
		\\
		\le &
		C  ( k^{\frac{11}{8}} + h^{\frac{11}{8}} ) \le k^{\frac54} + h^{\frac54}  ,
	\end{aligned}
	\label{convergence-1-2}
\end{equation}
in which the linear refinement requirement, $C_1 h \le k \le C_2 h$, has been applied. Subsequently, the $\| \cdot \|_\infty$ and $\| \cdot \|_{W_h^{1,8}}$ bound for the numerical error function $\tilde{\boldsymbol e}^{n, (2)}$ could be derived, with the help of inverse inequalities~\eqref{inverse-1}, \eqref{inverse-2} (by taking $q=8$) in Lemma~\ref{ccclemC1}:
\begin{align}
	&
	\| \tilde{\boldsymbol e}^{n, (2)}  \|_{\infty} \le \gamma {h}^{-1/2 }
	( \| \tilde{\boldsymbol e}^{n, (2)}  \|_2 + \| \nabla_h \tilde{\boldsymbol e}^{n, (2)}  \|_2 )
	\le  \gamma \Big( \frac{k^{\frac54}}{h^\frac12} + h^{\frac34} \Big) \le \frac18 ,  	
	\label{convergence-1-3} 	
	\\
	& 	
	\|  \nabla_h \tilde{\boldsymbol e}^{n, (2)} \|_8 \leq \gamma {h}^{-\frac{9}{8}} \| \nabla_h \tilde{\boldsymbol e}^{n, (2)} \|_2
	\le  \gamma \Big( \frac{k^{\frac54}}{h^\frac98} + h^{\frac18} \Big) \le 1 , 	
	\label{convergence-1-4} 	
\end{align}
provided that $k$ and $h$ are sufficiently small, the linear refinement requirement, $C_1 h \le k \le C_2 h$. In turn, we obtain the following functional bound for the numerical solution $\tilde{\boldsymbol m}_2$ at the first Runge-Kutta stage, which will be useful in the later analysis:
\begin{align}
	&
	\| \tilde{\boldsymbol m}_2 \|_\infty \le \| \tilde{\Phi}^{n, (2)} \|_\infty
	+ \| \tilde{\boldsymbol e}^{n, (2)}  \|_{\infty} \le \frac98 + \frac18 = \frac54 ,  	
	\label{bound-stage 1-1} 	
	\\
	& 	
	\| \nabla_h \tilde{\boldsymbol m}_2 \|_8 \le \| \nabla_h \tilde{\Phi}^{n, (2)} \|_8
	+ \|  \nabla_h \tilde{\boldsymbol e}^{n, (2)} \|_8 \le {\mathcal C}^* + 1 = \tilde{\mathcal C} , 	
	\label{bound-stage 1-2} 	
\end{align}
with an application of triangular inequality, along with the regularity estimate~\eqref{consistency-4}.

\noindent
{\bf Error estimate at Stage 2}  \, \, Taking a discrete inner product with~\eqref{consistency-5-2} by $2 \tilde{\boldsymbol e}^{n, (3)}$ gives
\begin{equation}
	\begin{aligned}
		&
		\| \tilde{\boldsymbol e}^{n, (3)} \|_2^2 - \| \tilde{\boldsymbol e}^{n, (2)} \|_2^2 + \| \tilde{\boldsymbol e}^{n, (3)} - \tilde{\boldsymbol e}^{n, (2)} \|_2^2
		+ \frac{\beta}{2} k \| \nabla_h \tilde{\boldsymbol e}^{n, (3)} \|_2^2
		\\
		& \quad
		= \frac{\beta}{2} k \langle \nabla_h \tilde{\boldsymbol e}^{n, (2)} ,  \nabla_h \tilde{\boldsymbol e}^{n, (3)} \rangle
		+  k \langle {\mathcal NLE}^{n, (2)} ,  \tilde{\boldsymbol e}^{n, (3)} \rangle   ,
	\end{aligned}
	\label{convergence-2-1}
\end{equation}
with an application of the summation-by-parts formula~\eqref{sum1}. The first term on the right hand side could be controlled by the Cauchy inequality:
\begin{equation}
	\langle \nabla_h \tilde{\boldsymbol e}^{n, (2)} ,  \nabla_h \tilde{\boldsymbol e}^{n, (3)} \rangle
	\le  \frac12 ( \| \nabla_h \tilde{\boldsymbol e}^{n, (2)} \|_2^2 + \|  \nabla_h \tilde{\boldsymbol e}^{n, (3)} \|_2^2 ) .
	\label{convergence-2-2}
\end{equation}
Regarding the inner product term associated with the nonlinear error, we observe that an application of Lemma~\ref{lem: NL error} implies the following estimates:
\begin{equation}
	\begin{aligned}
		&
		\| {\mathcal NLE}^{n, (2)} \|_\frac85  \le \tilde{M} ( \|  \tilde{\boldsymbol e}^{n,(2)}  \|_2
		+ \| \nabla_h  \tilde{\boldsymbol e}^{n,(2)}  \|_2 ) ,   \quad  \mbox{so that}
		\\
		&
		\langle {\mathcal NLE}^{n, (2)} ,  \tilde{\boldsymbol e}^{n, (3)} \rangle
		\le   \| {\mathcal NLE}^{n, (2)} \|_\frac85 \cdot \| \tilde{\boldsymbol e}^{n, (3)}  \|_\frac83
		\\
		&
		\le \tilde{M} ( \|  \tilde{\boldsymbol e}^{n,(2)}  \|_2
		+ \| \nabla_h  \tilde{\boldsymbol e}^{n,(2)}  \|_2 )  \cdot \| \tilde{\boldsymbol e}^{n, (3)}  \|_\frac83
		\\
		&
		\le \frac{\tilde{M}}{2} \|  \tilde{\boldsymbol e}^{n,(2)}  \|_2^2
		+ \frac{\beta}{144} \| \nabla_h  \tilde{\boldsymbol e}^{n,(2)}  \|_2^2
		+ (  \frac{\tilde{M}}{2} + \frac{36 \tilde{M}^2}{\beta} ) \| \tilde{\boldsymbol e}^{n, (3)}  \|_\frac83^2 ,
	\end{aligned}
	\label{convergence-2-3}
\end{equation}
based on the regularity estimate~\eqref{consistency-4} and the functional bound~\eqref{bound-stage 1-1}-\eqref{bound-stage 1-2}. Moreover, an application of the discrete Sobolev inequality~\eqref{Sobolev-1} (in Lemma~\ref{lem: Sobolev-1}) indicates that
\begin{align}
	\| \tilde{\boldsymbol e}^{n, (3)}  \|_\frac83 \le C \| \tilde{\boldsymbol e}^{n, (3)} \|_4
	\le  & \mathcal{C}  ( \| \tilde{\boldsymbol e}^{n, (3)} \|_2  	
	+   	\| \tilde{\boldsymbol e}^{n, (3)} \|_2^\frac14 \cdot \| \nabla_h \tilde{\boldsymbol e}^{n, (3)} \|_2^\frac34 ) ,
	\quad \mbox{so that}  \nonumber 
	\\
	(  \frac{\tilde{M}}{2} + \frac{36 \tilde{M}^2}{\beta} ) \| \tilde{\boldsymbol e}^{n, (3)}  \|_\frac83^2
	\le & \mathcal{C}  ( \| \tilde{\boldsymbol e}^{n, (3)} \|_2^2   	
	+   	\| \tilde{\boldsymbol e}^{n, (3)} \|_2^\frac12 \cdot \| \nabla_h \tilde{\boldsymbol e}^{n, (3)} \|_2^\frac32 )
	\nonumber
	\\
	\le &
	\mathcal{C}  \| \tilde{\boldsymbol e}^{n, (3)} \|_2^2   	
	+   \frac{\beta}{24}	\| \nabla_h \tilde{\boldsymbol e}^{n, (3)} \|_2^2 ,   \label{convergence-2-4}   	
\end{align}
in which the Young's inequality has been applied in the last step. Therefore, a substitution of~\eqref{convergence-2-2}-\eqref{convergence-2-4} into \eqref{convergence-2-1} yields
\begin{equation}
	\begin{aligned}
		&
		\| \tilde{\boldsymbol e}^{n, (3)} \|_2^2 - \| \tilde{\boldsymbol e}^{n, (2)} \|_2^2 + \| \tilde{\boldsymbol e}^{n, (3)} - \tilde{\boldsymbol e}^{n, (2)} \|_2^2
		+ \frac{5 \beta}{24} k \| \nabla_h \tilde{\boldsymbol e}^{n, (3)} \|_2^2
		\\
		&
		- \frac{37 \beta}{144} k \| \nabla_h \tilde{\boldsymbol e}^{n, (2)} \|_2^2
		\le \frac{\tilde{M} k}{2} \|  \tilde{\boldsymbol e}^{n,(2)}  \|_2^2
		+  \mathcal{C}  k \| \tilde{\boldsymbol e}^{n, (3)} \|_2^2 .
	\end{aligned}
	\label{convergence-2-5}
\end{equation}
Furthermore, its combination with~\eqref{convergence-1-1} gives
\begin{equation}
	\begin{aligned}
		&
		\| \tilde{\boldsymbol e}^{n, (3)} \|_2^2 - \| \boldsymbol e^n \|_2^2 + \| \tilde{\boldsymbol e}^{n, (2)} - \boldsymbol e^n \|_2^2
		+ \| \tilde{\boldsymbol e}^{n, (3)} - \tilde{\boldsymbol e}^{n, (2)} \|_2^2
		\\
		&
		+ \frac{35 \beta}{144} k \| \nabla_h \tilde{\boldsymbol e}^{n, (2)} \|_2^2
		+ \frac{5 \beta}{24} k \| \nabla_h \tilde{\boldsymbol e}^{n, (3)} \|_2^2
		\le \frac{\tilde{M} k}{2} \|  \tilde{\boldsymbol e}^{n,(2)}  \|_2^2
		+  \mathcal{C}  k \| \tilde{\boldsymbol e}^{n, (3)} \|_2^2 .
	\end{aligned}
	\label{convergence-2-6}
\end{equation}

Consequently, with an application of the a-priori estimate~\eqref{convergence-1-2}, we obtain
\begin{equation}
	\begin{aligned}
		\| \tilde{\boldsymbol e}^{n, (3)} \|_2 \le &  \Big( \frac{1 +  \frac{\tilde{M} k}{2} }{1 - \mathcal{C}  k } \Big)^\frac12
		\| \boldsymbol e^n \|_2 \le 2 ( k^{\frac{15}{8}} + h^{\frac{15}{8}} ) ,
		\\
		\| \nabla_h \tilde{\boldsymbol e}^{n, (3)} \|_2 \le  & \sqrt{4.8} \beta^{-\frac12} k^{-\frac12}
		( 1 +  \frac{\tilde{M} k}{2} )^\frac12 \| \boldsymbol e^n \|_2 \le \sqrt{5} \beta^{-\frac12} \Big( k^{\frac{11}{8}} + \frac{h^{\frac{15}{8}}}{k^\frac12} \Big)
		\\
		\le &
		C  ( k^{\frac{11}{8}} + h^{\frac{11}{8}} ) \le k^{\frac54} + h^{\frac54}  ,
	\end{aligned}
	\label{convergence-2-7}
\end{equation}
under the linear refinement requirement, $C_1 h \le k \le C_2 h$. Similarly, the $\| \cdot \|_\infty$ and $\| \cdot \|_{W_h^{1,8}}$ bound for both the numerical error function $\tilde{\boldsymbol e}^{n, (3)}$ and the numerical solution $\tilde{\boldsymbol m}_3$ could be derived as follows
\begin{align}
	&
	\| \tilde{\boldsymbol e}^{n, (3)}  \|_{\infty} \le \gamma {h}^{-1/2 }
	( \| \tilde{\boldsymbol e}^{n, (3)}  \|_2 + \| \nabla_h \tilde{\boldsymbol e}^{n, (3)}  \|_2 )
	\le  \gamma \Big( \frac{k^{\frac54}}{h^\frac12} + h^{\frac34} \Big) \le \frac18 ,  	
	\label{convergence-2-8} 	
	\\
	& 	
	\|  \nabla_h \tilde{\boldsymbol e}^{n, (3)} \|_8 \le \gamma {h}^{-\frac{9}{8}} \| \nabla_h \tilde{\boldsymbol e}^{n, (3)} \|_2
	\le  \gamma \Big( \frac{k^{\frac54}}{h^\frac98} + h^{\frac18} \Big) \le 1 , 	
	\label{convergence-2-9} 	
	\\
	&
	\| \tilde{\boldsymbol m}_3 \|_\infty \le \| \tilde{\Phi}^{n, (3)} \|_\infty
	+ \| \tilde{\boldsymbol e}^{n, (3)}  \|_{\infty} \le \frac98 + \frac18 = \frac54 ,  	
	\label{bound-stage 2-1} 	
	\\
	& 	
	\| \nabla_h \tilde{\boldsymbol m}_3 \|_8 \le \| \nabla_h \tilde{\Phi}^{n, (3)} \|_8
	+ \|  \nabla_h \tilde{\boldsymbol e}^{n, (3)} \|_8 \le {\mathcal C}^* + 1 = \tilde{\mathcal C} . 	
	\label{bound-stage 2-2} 	
\end{align}

\noindent
{\bf Error estimate at Stage 3}  \, \, Taking a discrete inner product with~\eqref{consistency-5-3} by $2 \tilde{\boldsymbol e}^{n, (4)}$ gives
\begin{equation}
	\begin{aligned}
		&
		\| \tilde{\boldsymbol e}^{n, (4)} \|_2^2 - \| \tilde{\boldsymbol e}^{n, (3)} \|_2^2 + \| \tilde{\boldsymbol e}^{n, (4)} - \tilde{\boldsymbol e}^{n, (3)} \|_2^2
		+ \frac{2 \beta}{3} k \| \nabla_h \tilde{\boldsymbol e}^{n, (4)} \|_2^2
		\\
		&
		= - \frac{2 \beta}{3} k \langle \nabla_h \tilde{\boldsymbol e}^{n, (2)} ,  \nabla_h \tilde{\boldsymbol e}^{n, (4)} \rangle
		- \frac{\beta}{6} k \langle \nabla_h \tilde{\boldsymbol e}^{n, (3)} ,  \nabla_h \tilde{\boldsymbol e}^{n, (4)} \rangle
		+  k \langle {\mathcal NLE}^{n, (3)} ,  \tilde{\boldsymbol e}^{n, (4)} \rangle   ,
	\end{aligned}
	\label{convergence-3-1}
\end{equation}
with an application of the summation-by-parts formula~\eqref{sum1}. The first two terms on the right hand side could be analyzed in the same way as in~\eqref{stability-5-2}-\eqref{stability-5-3}
\begin{align}
	&
	- \frac23 \langle \nabla_h \tilde{\boldsymbol e}^{n, (2)} ,  \nabla_h \tilde{\boldsymbol e}^{n, (4)} \rangle
	\le  \frac29 \| \nabla_h \tilde{\boldsymbol e}^{n, (2)} \|_2^2 + \frac12 \|  \nabla_h \tilde{\boldsymbol e}^{n, (4)} \|_2^2 ,
	\label{convergence-3-2-1}
	\\
	&
	- \frac16 \langle \nabla_h \tilde{\boldsymbol e}^{n, (3)} ,  \nabla_h \tilde{\boldsymbol e}^{n, (4)} \rangle
	\le  \frac{1}{12} \| \nabla_h \tilde{\boldsymbol e}^{n, (3)} \|_2^2
	+ \frac{1}{12} \|  \nabla_h \tilde{\boldsymbol e}^{n, (4)} \|_2^2  ,
	\label{convergence-3-2-2}
\end{align}
Again, the nonlinear error term, as well as the corresponding inner product, could be analyzed in a similar fashion:
\begin{equation}
	\begin{aligned}
		&
		\| {\mathcal NLE}^{n, (3)} \|_\frac85  \le \tilde{M} ( \|  \tilde{\boldsymbol e}^{n,(3)}  \|_2
		+ \| \nabla_h  \tilde{\boldsymbol e}^{n,(3)}  \|_2 ) ,
		\\
		&
		\langle {\mathcal NLE}^{n, (3)} ,  \tilde{\boldsymbol e}^{n, (4)} \rangle
		\le   \| {\mathcal NLE}^{n, (3)} \|_\frac85 \cdot \| \tilde{\boldsymbol e}^{n, (4)}  \|_\frac83
		\\
		&
		\le \tilde{M} ( \|  \tilde{\boldsymbol e}^{n,(3)}  \|_2
		+ \| \nabla_h  \tilde{\boldsymbol e}^{n,(3)}  \|_2 )  \cdot \| \tilde{\boldsymbol e}^{n, (4)}  \|_\frac83
		\\
		&
		\le \frac{\tilde{M}}{2} \|  \tilde{\boldsymbol e}^{n,(3)}  \|_2^2
		+ \frac{\beta}{24} \| \nabla_h  \tilde{\boldsymbol e}^{n,(3)}  \|_2^2
		+ (  \frac{\tilde{M}}{2} + \frac{6 \tilde{M}^2}{\beta} ) \| \tilde{\boldsymbol e}^{n, (4)}  \|_\frac83^2 ,
		\\
		&
		\| \tilde{\boldsymbol e}^{n, (4)}  \|_\frac83 \le C \| \tilde{\boldsymbol e}^{n, (4)} \|_4
		\le  \mathcal{C}  ( \| \tilde{\boldsymbol e}^{n, (4)} \|_2  	
		+   	\| \tilde{\boldsymbol e}^{n, (4)} \|_2^\frac14 \cdot \| \nabla_h \tilde{\boldsymbol e}^{n, (4)} \|_2^\frac34 ) ,
		\\
		&
		(  \frac{\tilde{M}}{2} + \frac{6 \tilde{M}^2}{\beta} ) \| \tilde{\boldsymbol e}^{n, (4)}  \|_\frac83^2
		\le  \mathcal{C}  ( \| \tilde{\boldsymbol e}^{n, (4)} \|_2^2   	
		+   	\| \tilde{\boldsymbol e}^{n, (4)} \|_2^\frac12 \cdot \| \nabla_h \tilde{\boldsymbol e}^{n, (4)} \|_2^\frac32 )
		\\
		&  \qquad \qquad \qquad
		\le  \mathcal{C}  \| \tilde{\boldsymbol e}^{n, (4)} \|_2^2   	
		+   \frac{\beta}{48}	\| \nabla_h \tilde{\boldsymbol e}^{n, (4)} \|_2^2 ,     	
	\end{aligned}
	\label{convergence-3-4} 	
\end{equation}
based on the a-priori bound estimate~\eqref{bound-stage 2-1}-\eqref{bound-stage 2-2} in the second RK stage and the regularity estimate~\eqref{consistency-4}. Subsequently, a substitution of~\eqref{convergence-3-2-1}-\eqref{convergence-3-4} into \eqref{convergence-3-1} gives
\begin{equation}
	\begin{aligned}
		&
		\| \tilde{\boldsymbol e}^{n, (4)} \|_2^2 - \| \tilde{\boldsymbol e}^{n, (3)} \|_2^2 + \| \tilde{\boldsymbol e}^{n, (4)} - \tilde{\boldsymbol e}^{n, (3)} \|_2^2
		+ \frac{\beta}{16} k \| \nabla_h \tilde{\boldsymbol e}^{n, (4)} \|_2^2
		\\
		&
		- \frac{2 \beta}{9} k \| \nabla_h \tilde{\boldsymbol e}^{n, (2)} \|_2^2
		- \frac{\beta}{8} k \| \nabla_h \tilde{\boldsymbol e}^{n, (3)} \|_2^2
		\le \frac{\tilde{M} k}{2} \|  \tilde{\boldsymbol e}^{n,(3)}  \|_2^2
		+  \mathcal{C}  k \| \tilde{\boldsymbol e}^{n, (3)} \|_2^2 ,
	\end{aligned}
	\label{convergence-3-5}
\end{equation}
and its combination with~\eqref{convergence-2-6} yields
\begin{equation}
	\begin{aligned}
		&
		\| \tilde{\boldsymbol e}^{n, (4)} \|_2^2 - \| \boldsymbol e^n \|_2^2 + \| \tilde{\boldsymbol e}^{n, (2)} - \boldsymbol e^n \|_2^2
		+ \| \tilde{\boldsymbol e}^{n, (3)} - \tilde{\boldsymbol e}^{n, (2)} \|_2^2
		+ \| \tilde{\boldsymbol e}^{n, (4)} - \tilde{\boldsymbol e}^{n, (3)} \|_2^2
		\\
		&
		+ \frac{\beta}{48} k \| \nabla_h \tilde{\boldsymbol e}^{n, (2)} \|_2^2
		+ \frac{\beta}{12} k \| \nabla_h \tilde{\boldsymbol e}^{n, (3)} \|_2^2
		+ \frac{\beta}{16} k \| \nabla_h \tilde{\boldsymbol e}^{n, (4)} \|_2^2
		\\
		&
		\le \frac{\tilde{M} k}{2} ( \|  \tilde{\boldsymbol e}^{n,(2)}  \|_2^2  + \|  \tilde{\boldsymbol e}^{n,(3)}  \|_2^2  )
		+  \mathcal{C}  k ( \| \tilde{\boldsymbol e}^{n, (3)} \|_2^2 + \| \tilde{\boldsymbol e}^{n, (4)} \|_2^2 ) .
	\end{aligned}
	\label{convergence-3-6}
\end{equation} 	
Similarly, with the help of the a-priori estimates~\eqref{convergence-1-2}, \eqref{convergence-2-7}, in the first and second RK stages, respectively, the following rough error estimates could be derived:
\begin{equation}
	\begin{aligned}
		\| \tilde{\boldsymbol e}^{n, (4)} \|_2 \le &  \Big( \frac{1 +  {\mathcal C} k}{1 - \mathcal{C}  k } \Big)^\frac12
		\| \boldsymbol e^n \|_2 \le 2 ( k^{\frac{15}{8}} + h^{\frac{15}{8}} ) ,
		\\
		\| \nabla_h \tilde{\boldsymbol e}^{n, (4)} \|_2 \le  & 4 \beta^{-\frac12} k^{-\frac12}
		( 1 +  {\mathcal C} k )^\frac12 \| \boldsymbol e^n \|_2
		\le 5 \beta^{-\frac12} \Big( k^{\frac{11}{8}} + \frac{h^{\frac{15}{8}}}{k^\frac12} \Big)
		\\
		\le &
		C  ( k^{\frac{11}{8}} + h^{\frac{11}{8}} ) \le k^{\frac54} + h^{\frac54}  ,
	\end{aligned}
	\label{convergence-3-7}
\end{equation}
and the $\| \cdot \|_\infty$ and $\| \cdot \|_{W_h^{1,8}}$ bound for both the numerical error function $\tilde{\boldsymbol e}^{n, (4)}$ and the numerical solution $\tilde{\boldsymbol m}_4$ also becomes available:
\begin{align}
	&
	\| \tilde{\boldsymbol e}^{n, (4)}  \|_{\infty} \le \gamma {h}^{-1/2 }
	( \| \tilde{\boldsymbol e}^{n, (4)}  \|_2 + \| \nabla_h \tilde{\boldsymbol e}^{n, (4)}  \|_2 )
	\le  \gamma \Big( \frac{k^{\frac54}}{h^\frac12} + h^{\frac34} \Big) \le \frac18 ,  	
	\label{convergence-3-8} 	
	\\
	& 	
	\|  \nabla_h \tilde{\boldsymbol e}^{n, (4)} \|_8 \le \gamma {h}^{-\frac{9}{8}} \| \nabla_h \tilde{\boldsymbol e}^{n, (4)} \|_2
	\le  \gamma \Big( \frac{k^{\frac54}}{h^\frac98} + h^{\frac18} \Big) \le 1 , 	
	\label{convergence-3-9} 	
	\\
	&
	\| \tilde{\boldsymbol m}_4 \|_\infty \le \| \tilde{\Phi}^{n, (4)} \|_\infty
	+ \| \tilde{\boldsymbol e}^{n, (4)}  \|_{\infty} \le \frac98 + \frac18 = \frac54 ,  	
	\label{bound-stage 3-1} 	
	\\
	& 	
	\| \nabla_h \tilde{\boldsymbol m}_4 \|_8 \le \| \nabla_h \tilde{\Phi}^{n, (4)} \|_8
	+ \|  \nabla_h \tilde{\boldsymbol e}^{n, (4)} \|_8 \le {\mathcal C}^* + 1 = \tilde{\mathcal C} . 	
	\label{bound-stage 3-2} 	
\end{align}

\noindent
{\bf Error estimate at Stage 4}  \, \, Taking a discrete inner product with~\eqref{consistency-5-4} by $2 \\e^{n+1}$ gives
\begin{equation}
	\begin{aligned}
		&
		\| \boldsymbol e^{n+1} \|_2^2 - \| \tilde{\boldsymbol e}^{n, (4)} \|_2^2 + \| \boldsymbol e^{n+1} - \tilde{\boldsymbol e}^{n, (4)} \|_2^2
		- k \langle \tau_0^n ,  \boldsymbol e^{n+1} \rangle
		\\
		&
		=  - \frac13 k \langle {\mathcal NLE}^{n, (2)} ,  \boldsymbol e^{n+1} \rangle
		- \frac13 k \langle {\mathcal NLE}^{n, (3)} ,  \boldsymbol e^{n+1} \rangle
		+ \frac23 k \langle {\mathcal NLE}^{n, (4)} ,  \boldsymbol e^{n+1} \rangle  .
	\end{aligned}
	\label{convergence-4-1}
\end{equation}
The local truncation error inner product term could be controlled in a straightforward way:
\begin{equation}
	\langle \tau_0^n ,  \boldsymbol e^{n+1} \rangle \le \frac12 (  \| \tau_0^n \|_2^2 + \|  \boldsymbol e^{n+1}  \|_2^2 ) .
	\label{convergence-4-2}
\end{equation}
The nonlinear error terms could be analyzed in a similar manner:
\begin{equation}
	\begin{aligned}
		&
		\| {\mathcal NLE}^{n, (4)} \|_\frac85  \le \tilde{M} ( \|  \tilde{\boldsymbol e}^{n,(4)}  \|_2
		+ \| \nabla_h  \tilde{\boldsymbol e}^{n,(4)}  \|_2 ) ,
		\\
		&
		- \frac13 \langle {\mathcal NLE}^{n, (2)} ,  \boldsymbol e^{n+1} \rangle
		\le   \frac13 \| {\mathcal NLE}^{n, (2)} \|_\frac85 \cdot \| \boldsymbol e^{n+1}  \|_\frac83
		\\
		&
		\le \frac13 \tilde{M} ( \|  \tilde{\boldsymbol e}^{n, (2)}  \|_2
		+ \| \nabla_h  \tilde{\boldsymbol e}^{n, (2)}  \|_2 )  \cdot \| \boldsymbol e^{n+1}  \|_\frac83
		\\
		&
		\le \frac{\tilde{M}}{6} \|  \tilde{\boldsymbol e}^{n,(2)}  \|_2^2
		+ \frac{\beta}{144} \| \nabla_h  \tilde{\boldsymbol e}^{n,(2)}  \|_2^2
		+ (  \frac{\tilde{M}}{6} + \frac{4 \tilde{M}^2}{\beta} ) \| \boldsymbol e^{n+1}  \|_\frac83^2 ,
		\\
		&
		- \frac13 \langle {\mathcal NLE}^{n, (3)} ,  \boldsymbol e^{n+1} \rangle
		\le   \frac13 \| {\mathcal NLE}^{n, (3)} \|_\frac85 \cdot \| \boldsymbol e^{n+1}  \|_\frac83
		\\
		&
		\le \frac13 \tilde{M} ( \|  \tilde{\boldsymbol e}^{n, (3)}  \|_2
		+ \| \nabla_h  \tilde{\boldsymbol e}^{n ,(3)}  \|_2 )  \cdot \| \boldsymbol e^{n+1}  \|_\frac83
		\\
		&
		\le \frac{\tilde{M}}{6} \|  \tilde{\boldsymbol e}^{n, (3)}  \|_2^2
		+ \frac{\beta}{36} \| \nabla_h  \tilde{\boldsymbol e}^{n, (3)}  \|_2^2
		+ (  \frac{\tilde{M}}{6} + \frac{\tilde{M}^2}{\beta} ) \| \boldsymbol e^{n+1}  \|_\frac83^2 ,
		\\
		&
		\frac23 \langle {\mathcal NLE}^{n, (4)} ,  \boldsymbol e^{n+1} \rangle
		\le   \frac23 \| {\mathcal NLE}^{n, (4)} \|_\frac85 \cdot \| \boldsymbol e^{n+1}  \|_\frac83
		\\
		&
		\le \frac23 \tilde{M} ( \|  \tilde{\boldsymbol e}^{n, (4)}  \|_2
		+ \| \nabla_h  \tilde{\boldsymbol e}^{n, (4)}  \|_2 )  \cdot \| \boldsymbol e^{n+1}  \|_\frac83
		\\
		&
		\le \frac{\tilde{M}}{3} \|  \tilde{\boldsymbol e}^{n, (4)}  \|_2^2
		+ \frac{\beta}{48} \| \nabla_h  \tilde{\boldsymbol e}^{n,(4)}  \|_2^2
		+ (  \frac{\tilde{M}}{3} + \frac{16 \tilde{M}^2}{3 \beta} ) \| \boldsymbol e^{n+1}  \|_\frac83^2 .
	\end{aligned}
	\label{convergence-4-4} 	
\end{equation}
In turn, a substitution of~\eqref{convergence-4-2}, \eqref{convergence-4-4} into \eqref{convergence-4-1} leads to
\begin{equation}
	\begin{aligned}
		&
		\| \boldsymbol e^{n+1} \|_2^2 - \| \tilde{\boldsymbol e}^{n, (4)} \|_2^2 + \| \boldsymbol e^{n+1} - \tilde{\boldsymbol e}^{n, (4)} \|_2^2
		- \frac{\beta}{48} k \| \nabla_h \tilde{\boldsymbol e}^{n, (4)} \|_2^2
		\\
		&
		- \frac{\beta}{144} k \| \nabla_h \tilde{\boldsymbol e}^{n, (2)} \|_2^2
		- \frac{\beta}{36} k \| \nabla_h \tilde{\boldsymbol e}^{n, (3)} \|_2^2
		- \frac12 k (  \| \tau_0^n \|_2^2 + \|  \boldsymbol e^{n+1}  \|_2^2 )
		\\
		&
		\le \frac{\tilde{M} k}{6} ( \|  \tilde{\boldsymbol e}^{n, (2)}  \|_2^2  + \|  \tilde{\boldsymbol e}^{n, (3)}  \|_2^2
		+ 2 \|  \tilde{\boldsymbol e}^{n, (4)}  \|_2^2 )
		+ (  \frac{\tilde{2 M}}{3} + \frac{31 \tilde{M}^2}{3 \beta} ) \| \boldsymbol e^{n+1}  \|_\frac83^2  ,
	\end{aligned}
	\label{convergence-4-5}
\end{equation}
and its combination with~\eqref{convergence-3-6} yields
\begin{equation}
	\begin{aligned}
		&
		\| \boldsymbol e^{n+1} \|_2^2 - \| \boldsymbol e^n \|_2^2 + \| \tilde{\boldsymbol e}^{n, (2)} - \boldsymbol e^n \|_2^2
		+ \| \tilde{\boldsymbol e}^{n, (3)} - \tilde{\boldsymbol e}^{n, (2)} \|_2^2
		+ \| \tilde{\boldsymbol e}^{n, (4)} - \tilde{\boldsymbol e}^{n, (3)} \|_2^2
		\\
		&
		+ \| \boldsymbol e^{n+1} - \tilde{\boldsymbol e}^{n, (4)} \|_2^2
		+ \frac{\beta}{72} k \| \nabla_h \tilde{\boldsymbol e}^{n, (2)} \|_2^2
		+ \frac{\beta}{18} k \| \nabla_h \tilde{\boldsymbol e}^{n, (3)} \|_2^2
		+ \frac{\beta}{24} k \| \nabla_h \tilde{\boldsymbol e}^{n, (4)} \|_2^2
		\\
		&
		\le
		\mathcal{C}  k ( \|  \tilde{\boldsymbol e}^{n,(2)}  \|_2^2 + \| \tilde{\boldsymbol e}^{n, (3)} \|_2^2
		+ \| \tilde{\boldsymbol e}^{n, (4)} \|_2^2 ) + \hat{M} k \| \boldsymbol e^{n+1}  \|_\frac83^2
		+ \frac12 k (  \| \tau_0^n \|_2^2 + \|  \boldsymbol e^{n+1}  \|_2^2 ).
	\end{aligned}
	\label{convergence-4-6}
\end{equation}
with $\hat{M} = 	\frac{\tilde{2 M}}{3} + \frac{31 \tilde{M}^2}{3 \beta}$. To control the terms $\hat{M} k \| \boldsymbol e^{n+1}  \|_\frac83^2$, we see that an application of triangle inequality implies that
\begin{equation}
	\begin{aligned}
		&
		\| \boldsymbol e^{n+1}  \|_\frac83  \le  \| \tilde{\boldsymbol e}^{n, (4)} \|_\frac83 + \| \boldsymbol e^{n+1} - \tilde{\boldsymbol e}^{n, (4)} \|_\frac83 ,
		\quad \mbox{so that}
		\\
		&
		\hat{M} \| \boldsymbol e^{n+1}  \|_\frac83^2  \le  2 \hat{M} ( \| \tilde{\boldsymbol e}^{n, (4)} \|_\frac83^2
		+ \| \boldsymbol e^{n+1} - \tilde{\boldsymbol e}^{n, (4)} \|_\frac83^2 ) .
	\end{aligned}
	\label{convergence-4-7}
\end{equation}
Meanwhile, an application of inverse inequality \eqref{inverse-2} (by taking $q = \frac83$, in Lemma~\ref{ccclemC1}) results in
\begin{equation}
	\begin{aligned}
		& 	
		\| \boldsymbol e^{n+1} - \tilde{\boldsymbol e}^{n, (4)}  \|_\frac83
		\le \gamma {h}^{-\frac38 } \| \boldsymbol e^{n+1} - \tilde{\boldsymbol e}^{n, (4)}  \|_2 ,   \quad \mbox{so that}
		\\
		&
		2 \hat{M} k \| \boldsymbol e^{n+1} - \tilde{\boldsymbol e}^{n, (4)}  \|_\frac83^2
		\le 2 \hat{M} \gamma^2 k {h}^{-\frac34 } \| \boldsymbol e^{n+1} - \tilde{\boldsymbol e}^{n, (4)}  \|_2^2
		\le  \frac12 \| \boldsymbol e^{n+1} - \tilde{\boldsymbol e}^{n, (4)}  \|_2^2 ,  	
	\end{aligned}
	\label{convergence-4-8}	
\end{equation}
provided that $2 \hat{M} \gamma^2 k {h}^{-\frac34 } \le \frac12$, which is always valid under the linear refinement requirement, $C_1 h \le k \le C_2 h$, and the assumption that $k$ and $h$ are sufficiently small. In addition, we recall the preliminary estimate~\eqref{convergence-3-4} for $\| \tilde{\boldsymbol e}^{n, (4)}  \|_\frac83$:
\begin{equation}
	\begin{aligned}
		&
		\| \tilde{\boldsymbol e}^{n, (4)}  \|_\frac83 \le C \| \tilde{\boldsymbol e}^{n, (4)} \|_4
		\le  \mathcal{C}  ( \| \tilde{\boldsymbol e}^{n, (4)} \|_2  	
		+   	\| \tilde{\boldsymbol e}^{n, (4)} \|_2^\frac14 \cdot \| \nabla_h \tilde{\boldsymbol e}^{n, (4)} \|_2^\frac34 ) ,  \quad
		\mbox{so that}
		\\
		&
		2 \hat{M}  \| \tilde{\boldsymbol e}^{n, (4)}  \|_\frac83^2
		\le  \mathcal{C}  ( \| \tilde{\boldsymbol e}^{n, (4)} \|_2^2   	
		+   	\| \tilde{\boldsymbol e}^{n, (4)} \|_2^\frac12 \cdot \| \nabla_h \tilde{\boldsymbol e}^{n, (4)} \|_2^\frac32 )
		\\
		&  \qquad \qquad \qquad
		\le  \mathcal{C}  \| \tilde{\boldsymbol e}^{n, (4)} \|_2^2   	
		+   \frac{\beta}{48}	\| \nabla_h \tilde{\boldsymbol e}^{n, (4)} \|_2^2 .      	
	\end{aligned}
	\label{convergence-4-9} 	
\end{equation}
Therefore, a substitution of~\eqref{convergence-4-7}-\eqref{convergence-4-9} into \eqref{convergence-4-6} gives
\begin{equation}
	\begin{aligned}
		&
		\| \boldsymbol e^{n+1} \|_2^2 - \| \boldsymbol e^n \|_2^2 + \| \tilde{\boldsymbol e}^{n, (2)} - \boldsymbol e^n \|_2^2
		+ \| \tilde{\boldsymbol e}^{n, (3)} - \tilde{\boldsymbol e}^{n, (2)} \|_2^2
		+ \| \tilde{\boldsymbol e}^{n, (4)} - \tilde{\boldsymbol e}^{n, (3)} \|_2^2
		\\
		&
		+ \frac12 \| \boldsymbol e^{n+1} - \tilde{\boldsymbol e}^{n, (4)} \|_2^2
		+ \frac{\beta}{72} k \| \nabla_h \tilde{\boldsymbol e}^{n, (2)} \|_2^2
		+ \frac{\beta}{18} k \| \nabla_h \tilde{\boldsymbol e}^{n, (3)} \|_2^2
		+ \frac{\beta}{48} k \| \nabla_h \tilde{\boldsymbol e}^{n, (4)} \|_2^2
		\\
		&
		\le
		\mathcal{C}  k ( \|  \tilde{\boldsymbol e}^{n,(2)}  \|_2^2 + \| \tilde{\boldsymbol e}^{n, (3)} \|_2^2
		+ \| \tilde{\boldsymbol e}^{n, (4)} \|_2^2 )
		+ \frac12 k (  \| \tau_0^n \|_2^2 + \|  \boldsymbol e^{n+1}  \|_2^2 ) .
	\end{aligned}
	\label{convergence-4-10}
\end{equation}
Moreover, by making use of the triangle inequalities
\begin{equation}
	\begin{aligned}
		&
		\|  \tilde{\boldsymbol e}^{n,(2)}  \|_2 \le \| \boldsymbol e^n \|_2 + \| \tilde{\boldsymbol e}^{n, (2)} - \boldsymbol e^n \|_2  ,
		\\
		&
		\|  \tilde{\boldsymbol e}^{n,(3)}  \|_2 \le \| \boldsymbol e^n \|_2 + \| \tilde{\boldsymbol e}^{n, (2)} - \boldsymbol e^n \|_2
		+ \| \tilde{\boldsymbol e}^{n, (3)} - \tilde{\boldsymbol e}^{n, (2)} \|_2 ,
		\\
		&
		\|  \tilde{\boldsymbol e}^{n,(4)}  \|_2 \le \| \boldsymbol e^n \|_2 + \| \tilde{\boldsymbol e}^{n, (2)} - \boldsymbol e^n \|_2
		+ \| \tilde{\boldsymbol e}^{n, (3)} - \tilde{\boldsymbol e}^{n, (2)} \|_2
		+ \| \tilde{\boldsymbol e}^{n, (4)} - \tilde{\boldsymbol e}^{n, (3)} \|_2 ,
	\end{aligned}
	\label{convergence-4-11}
\end{equation}
we arrive at
\begin{equation}
	\begin{aligned}
		&
		\| \boldsymbol e^{n+1} \|_2^2 - \| \boldsymbol e^n \|_2^2
		+ \frac{\beta}{72} k \| \nabla_h \tilde{\boldsymbol e}^{n, (2)} \|_2^2
		+ \frac{\beta}{18} k \| \nabla_h \tilde{\boldsymbol e}^{n, (3)} \|_2^2
		+ \frac{\beta}{48} k \| \nabla_h \tilde{\boldsymbol e}^{n, (4)} \|_2^2
		\\
		&
		\le
		\mathcal{C}  k \|  \boldsymbol e^n  \|_2^2
		+ \frac12 k (  \| \tau_0^n \|_2^2 + \|  \boldsymbol e^{n+1}  \|_2^2 ) ,
	\end{aligned}
	\label{convergence-4-12}
\end{equation}
provided that $k$ is sufficiently small. In turn, an application of discrete Gronwall inequality \cite{Girault1986} yields the desired convergence estimate at the next time step
\begin{equation}
	\| \boldsymbol e^{n+1} \|_2  \le \mathcal{C} (k^2+h^2) ,  \label{convergence-4-13} 	
\end{equation}
based on the fact that $\| \tau_0^n \|^2 \le {\mathcal C} (k^2 + h^2)$. As a result, we see that the a-priori assumption~\eqref{bound-2} has also been validated at the next time step $t^{n+2}$, provided that $k$ and $h$ are sufficiently small. By mathematical induction, this completes the proof of Theorem~\ref{thm: convergence}.
\end{proof}
\vskip2mm
\begin{remark}
	For the multi-step IMEX numerical schemes to various nonlinear PDEs, there have been some existing works of convergence analysis~\cite{cheng16b, gottlieb12b}, etc. Meanwhile, for the IMEX-RK numerical schemes, the theoretical works have been very limited, due to the theoretical challenges associated with the multi-stage nature, lack of sufficient numerical diffusion, etc. A linearized stability analysis is provided for the IMEX-RK1 and IMEX-RK2 schemes for the diffusion problem~\cite{wang2020local}, as well as the error estimate for the constant-coefficient diffusion equation. The convergence analysis has been reported for the IMEX-RK numerical methods to various nonlinear PDEs, such as the convection-diffusion equation~\cite{WangH2016}, the porous media equation~\cite{Nan2021, WangH2019}, etc. On the other hand, the degree of nonlinearity of the LL equation (even the simplified version \eqref{equation-LL-mod}, with only the damping term) is higher than these reported models. As a result, the associated theoretical analysis presented in this article contains more techniques than the existing works.
\end{remark}
\vskip2mm
\begin{remark}
	At the intermediate Runge-Kutta stages, the $\ell^2$ error bound estimate~\eqref{convergence-1-2}, \eqref{convergence-2-7}, \eqref{convergence-3-7}, the associated $\ell^\infty$ and $W_h^{1,8}$ error bounds~\eqref{convergence-1-3}-\eqref{convergence-1-4}, \eqref{convergence-2-8}-\eqref{convergence-2-9}, \eqref{convergence-3-8}-\eqref{convergence-3-9}, stand for a rough error estimate. These error bound estimates do not preserve a full accuracy order; instead, the motivation of these analyses is to obtain a rough bound of the numerical error function, so that a preliminary bound becomes available for the numerical solution at the associated Runge-Kutta stages, as derived in~\eqref{bound-stage 1-1}-\eqref{bound-stage 1-2},  \eqref{bound-stage 2-1}-\eqref{bound-stage 2-2}, \eqref{bound-stage 3-1}-\eqref{bound-stage 3-2}. As a result of these preliminary bounds, the nonlinear error terms could be analyzed with the help of Lemma~\ref{lem: NL error}, so that the nonlinear analysis would go through.
	
	In comparison, with these nonlinear analyses established, the numerical error inequalities~\eqref{convergence-1-1}, \eqref{convergence-2-6}, \eqref{convergence-3-6} stand for a refined error estimate. Finally, the derived estimate~\eqref{convergence-4-12} becomes an inequality in which we are able to derive the desired accuracy order for the numerical error.
	
	A combination of rough error estimate and refined error estimate has been reported for various nonlinear PDEs, such as ternary Flory-Huggins-Cahn-Hilliard system~\cite{Dong2022a}, Poisson-Nernst-Planck system~\cite{LiuC2021a}, the porous medium equation by an energetic variational approach~\cite{duan22b, duan20a}, the reaction-diffusion system with detailed balance~\cite{LiuC2022b}, etc. This article reports the application of such a technique to the LL equation for the first time.
\end{remark}

\vskip2mm
\begin{remark}
	For simplicity, we only consider the damping term in the convergence analysis, as the PDE system formulated as~\eqref{equation-LL-mod}. For the full LL equation~\eqref{eq-3}, the convergence estimate may still go through, under a large damping parameter assumption, combined with a technical requirement of $k = {\mathcal O} (h^2)$; the details are left to interested readers, for the sake of brevity. Of course, such a requirement only stands for a theoretical difficulty, and this restrictive time step constraint is not needed in the practical computation, as demonstrated in extensive numerical experiments.
\end{remark}

\vskip2mm
\begin{remark}
	A theoretical analysis of the IMEX-RK3 numerical algorithm, including the linearized stability estimate and optimal rate convergence analysis, is expected to be even more challenging, due to the complicated diffusion coefficient stencil in the RK stages. This theoretical work will be studied in a forthcoming paper.
\end{remark}




\section{Conclusions} \label{section:conclusion}

In this paper, we propose implicit-explicit Runge-Kutta numerical methods for solving the Landau-Lifshitz equation. By introducing an artificial damping term, IMEX-RK methods only solve a few linear systems  with constant coefficients and SPD structures, regardless of the damping parameter in a magnetic material. Accuracy, efficiency, and the insensitive dependence on the artificial damping parameter of IMEX-RK2 and IMEX-RK3 have been verified in both the 1-D and 3-D computations. Micromagnetics simulations using the full Landau-Lifshitz equation are conducted, including three stable structures and the first benchmark problem from NIST. Reasonable results are generated by the IMEX-RK methods, in qualitative and quantitative agreements with available results. In addition, the linearized stability estimate and optimal rate convergence analysis are presented for an alternate IMEX-RK2 numerical scheme, the SSP-IMEX-RK2 algorithm. This analysis has provided a theoretical evidence of the robust performance of the IMEX-RK schemes.

\section*{Acknowledgments}
This work is supported in part by the grants NSFC 11971021 (J.~Chen) and NSF DMS-2012269 (C.~Wang).





          \end{document}